\documentclass[leqno]{amsart}

\usepackage{amsthm}
\usepackage{amssymb}
\usepackage{mathrsfs}
\usepackage{graphicx}
\usepackage{todonotes}
\usepackage{xcolor}

{%
\setcounter{enumi}{0}

\begin{enumerate}}%
{\end{enumerate} }
\usepackage[inner=3cm,outer=3cm,bottom=4cm,top=4cm]{geometry}
\makeatletter

\renewcommand*{\descriptionlabel}[1]{%
  \let\orglabel\label
  \let\label\@gobble
  \phantomsection
  \edef\@currentlabel{#1\unskip}%
  \let\label\orglabel
  \hspace\labelsep \upshape\bfseries #1%
}
\makeatother
{%
\setcounter{enumi}{0}

\begin{enumerate}}%
{\end{enumerate} }

\newcommand{\R}{\mathbb{R}}

\newcommand{\LL}{\mathcal{L}}

\newcommand{\mL}{\mathcal{L}}
\newcommand{\bfj}{\textbf{j}}
\DeclareMathOperator{\Div}{div}

\def\N{\mathbb{N}}

\usepackage[colorlinks]{hyperref}
\usepackage{enumitem}
\newtheorem{thm}{Theorem}[section]
\newtheorem{lem}[thm]{Lemma}

\theoremstyle{definition}
\newtheorem{definition}[thm]{Definition}
\newtheorem{example}[thm]{Example}
\newtheorem{rem}[thm]{Remark}

\numberwithin{equation}{section}

\definecolor{purple-red}{rgb}{1.0, 0.27, 0.0}

\newcommand{\cb}%{\normalcolor}
{\color{blue}}
\allowdisplaybreaks
\DeclareRobustCommand{\SkipTocEntry}[5]{}

\title[Schauder %estimates 
regularity for fractional HJ equations]{
A Schauder regularity theory for %subcritical 
nonlocal and mixed local-nonlocal viscous Hamilton--Jacobi equations} %with low-regularity data}
\author{Espen R. Jakobsen}
\address{Norwegian University of Science and Technology, H\o{}gskoleringen 1, 7034 Trondheim, Norway}
\email{espen.jakobsen@ntnu.no}
\author{Robin Ø. Lien}
\address{Norwegian University of Science and Technology, H\o{}gskoleringen 1, 7034 Trondheim, Norway}
\email{robin.o.lien@ntnu.no}
\author{Artur Rutkowski}
\address{Wroc\l aw University of Science and Technology,
Wyb. Wyspia\'nskiego 27, 50-370 Wroc\l aw, Poland}
\email{artur.rutkowski@pwr.edu.pl}
\begin{document}
\begin{abstract}
    We prove space-time Schauder estimates -- optimal regularity estimates in Hölder spaces -- and well-posedness results for mild and classical solutions of %fractional
    viscous Hamilton--Jacobi
    %--Bellman 
    equations with subcritical nonlocal and mixed local-nonlocal diffusions in $\R^d$.
    %and low-regularity %initial
    %data. and Hamiltonians. The nonlinear equationPDE is
    Our 
    %results 
    spatial Schauder estimates hold under mild assumptions on the nonlocal/mixed operators and Hamiltonians.
    %cover %a large class of 
    %fractional and mixed diffusions whose generators include 
    The Laplacian, fractional Laplacians, nonsymmetric, spectrally one-sided, %non-absolutely continuous 
     and strongly anisotropic integral operators, as well as sums of such operators are covered.
    We observe an interplay between the regularity of the initial data and the growth of the Hamiltonian in the gradient, and 
    %our results cover 
     %give  spatial Schauder  results for the
     develop a spatial Schauder theory for two canonical cases: (i) Lipschitz initial data and general Hamiltonians that are Hölder in space and merely locally Lipschitz in the gradient, and (ii) Hölder initial data and Hamiltonians that are Hölder in space and locally Lipschitz with power growth in the gradient. 
    %the case of Lipschitz initial data with general locally Lipschitz Hamiltonians, and the case of H\"older initial data with power-type Hamiltonians. 
    We compute explicit blow-up rates for $C^1$ and higher order H\"older norms as $t\to 0$. The results include short %time 
    and long time existence of mild solutions, optimal regularity in Hölder spaces and corresponding Schauder a priori estimates, and that spatially smooth mild solutions are regular in time and pointwise classical solutions. Under further assumptions on the diffusion operator, we then prove time and space-time Schauder regularity estimates in optimal H\"older spaces which respect the natural fractional parabolic scaling. These results generalize classical linear local and fractional Schauder estimates to our non-linear fractional, possibly anisotropic and nonsymmetric setting.
    %and possibly local-nonlocal setting.
\end{abstract}
\maketitle
\tableofcontents

\section{Introduction}
 We investigate well-posedness and Schauder regularity estimates for the initial value problem for the nonlocal viscous Hamilton--Jacobi equation:
\begin{align} \tag{vHJ} \label{eq:hjb}
\left\{
    \begin{aligned}
         \partial_t u -\mathcal{L}u-H(t,x,Du)&=0, \qquad &(t,x)&\in (0,T]\times \R^d, \\ 
        u(0,x)&=u_0(x), \qquad &x &\in \R^d,
    \end{aligned}
\right.
\end{align}
where $\mL$ is the generator of a L\'evy process defined by
\begin{align}\label{eq:L}
    \mathcal{L}\varphi(x) &= \Div(AD\varphi(x)) + \int_{\R^d}\big[\varphi(x+z)-\varphi(x)-\mathbf{1}_{|z|<1}D\varphi(x)\cdot z\big]\,  d\mu(z),\quad x\in \R^d,
\end{align}
$A$ is a nonnegative definite matrix, and $\mu\geq 0$ is a Borel measure satisfying $\int_{\R^d}(1 \wedge |z|^2) \ d\mu(z) < \infty$, i.e. $\mu$ is a \textit{L\'evy measure}. In other words, $\mL$ is a \textit{L\'evy operator} with  the \textit{L\'evy triplet} $(A,0,\mu)$. %\begin{rem} 
%In general, 
General L\'evy operators %can also
have a constant drift term $B Du$ 
 %-- then the 
 and L\'evy triplet %would be 
 $(A,B,\mu)$.  Assuming $B=0$ does not affect the generality, since $B Du$ can be absorbed into the Hamiltonian $H$.

We assume that the nonpositive operator $\mL$ is
%subcritical, i.e. 
of order $\alpha\in(0,2]$, a condition we express in the language of the heat kernel of $\mL$ in assumption \ref{NDa} below. This condition is satisfied by a large class of local, nonlocal, and mixed local-nonlocal L\'evy operators, including fractional Laplacians and strongly anisotropic and nonsymmetric operators, see Example~\ref{ex:NDa} below. 
In most of the paper we assume that $\mL$ is subcritical, i.e. $\alpha\in (1,2]$, but we also give results for the critical case $\alpha=1$. 
\medskip

\noindent \textbf{Main results.} %The main results of the paper are the s
We prove short-time existence and optimal % spatial  
Schauder regularity %(with blow-up rates  as $t\to0$) 
of solutions of \eqref{eq:hjb}.
We also show that solutions are classical  %and we give their long-time existence
and give long-time existence results. 

\medskip
\noindent \textit{Short-time existence (Theorems \ref{thm::hjb_sol_existence} and \ref{thm::case_b_hjb_sol_existence}).} The proof of the short-time existence of mild solutions of \eqref{eq:hjb} in is based on a Duhamel formulation and a fixed point argument in a subset of $C((0,T]; C^1(\R^d))$. %When proving short-time existence
There is an interplay between the regularity of $u_0$ and the growth of $H=H(t,x,p)$ in $p$ that results in 
%requires care. The result is therefore for
two different cases:
    
    \smallskip
    (I) Lipschitz initial data and locally Lipschitz $H$ in $p$. 
    
    \smallskip
    (II) H\"older initial data and $H$ with explicit power-type growth in $p$.

\smallskip
\noindent Case (I) is rather standard as the solutions have uniformly bounded gradient. Case (II) allows for solutions with gradient blow-up as $t\to 0$, and to control it, the growth of the Hamiltonian needs to be compensated for by regularity of the initial condition. We refer to \ref{assump:B1} for the range of the admissible H\"older exponents $\delta$ of $u_0$ with respect to the growth rate $r$ of $H$.
%while case (A) does not. We 
%also note that we do not assume 
We note that the short-time existence does not require any smoothness or
integrability of $H$ in $x$ or $t$, just continuity and boundedness.
%(only continuity and boundedness in $t$ and $x$).

\medskip
\noindent \textit{Optimal Schauder regularity with 
 blow-up rates as $t\to0$ (Theorems \ref{thm::full_schauder_regularity}, \ref{thm:full_schauder_regularity_case_b}, and \ref{thm:spacetimeSchauder}).}  %Under further assumptions on the $x$-regularity of $H$, we are able to achieve optimal regularity of the short-time solutions through a
We show that under condition \ref{NDa}, the solutions to \eqref{eq:hjb} gain spatial H\"older regularity of order $\alpha$ over the regularity of $H$ in $x$. To this end we apply the
``diagonal splitting'' of the space-time cylinder 
%inspired
previously used e.g. by Chaudru de Raynal, Menozzi, and Priola \cite{deraynal2019schauder} 
%(who study 
in the context of linear equations with drift and supercritical diffusion, i.e. $\alpha\in(0,1)$. We note that \ref{NDa} is slightly weaker than the assumption used in \cite{deraynal2019schauder}, which involves a certain moment condition on the heat kernel needed to handle the unbounded drift coefficient. In the case of unbounded gradients, the optimal H\"older norm blow-up rate is quite cumbersome to obtain -- because of two singularities coming from the gradient and the heat kernel we need to apply a ``doubly fractional'' version of Gr\"onwall's inequality, see Lemma~\ref{lem:generalized_gronwall_2}, and in some cases we also use a bootstrap argument. We also prove Schauder estimates in time in Theorem~\ref{thm:timeSchauder}. To  transfer the space regularity to time regularity we have to impose an upper bound on the order of $\mL$, condition \ref{assump:L2'}.  %and a mild scaling condition \ref{assump:L3} on the heat kernel. 
Finally, having estimates in both space and time, we give joint space-time estimates in Theorem~\ref{thm:spacetimeSchauder}. Our methods allow for several extensions, for example 
to critical diffusions such as $(-\Delta)^{1/2}$ under a smallness assumption on the data, to diffusions with modulated jumps,  and to Hamiltonians depending also on $u$ and nonlocal terms $Qu$ of order less than or equal to 1. We refer to Section~\ref{sec::remarks_and_extensions} for further discussion.

\medskip
\noindent \textit{Classical solutions and long-time existence (Theorem \ref{thm:u_is_classical_sol} and \ref{thm:long_time_existence}).}
%While the procedure is quite standard and well-known,
%For the sake of completeness w
Finally, we show that solutions are classical and that we have long-time existence under further assumptions on $H$. 
%Classicalness requires
 If we assume Hölder regularity of $H$ in $x$, then the solutions are classical, which follows from
%and uses 
the regularity results mentioned above.
%, together with results from semigroup theory pertaining to the generator of stochastic processes, to show continuous differentiability in time.
Subsequently, using %standard 
results from viscosity solution theory, we achieve long-time existence in the case of $x$-Lipschitz $H$ with explicit power-type growth conditions. More precisely, we 
get well-posedness and global Lipschitz bounds for viscosity solutions on the (long) time interval $[0,T]$, and conclude by uniqueness and previous results %for mild and classical solutions, 
that this solution coincides locally with smooth mild/classical solutions.
%on a small time interval around any time $t\in(0,T)$.
%get global Lipschitz bounds for  solutions,
% whence a bootstrap argument yields a unique classical solution with optimal regularity of \eqref{eq:hjb} on any finite time-horizon by piecing together unique short-time classical solutions on overlapping time intervals. 
\smallskip

\noindent\textbf{Background.}
Viscous HJ equations arise as dynamical programming equations in optimal stochastic control problems and differential games. The unknown function $u$ then represents the (upper or lower) value function for the optimally controlled process \cite{Ca:Book,MR2179357,YZ:Book}. Our setting corresponds to the case of controlled drift in the presence of noise given by a pure-jump L\'evy process $X_t$ with generator $\mL$. For more details on this connection, see e.g. \cite{MR3931325,Ha:Book}. For the classical case of Brownian noise, corresponding to \eqref{eq:hjb} with $\mL$ replaced by $\Delta$, we refer to \cite[Chapter~IV.3]{MR2179357}, see also \cite{Ba:Book}.
%Equation \eqref{eq:hjb} also arises in the study of large deviations of stochastic processes.
We also mention that the equation is closely connected with the theory of large deviations of stochastic processes \cite{MR2260560,Ba:Book,BC13}, and the KPZ equation from physics \cite{KPZ86}, a stochastically forced version of the equation and an important example of an ill-posed problem that can be solved in Hairer's theory of regularity structures \cite{MR3274562}.

Studies of viscous HJ equations were initiated almost half a century ago, with the pioneering works of Kru\v{z}kov \cite{MR0404870} and Crandall and Lions \cite{MR690039}. The topic has been very active,
%. To give 
 and a sample of various aspects of the research concerning HJ equations
 %, we refer to 
 can be found in \cite{MR4000845,MR3285244,MR2732926,MR3575590,MR4264951,MR4333510,MR3095206,MR2664465,MR4044675}, but this list is far from complete. A strong incentive for further investigation of HJ equations came from the theory of mean field games \cite{huang2006large,MR2295621}, which is a major motivation behind our work. Mean field games have also been studied for nonlocal operators \cite{MR3934106,MR4309434,MR4223351,fracnonsep}, but the setting of classical solutions usually requires strong assumptions on data. Our work %should 
 will facilitate analysis of mean field game systems in a lower regularity setting.

Let us present the literature more closely related to our methods and settings. We use general viscosity solution theory for nonlocal and mixed local-nonlocal operators as presented in Jakobsen and Karlsen \cite{MR2129093}. There are many other references on this topic including \cite{MR2243708,MR2422079,MR4244533}.
%developed a viscosity framework for rather general nonlocal and mixed local-nonlocal operators. 
For $\mL = -(-\Delta)^{\alpha/2}$ Imbert \cite{MR2121115} established a comparison principle and Lipschitz estimates in the context of viscosity solutions, and in the case $H = H(Du)$, showed that mild solutions are smooth classical solutions. Improvements and extensions to general nondegenerate fractional operators $\mL$ can be found in \cite{MR4309434,ERJAR}, and elements of a first (suboptimal) Schauder theory is given in \cite{ERJAR} -- see also \cite{JR24}. 
%The heat kernel assumption \ref{NDa} was used in \cite{MR4309434} to get a regularity gain of 1 derivative. A rather straightforward upgrade to $\alpha-\epsilon$ was given in \cite{ERJAR}. By \textit{full Schauder regularity} we mean a gain of $\alpha$ derivatives, which we achieve in the present work.
The setting of mild solutions and Duhamel formula that we use here was also present in e.g. \cite{MR2471928,MR2259335,MR2373320}. Schauder estimates for linear parabolic equations involving quite general L\'evy-type operators were given by Mikulevi\v{c}ius and Pragarauskas \cite{MR1246036}. More recently, Dong, Jin and Zhang \cite{MR3803717} gave Schauder estimates for fully nonlinear equations involving nonlocal operators with the L\'evy measure comparable to the one of the fractional Laplacian. Our present Schauder estimates improve %those given 
the regularity estimates in \cite{MR2121115,MR4309434,ERJAR} in several ways: We obtain more (and maximal) regularity under the same assumptions, we give new results for Hamiltonians $H$ of low regularity  (measured in Hölder scales) and with much more general growth in the gradient, and we give new, more precise and explicit, a priori estimates for high order Hölder norms.
%that account for the optimal blowup as $t\to0$.

There are several recent works on fractional HJ equations with low-regularity data, admitting solutions with unbounded gradient. Mild solutions for the HJ equation on $\R^d$ with $H=H(Du)=|Du|^r$ with critical diffusion, i.e. $\alpha=1$, were studied by Iwabuchi and Kawakami \cite{MR3623641}, for initial conditions in Besov spaces. Goffi \cite{MR4350574} considered forcing terms in $L^p$ and initial conditions which are continuous or are in Besov spaces on the torus, with power-type growth conditions on $H$. For the (deterministic) fractional KPZ equation with a forcing term on bounded domains with homogeneous Dirichlet conditions, Abdellaoui, Peral, Primo, and Soria \cite{MR4358140} gave the existence and nonexistence results under certain integrability assumptions on the initial condition and the forcing term, in relation to the power in the gradient term. Matioc and Walker \cite{matioc2023wellposednessquasilinearparabolicequations} established a general semigroup framework for quasilinear equations in time-weighted spaces, tracing the blow-up of certain norms, but it does not allow for the scale of H\"older spaces. Here we also mention the earlier paper by Benachour and Lauren\c{c}ot \cite{MR1720778}, which addressed the case of local diffusion.

The paper is organized as follows. We begin by covering the relevant notation. In Section \ref{sec::operator_L_and_heat_kernel} we list the assumptions on $\mathcal{L}$, give some examples, and establish regularizing effects of the heat kernel of $\mathcal{L}$. Sections \ref{sec::short_time_existence}, \ref{sec::full_schauder_regularity} and \ref{sec::long_time_existence} compose the main parts of the paper, covering the short-time existence, optimal regularity, and solutions being classical and long-time existence respectively. Finally, we discuss extensions and give closing remarks in Section \ref{sec::remarks_and_extensions}. Some technical results used throughout the paper are collected in the Appendix \ref{sec::app}. 
\medskip

\noindent \textbf{Notation.}
Below, $\mathbb{N} = \{1,2,\ldots\}$ and $d\in \mathbb{N}$. As usual, for $r\in \R$, the function $\lceil r\rceil$ gives the smallest integer greater than or equal to $r$ and $\lfloor r\rfloor$ gives the largest integer less than or equal to $r$.

Let $\kappa=(k_1,\dots, k_d)$, $k_i\in\N\cup\{0\}$, be a multi-index of order $|\kappa|=k_1 + \dots + k_d$  and define the spatial derivative 
%\begin{align*}
    $\partial^\kappa:= \frac{\partial^{|\kappa|}}{\partial x_1^{k_1}\partial x_2^{k_2} \dots \partial x_d^{k_d}}$. We will sometimes use the subscript to indicate that an operator acts with respect to the variable in the subscript. By $D^m \varphi$, $m\in \{0,1,2,\ldots\}$, we mean the tensor of all derivatives of order $m$.
For $n\in\N\cup\{0\}$, let $C_b^n(\R^d)$ be the space of % times differentiable 
functions with $n$ continuous and bounded derivatives and norm
$\|\varphi \|_{C_b^n}=\sum_{j=0}^n \max_{|\kappa|=j}\|\partial^\kappa \varphi\|_\infty,$
where $\|\cdot \|_\infty$ is the $L^\infty$ (or $C_b$) norm. For $\gamma=n+\beta$, where $\beta\in(0,1)$, we define the H\"older space
$$C_b^\gamma(\R^d) :=\{\varphi\in C_b^n(\R^d) \ : \ \|\varphi \|_{C_b^\gamma}<\infty \},$$
where
$$\|\varphi \|_{C_b^\gamma} := \|\varphi \|_{C_b^n}+\max_{|\kappa|=n}[\partial^\kappa\varphi]_\beta \qquad \text{and} \qquad [\varphi]_{\beta}:= \sup_{\substack{x, h\in \R^d \\ h \neq 0}}\frac{
|\varphi(x+h)-\varphi(x)|}{|h|^{\beta}}.$$
Note that $[\cdot]_\beta$ defines a seminorm for all  $\beta \in (0,1]$ and that $[\varphi]_1$ is the Lipschitz constant of $\varphi$. We also let $[\varphi]_0 = \|\varphi\|_\infty$. %For $n\geq 1$, we define $$[\varphi]_\gamma := \max_{|\kappa|=n}[D^\kappa\varphi]_\beta.$$

Norms, seminorms and operators act on the unspecified arguments of the function they are applied on, e.g. for a function $\varphi\colon A\times B \rightarrow \R$ we will write $\|\varphi\|_\infty := \sup_{(a,b)\in A\times B } |\varphi(a,b)|$, while $\|\varphi(\cdot,b)\|_\infty :=\sup_{a\in A} |\varphi(a,b)|$, and so on. 

\section{Optimal regularizing effect of the fractional heat kernel}\label{sec::operator_L_and_heat_kernel}
In this section we introduce the heat kernel and the heat semigroup associated with the nonlocal diffusion operator $\mL$ defined in \eqref{eq:L}, and prove optimal parabolic regularizing effect of the semigroup under general assumptions. Our results are equivalent to proving optimal regularity results for the fractional heat equation $\partial_tv -\mL v = f$.
 Taking the Fourier transform of \eqref{eq:L}, we find that 
%\begin{align*}
    $\mathcal{F}(\mathcal{L}\varphi)(\xi)=\widehat{\mathcal{L}}(\xi) \widehat{\varphi}(\xi)$, 
%\end{align*}
where the symbol %$\widehat{\mathcal{L}}$ of the operator $\mathcal{L}$ is given by
\begin{align*}
    \widehat{\mathcal{L}}(\xi)=  \xi A \xi + \int_{\R^d}(1-e^{i \xi \cdot z}+i\xi \cdot z \mathbf{1}_{|z|<1})\ d\mu(z),\qquad \xi\in \R^d.
\end{align*}
%That is, $\widehat{\mathcal{L}}$ is the symbol of the operator $\mathcal{L}$. 
 The heat kernel $p_t$ of 
$\mathcal{L}$ is then given by
\begin{align*}
    \mathcal{F}p_t(\xi)=  (e^{-t\widehat{\mathcal{L}}(\cdot)})(\xi),\qquad \xi\in \R^d.
\end{align*} 
By the Lévy--Khintchine theorem \cite[Theorem 1.2.14]{MR2072890} $p_t$ are probability measures for $t>0$, and %, being the inverse Fourier transform of characteristic functions of Lévy processes. 
 they converge weakly to the Dirac delta as $t\rightarrow 0$  \cite[Proposition 1.4.4]{MR2072890}. We \textit{assume} that the heat kernel is absolutely continuous with respect to the Lebesgue measure -- this is contained in the assumption \ref{NDa} below. With the convention that $p_t(x)\, dx :=p_t(dx)$, we then have $(\partial_t  -\mathcal{L})p_t =0$ in the classical sense. The semigroup associated with the heat kernel is
$$P_t\varphi(x) := p_t*\varphi(x)=\int_{\R^d} p_t(x-y)\varphi(y) dy,\qquad t\geq 0,\ x\in \R^d.$$
For $\varphi \in C^2_c(\R^d)$, the operator $\mL$ coincides with the generator of $P_t$ \cite[Theorem~31.5]{MR3185174}. 
%Throughout the paper we will write $P_t\varphi(x)=P_t[\varphi](x)$ interchangeably.

We will assume that the nonpositive operator $\mL$ is a non-degenerate operator of order $\alpha$, a condition we formulate in terms of its heat kernel.\bigskip
\makeatletter
\newcommand{\myitem}[1]{%
\item[#1]\protected@edef\@currentlabel{#1}%
}
\makeatother
\begin{enumerate}
%\medskip
% \myitem{$\mathbf{(L0)}$}\label{assump:L1} $
%     \mathcal{L}\varphi(x) = \int_{\R^d}\big[\varphi(x+z)-\varphi(x)-\mathbf{1}_{|z|<1}D\varphi(x)\cdot z\big]d\mu(z)$ for $\phi\in C^2_b(\R^d)$ and $x\in \R^d$,\\[0.2cm]
% where $\mu\geq 0$ is a Borel measure satisfying $\int_{\R^d}(1 \wedge |z|^2) \ d\mu(z) < \infty$.
% \bigskip\myitem{$\mathbf{(L2)}$}
\myitem{$\mathbf{(L1)}$}
%$\mathbf{(L1 \:\! | \:\!  \textnormal{$m$})}$}
\label{NDa} 
The heat kernel $p_t\in C_b^\infty(\R^d)$ and there exist $\alpha\in(1,2]$ and $c_0>0$ such that for $t\in (0,T]$,
        \begin{align}\label{eq:NDainequality}
        \int_{\R^d} |Dp_t(y)| dy \leq c_0t^{-\frac{1}{\alpha}}.
      \end{align}
\end{enumerate}
For the linear estimates of this section we can relax this assumption:
\bigskip
\begin{enumerate}
%\medskip
% \myitem{$\mathbf{(L0)}$}\label{assump:L1} $
%     \mathcal{L}\varphi(x) = \int_{\R^d}\big[\varphi(x+z)-\varphi(x)-\mathbf{1}_{|z|<1}D\varphi(x)\cdot z\big]d\mu(z)$ for $\phi\in C^2_b(\R^d)$ and $x\in \R^d$,\\[0.2cm]
% where $\mu\geq 0$ is a Borel measure satisfying $\int_{\R^d}(1 \wedge |z|^2) \ d\mu(z) < \infty$.
% \bigskip\myitem{$\mathbf{(L2)}$}
\myitem{$\mathbf{(L1')}$}
%$\mathbf{(L1 \:\! | \:\!  \textnormal{$m$})}$}
\label{NDa'} 
%\eqref{eq:NDainequality} holds with $\alpha\in(0,2]$.
The heat kernel $p_t\in C_b^\infty(\R^d)$ and satisfies \eqref{eq:NDainequality} for some $\alpha\in(0,2]$ and every $t\in (0,T]$.
%         \begin{align*}
%         \int_{\R^d} |Dp_t(y)| dy \leq c_0t^{-\frac{1}{\alpha}}.
%       \end{align*}
\end{enumerate}

\smallskip
\begin{rem}\label{rem:NDam}
Let $m\in\N$. Since $D^mp_t=(Dp_{\frac tm})^{*m}$, by Young's inequality for convolutions and \ref{NDa'},
       \begin{align}\label{NDam}
        \int_{\R^d} |D^mp_t(y)| dy \leq ( m^{\frac1\alpha} c_0)^mt ^{-\frac{m}{\alpha}}.
      \end{align}
We will mostly use $m=2$ (for short-time existence in Section \ref{sec:shorttime}) and $m=\lceil \alpha+\beta\rceil$ where $\beta$ is the H\"older continuity of the data given by \ref{assump:H} or \ref{assump:B3} (for Schauder regularity in Section \ref{sec::full_schauder_regularity}).  
\end{rem}
\begin{example}
\label{ex:NDa}
Operators $\mathcal{L}$ in \eqref{eq:L} satisfying \ref{NDa}, cf. \cite[Section~4]{MR4309434}. 
\begin{enumerate}
\item Theorem~4.3
%\footnote{The proof of this result is based on the heat kernel estimates in \cite[Theorem~5.2]{MR4308627}.} 
in \cite{MR4309434} (see also \cite[Theorem~5.2]{MR4308627}): If $A=0$  and there is $\alpha\in (1,2)$ such that $\frac{d\mu}{dz}\approx |z|^{-d-\alpha}$ for $|z|\leq 1$,
%where $\mu$ it the L\'evy measure in \eqref{eq:L}, 
then  $\mathcal{L}$ satisfies \ref{NDa}. Here $\mL$ is purely nonlocal with the L\'evy measure $\mu$ absolutely continuous in $|z|\leq1$, but {\em no assumption} on the tails ($|z|>1$).
An example is the fractional Laplacian $-(-\Delta)^{\alpha/2}$ with L\'evy triplet $(0,0,c_{d,\alpha}|z|^{-d-\alpha}\, dz)$. \medskip
    \item $\mL  = -(-\partial^2_{x_1})^{\alpha_1/2}-(-\partial^2_{x_2})^{\alpha_2/2} -\ldots -(-\partial^2_{x_d})^{\alpha_d/2}$ for $\alpha_i\in (1,2)$ satisfies \ref{NDa} with $\alpha = \min_i \alpha_i$. Here the corresponding L\'evy measure $\mu$ is not absolutely continuous.
    \smallskip
    \item The Riesz--Feller operator on $\R$ corresponding to triplet $(0,0,|z|^{-1-\alpha}\textbf{1}_{(0,\infty)}(z)\, dz)$ satisfies \ref{NDa}, see \cite[Lemma~2.1 (G7) and Proposition~2.3]{MR3360395}. This corresponds to a spectrally one-sided process and a very non-symmetric $\mL$.\smallskip
    \item The generator of the CGMY process in Finance satisfies \ref{NDa}, see \cite[Example~4.4]{MR4309434}. This is an example of a tempered and non-symmetric process and $\mL$.\smallskip
    \item Operators $\mL$ with strictly positive definite matrix $A$ satisfy \ref{NDa} with $\alpha=2$. Here the principal part of $\mL$ is the non-degenerate second derivative term.\smallskip
        \item If $\mL$ satisfies \ref{NDa} and $\tilde \mL$ is any L\'evy operator \eqref{eq:L}, then $\mL + \tilde\mL$ satisfies \ref{NDa}. One example %of $\tilde\mL$ 
        is the degenerate second order operator $\tilde \mL \phi = \Div(AD\phi)$ where $A\geq0$ is not invertible. Here there will be a gain of regularity of $\alpha$ derivatives in all directions, but as we will see in Section \ref{subseq:directions}, we gain 2 full derivatives in non-degenerate directions of $A$. 
        
        % with the same $\alpha$.
    \medskip
\end{enumerate}
\end{example}

We now show that the heat semigroup $P_t$ has optimal smoothing properties in H\"older scales, in the sense that 
%when applied to functions and integrals in time. We show estimates that give a 
there is a gain of $\alpha$ derivatives on the input data. We explicitly quantify the rates of blow-up for our H\"older estimates as $t\to 0$.

\begin{thm} \label{lem::lem1}
    Assume \ref{NDa'}, $\beta \in [0,1]$, and 
    %.
    %with $m=\lceil \alpha + \beta \rceil$. If 
    $\varphi \in C_b^{\beta}(\R^d)$. Then
    %and $\gamma\in(0,1)$, then 
    for $k\in\N$, % \geq 1$, 
    $\gamma\in(0,1)$, and  $t\in(0,T]$,
    %\smallskip
    \begin{align*}
     &(i)\ \ 
        \big\| D^k P_t \varphi(\cdot) \big\|_\infty \leq c_{k,\beta,d}[\varphi]_{\beta}t^{-\frac{k-\beta}{\alpha}}, \hspace{9.5cm}\\[0.2cm] 
        %\qquad \text{and} \qquad \ %&& %k=1,\dots,m+\lfloor\beta\rfloor
        %k\in\N,\hspace{6cm}\\[0.2cm]
        % &(ii) \quad [D^kP_t\varphi(\cdot)]_{\alpha+\beta-k}\leq2^{k+1-(\alpha+\beta)}c_{\beta,d}[\varphi]_\beta t^{-1}, \qquad && k=m-1, 
        & (ii) \ \  [D^kP_t\varphi(\cdot)]_{\gamma}\leq2^{1-\gamma}c_{k+1%\gamma
        ,\beta,d}[\varphi]_\beta t^{-\frac{k+\gamma
        -\beta}{\alpha}},
       %\qquad && %k=m-1+\lfloor \beta \rfloor,
       % k\in\N,\quad \gamma\in(0,1),
    \end{align*}
    where $c_{k,\beta,d}= 
    ( k^{\frac1\alpha} c_0)^{k-\beta}$
    %c_0^{ k-\beta}$ 
     for $\beta\in\{0,1\}$, $c_{k,\beta,d}= 
    ( k^{\frac1\alpha} c_0)^{k-\beta}%c_0^{ k-\beta}
    C_{\beta,d}$ for $\beta\in(0,1)$, and $C_{\beta,d}$ is given by Theorem~\ref{thm:holder_interpolation}.
\end{thm}

Note that when $\gamma=\alpha+\beta- k$ and $k< \alpha+\beta$, then %(m-1)$ 
the estimate in {\em(ii)} is reminiscent of standard semigroup results of the type $\|\mL P_t \varphi\|\leq Ct^{-1}\|\varphi\| $.
\begin{proof}
{\em (i)} \ Consider first $\beta \in(0,1]$. Let $m\in\N\cup\{0\}$ and note that by \ref{NDa'} (see Remark \ref{rem:NDam}) (or $\|p_t\|_{L^1}=1$ when $m=0$), %$k=0$), 
$x,x' \in \R^d$,
\begin{align*}
    |D^{m}P_t\varphi(x) - D^mP_t\varphi(x')| &\leq \int_{\R^d}|D^mp_t(y)||\varphi(x-y) - \varphi(x'-y)|\, dy\\
    &%\label{estim_Dk}
    \leq |x-x'|^\beta [\varphi]_{\beta} \|D^mp_t\|_{L^1} \leq |x-x'|^\beta [\varphi]_{\beta}  ((m\, \vee \, 1)^{\frac1\alpha} c_0)^{m}  t^{-\frac{m}{\alpha}},
    %& \leq |x-x'|^\beta [\varphi]_{\beta}  t^{-\frac{k}{\alpha}} \cdot \begin{cases}
     %    1, &\qquad k=0, \\
     %     (k^{\frac1\alpha} c_0)^{k} , &\qquad k\geq 1,\\         
    %\end{cases}
\end{align*}
that is, 
%\begin{align*}
$[D^{m}P_t\varphi]_{\beta} \leq  ((m\, \vee \, 1)^{\frac1\alpha} c_0)^{m}[\varphi]_{\beta}t^{-\frac{m}{\alpha}}$.
Since $D^mp_t\in L^1(\R^d)$, we could interchange derivatives and integrals using standard arguments for convolutions. 
Recall that $k\geq 1$. When $\beta =1$, this yields
%\footnote{We could take $(k-1)^{\frac 1\alpha}$ instead of $k^{\frac 1\alpha}$ in the following estimate. Optimality is sacrificed for a cleaner estimate.}
\begin{align*}
    \|D^k P_t \varphi\|_\infty= [D^{k-1}P_t\varphi]_{1}\leq ( %(k-1)
    k^{\frac1\alpha} c_0)^{k-1} [\varphi]_1 t^{-\frac{k-1}{\alpha}},
    %\qquad k\geq 1, 
\end{align*}
while for $\beta \in(0,1)$, by interpolation (Lemma \ref{thm:holder_interpolation}), 
% and \eqref{estim_Dk},  
we get
\begin{align*}
    \|D^kP_t\varphi\|_{\infty} &\leq C_{\beta, d}[D^{k-1}P_t\varphi]_{\beta}^\beta  [D^kP_t\varphi]_{\beta}^{1-\beta}
    %\\
   % &\leq C_{\beta, d} ( (k-1)^{\frac1\alpha} c_0)^{(k-1)\beta} ( k^{\frac1\alpha} c_0)^{k(1-\beta)}
   %  %c_0^{ k-\beta}
   %  [\varphi]_{\beta}t^{-\frac{k-\beta}{\alpha}}\\
   %  &\leq C_{\beta, d} ( k^{\frac1\alpha} c_0)^{(k-1)\beta} ( k^{\frac1\alpha} c_0)^{k(1-\beta)}
   %  %c_0^{ k-\beta}
   %  [\varphi]_{\beta}t^{-\frac{k-\beta}{\alpha}}\\
   %  &=
   \leq C_{\beta, d} 
    ( k^{\frac1\alpha} c_0)^{k-\beta}
    %c_0^{ k-\beta}
    [\varphi]_{\beta}t^{-\frac{k-\beta}{\alpha}}
    ,% \quad k\geq 1.
\end{align*}
When $\beta=0$, we directly get that 
\begin{align*}
    |D^kP_t\varphi(x)| &\leq \int_{\R^d}|D^kp_t(y)||\varphi(x-y)|\, dy %\leq \|\varphi\|_{\infty} \|D^kp_t\|_{L^1}
    \leq \|\varphi\|_{\infty} ( k^{\frac1\alpha} c_0)^{k} t^{-\frac{k}{\alpha}}.% \qquad k\geq 1.
\end{align*}

\noindent{\em (ii)} \ Interpolation (Theorem \ref{thm:holder_interpolation} with $\eta=1$) and {\em (i)} yield the estimate.
\end{proof}

Next we look at the regularizing effect of the heat semigroup on a space-time convolution term coming from the Duhamel formula for heat equations with right-hand sides $f$.
% Recall 
% the beta-function, $B(s_1, s_2):=\int_0^1 x^{s_1-1}(1-x)^{s_2-1} \, dx$, for $s_1,s_2>0$. 

\begin{thm}
\label{prop:f_bound}
    Assume \ref{NDa'}, $\beta \in [0,1]$, $\alpha+\beta \in (1,2) \cup (2,3)$, %, $\alpha+\beta\notin \{ 2,3\}$, 
    %\ref{NDa} with $m=\lceil \alpha +\beta \rceil$, 
    $\gamma \in [0,1)$,  
    and for $t \in (0,T]$, $f(t,\cdot)\in C_b^{\beta}(\R^d)$, $\|f(t,\cdot)\|_{C^\beta_b} = O(t^{-\gamma})$ as $t\to 0$, and %. Let 
    $$w(t,x):=\int_0^t P_{t-s}[f(s)](x)\, ds.$$ 
    Then for $t\in(0,T]$, if
    %$k\in\{1,\lfloor\alpha+\beta\rfloor \}$ 
     $k=1$ or $k=\lfloor\alpha+\beta\rfloor$, 
    \smallskip 
    \begin{align*}
         &(i)\quad  \|D^kw(t,\cdot) \|_\infty \leq 
         c^{(k)}_{\gamma,\alpha,\beta} \sup_{s\in(0,T]} \|s^\gamma f(s,\cdot) \|_{C_b^\beta}t^{-\gamma}t^{\frac{\alpha+\beta-k}{\alpha}},
         %\qquad && k= 1, \dots,  m-1, 
         \hspace{6.5cm}
         \\
         \intertext{and if $ k = \lfloor \alpha+\beta\rfloor $ and $\{\alpha+\beta\}=\alpha+\beta-\lfloor \alpha+\beta\rfloor$ we also have\smallskip}
         & (ii) \quad [D^{\lfloor \alpha+\beta\rfloor}w(t,\cdot)]_{\{\alpha+\beta\}
         %\alpha+\beta- \lfloor \alpha+\beta\rfloor
         } \leq \overline{c}%^{(\lfloor \alpha+\beta\rfloor)}
         _{\gamma,\alpha,\beta} \sup_{s\in(0,T]} \|s^{\gamma}f(s,\cdot) \|_{C_b^\beta}t^{-\gamma}, 
    \end{align*}
    with constants $c^{(k)}_{\gamma,\alpha,\beta}$ and $\overline{c}_{\gamma,\alpha,\beta}$ only depending on $\alpha,\beta,\gamma,k$.
\end{thm}
\begin{rem}\label{rem:w-thm}
   (a) By assumptions $\alpha+\beta\in(1,3)$, so %in case $(ii)$,
   $\lfloor \alpha+\beta\rfloor$ %\in\{1,2\}$
    is either $1$ or $2$.
   \smallskip

\noindent (b) We do not consider 
   $\alpha+\beta\in \mathbb{N}$ in this paper. It is known that for equations of the form $\mL u = f$ with $\alpha+\beta\in \mathbb{N}$, it is possible that $u\notin C^{\alpha+\beta}_b(\R^d)$ and even $D^{\alpha+\beta-1}u$ is not Lipschitz, see \cite[Section~5]{MR4194536}. Repeating the proof below in this case
   yields an additional logarithmic factor in the modulus of continuity of $D^{\alpha+\beta-1}u$. 
  \smallskip
  
  \noindent (c) From the proofs it follows that
    \begin{align*}
    &c^{(k)}_{\gamma,\alpha,\beta}= c_{k,\beta,d} B(1-\gamma,\tfrac{\alpha+\beta-k}{\alpha}), \\  %&\overline{c}^{(k)}_{\gamma,\alpha,\beta} = 4c_{\beta,d} \big(B( 1-\gamma,\tfrac{\alpha+\beta-k}{\alpha})+\tfrac{1}{1-\gamma}+\tfrac{\alpha }{k+1-(\alpha+\beta)}+\tfrac{\alpha}{\alpha+\beta-k}\big),
    &%\overline{c}^{(k)}_{\gamma,\alpha,\beta}
    \overline{c}_{\gamma,\alpha,\beta}
    =  4c_{%k
    \lfloor \alpha+\beta\rfloor,\beta,d}\big(\tfrac 1{1-\gamma} + \tfrac{\alpha}{\{\alpha+\beta\}
    %\alpha+\beta-\lfloor \alpha+\beta\rfloor%k
    }\big)\vee 2\Big( c_{\lfloor \alpha+\beta\rfloor%k
    ,\beta,d}B(1-\gamma,\tfrac{\{\alpha+\beta\}}{\alpha}) + c_{\lfloor \alpha+\beta\rfloor+1,\beta,d}(\tfrac{\alpha}{%k+
    1-\{\alpha+\beta\}
    }+\tfrac{1}{1-\gamma})\Big),
\end{align*}
    with $c_{k,\beta,d}$ from %the constant from 
    Theorem \ref{lem::lem1} and 
 $B(s_1, s_2)=\int_0^1 x^{s_1-1}(1-x)^{s_2-1} \, dx$ for $s_1,s_2>0$ -- the beta function.
 %\smallskip\noindent (c)     
For $\gamma=0$ we have (slightly better) constants:
        $c^{(k)}_{0,\alpha,\beta} = \tfrac{\alpha c_{k,\beta,d}}{\alpha+\beta-k}$, %\qquad 
        $\overline{c}_{0,\alpha,\beta}= \alpha  \big(\tfrac{c_{k+1,\beta,d}}{k+1-(\alpha+\beta)}+\tfrac{2c_{k,\beta,d}}{\alpha+\beta-k}\big)$.
\end{rem}

The following proof is adapted from the arguments in \cite{deraynal2019schauder}.
\begin{proof}
    {\em (i)} \ Let $ t\in (0,T]$, $x\in\R^d$.
    %, and $k=1,\dots, m-1$. 
    Using Theorem \ref{lem::lem1} and setting $z=\frac{s}{t}$, 
    \begin{align*}
        |D^kw(t,x)| &\leq \int_{0}^t s^{-\gamma}|D^kP_{t-s}[s^{\gamma} f(s)](x)|\, ds  \\&\leq  \sup_{s \in (0,T]} \|s^{\gamma} f(s,\cdot) \|_{C_b^\beta} c_{k,\beta,d} \int_0^t s^{-\gamma}(t-s)^{-\frac{k-\beta}{\alpha}}\, ds \\ 
        &= t^{\frac{\alpha+\beta-k}{\alpha}-\gamma} \sup_{s\in(0,T]} \|s^{\gamma} f(s,\cdot)\|_{C_b^\beta} c_{k,\beta,d} \int_0^1 z^{-\gamma}(1-z)^{-\frac{k-\beta}{\alpha}} \, dz.
    \end{align*}
    Interchanging derivative and integral in the first line holds
    by the Dominated Convergence Theorem. 
\smallskip

\noindent {\em (ii)} \ Consider first the case $\alpha+\beta \in (1,2)$ and
$k=\lfloor\alpha+\beta\rfloor=1$.  %It then remains to estimate $[Dw]_{\alpha+\beta-1}$. 
We use Theorem \ref{lem::lem1} repeatedly.
    Let $x, x' \in \R^d$ and
    $$\Delta:=\Big|\int_{0}^tDP_{t-s}[f(s)](x)\, ds- \int_{0}^tDP_{t-s}[f(s)](x')\, ds\Big|.$$
    %and $t_0 = (t-|x-x'|^\alpha)\lor 0$. 
    If $|x-x'|^{\alpha}\geq t$, %(so $t_0=0$) 
    then 
    %the result is straightforward:
    \begin{align*}
        \Delta & \leq 2\int_{0}^t\big\|DP_{t-s}[f(s)]\big\|_\infty ds   \leq 2 c_{1,\beta,d} \sup_{s\in(0,T]}\|s^{\gamma} f(s,\cdot) \|_{C_b^\beta} \int_0^{t}s^{-\gamma}(t-s)^{-\frac{1-\beta}{\alpha}}ds \\
        &\leq 2c_{1,\beta,d}\sup_{s\in(0,T]}\|s^{\gamma} f(s,\cdot) \|_{C_b^\beta} \Big[\big(\tfrac{t}{2} \big)^{-\frac{1-\beta}{\alpha}} \int_0^{\frac{t}{2}}s^{-\gamma}ds +\big(\tfrac{t}{2}\big)^{-\gamma} \int_{\frac{t}{2}}^t (t-s)^{-\frac{1-\beta}{\alpha}}ds \Big]\\
        &=2c_{1,\beta,d}\sup_{s\in(0,T]}\|s^{\gamma} f(s,\cdot) \|_{C_b^\beta} \Big[\big(\tfrac{t}{2} \big)^{-\frac{1-\beta}{\alpha}} \tfrac{1}{1-\gamma}\big( \tfrac{t}{2}\big)^{1-\gamma} +\big(\tfrac{t}{2}\big)^{-\gamma} \tfrac{\alpha }{\alpha+\beta-1} \big(\tfrac{t}{2} \big)^{\frac{\alpha+\beta-1}{\alpha}}\Big] \\
        &=t^{-\gamma}t^{\frac{\alpha+\beta-1}{\alpha}}\sup_{s\in(0,T]}\|s^{\gamma} f(s,\cdot) \|_{C_b^\beta} \;2^{\tfrac{1-\beta}{\alpha}+\gamma}\;c_{1,\beta,d}  \Big(\tfrac{1}{1-\gamma} +\tfrac{\alpha}{\alpha+\beta-1} \Big) \\
        % &\leq t^{-\gamma}|x-x'|^{\alpha+\beta-1}\sup_{t\in(0,T]}\|t^{\gamma} f(t,\cdot) \|_{C_b^\beta} 2^{\frac{1-\beta}{\alpha}+\gamma}c_0  \bigg(\frac{1}{1-\gamma} +\frac{\alpha}{\alpha+\beta-1} \bigg).
         &\leq t^{-\gamma}|x-x'|^{\alpha+\beta-1}\sup_{s\in(0,T]}\|s^{\gamma} f(s,\cdot) \|_{C_b^\beta} \;4\;c_{1,\beta,d}  \Big(\tfrac{1}{1-\gamma} +\tfrac{\alpha}{\alpha+\beta-1} \Big).
    \end{align*}
    When $t_0:=|x-x'|^{\alpha}< t$, we split the integral  %defining $Dw$ 
    as $\int_0^t  = \int_0^{t-t_0} +\int_{t-t_0}^t $ obtaining $\Delta\leq \Delta_1+\Delta_2$. %We estimates the $\int_{t_0}^t$-term b
    By two changes of variables, including
    %For the second term we find using the
 %   into an ``off-diagonal'' and an ``on-diagonal'' part. For the off-diagonal part, with the 
 %change of variables 
 $z=\frac{s}{t_0}$, 
    \begin{align*}
       \Delta_2 & :=\Big| \int_{t-t_0}^t DP_{t-s}[f(s)](x)\, ds-  \int_{t-t_0}^t DP_{t-s}[f(s)](x')\, ds\Big| \\ 
       &\leq 2 \int_{0}^{t_0}%|x-x'|^{\alpha}} 
        \|DP_{s}[f(t-s)]\|_\infty\, ds \\
        %+ \int_{0}^{t_0} \|DP_{s}[f(t-s)]\|_\infty\, ds \\
        &\leq 2c_{1,\beta,d}  \sup_{s\in(0,T]}\|s^{\gamma} f(s,\cdot) \|_{C_b^\beta} \int_0^{t_0}(t-s)^{-\gamma}s^{-\frac{1-\beta}{\alpha}}\,ds \\
        &= t^{-\gamma} %|x-x'|^{\alpha+\beta-1}2c_0\sup_{s\in(0,T]}\|s^{\gamma} f(s,\cdot) \|_{C_b^\beta} \int_0^{1}\Big(1-\frac{|x-x'|^{\alpha}}{t}z\Big)^{-\gamma}z^{-\frac{1-\beta}{\alpha}}dz \\
      t_0^{\frac{\alpha+\beta-1}\alpha}2c_{1,\beta,d}\sup_{s\in(0,T]}\|s^{\gamma} f(s,\cdot) \|_{C_b^\beta} \int_0^{1}\big(1-\tfrac{t_0}{t}z\big)^{-\gamma}z^{-\frac{1-\beta}{\alpha}}dz \\
        & \leq t^{-\gamma} |x-x'|^{\alpha+\beta-1}2c_{1,\beta,d} \sup_{s\in(0,T]}\|s^{\gamma} f(s,\cdot) \|_{C_b^\beta} \int_0^{1}(1-z)^{-\gamma}z^{-\frac{1-\beta}{\alpha}}dz.
    \end{align*}
    %For the $\int_0^{t_0}$-term, we first 
    Then we note that 
    \begin{align*}
        \Delta_1&:=\Big| \int_0^{t-t_0} DP_{t-s}[f(s)](x)\, ds- \int_0^{t-t_0} DP_{t-s}[f(s)](x')\, ds\Big| \\
        &\leq |x-x'| \int_0^{t-t_0}\int_0^1 |D^2P_{t-s}[f(s)](x+\lambda(x-x'))|\,d\lambda\, ds \\
        & \leq  c_{2,\beta,d} \sup_{s\in(0,T]} \|s^{\gamma}f(s,\cdot) \|_{C_b^\beta} |x-x'|\: I_\gamma, 
    \end{align*}
    where $I_\gamma:=\int_{t_0}^t (t-s)^{-\gamma} s^{-\frac{2-\beta}{\alpha}}\, ds$. For $\gamma = 0$,
    \begin{align*}
        %\int_{t_0}^t s^{-\frac{2-\beta}{\alpha}} ds 
        I_0= \tfrac{\alpha}{2-(\alpha+\beta)} (t_0^{\frac{\alpha+\beta-2}{\alpha}}-t^{\frac{\alpha + \beta -2}{\alpha}})\leq \tfrac{\alpha}{2-(\alpha+\beta)}t_0^{\frac{\alpha+\beta-2}{\alpha}}.
    \end{align*}
    For $\gamma \in(0,1)$ we consider two cases. For $t_0\geq \frac{t}{2}$, 
    \begin{align*}
         %\int_{t_0}^t (t-s)^{-\gamma} s^{-\frac{2-\beta}{\alpha}} ds & 
         I_\gamma & 
         %\leq \int_{\frac{t}{2}}^t (t-s)^{-\gamma} s^{-\frac{2-\beta}{\alpha}} ds 
         \leq  t_0^{-\frac{2-\beta}{\alpha}} \int_{\frac{t}{2}}^t (t-s)^{-\gamma} ds 
         = %\frac{2^{\gamma+\frac{2-(\alpha+\beta)}{\alpha}}}
         t_0^{\frac{\beta-2}{\alpha}}  \tfrac1{1-\gamma} \big(\tfrac t 2\big)^{1-\gamma} \leq \tfrac{1}{1-\gamma} \big(\tfrac t 2\big)^{-\gamma}t_0^{\frac{\alpha+\beta-2}{\alpha}},
         % \leq \frac{2^{\gamma+\frac{2-(\alpha+\beta)}{\alpha}}}{1-\gamma} t^{-\gamma}|x-x'|^{\alpha+\beta-2}.
    \end{align*}
    %where the last inequality follows since $\alpha+\beta-2<0$.
    while for $t_0< \frac{t}{2}$,
        \begin{align*}
         I_\gamma 
         &\leq \big(\tfrac{t}{2} \big)^{-\gamma}\int_{t_0}^{\frac t2} s^{-\frac{2-\beta}{\alpha}}\, ds +\big(\tfrac{t}{2} \big)^{-\frac{2-\beta}{\alpha}}\int_{\frac t2}^{t}(t-s)^{-\gamma} \, ds\\
         &\leq \tfrac{\alpha}{2-\alpha-\beta}\big(\tfrac{t}{2} \big)^{-\gamma}t_0^{-\frac{2-\alpha-\beta}{\alpha}} + \tfrac{1}{1-\gamma}\big(\tfrac{t}{2} \big)^{-\frac{2-\alpha-\beta}{\alpha}-\gamma}\\
         &\leq 2^\gamma\big(\tfrac{\alpha}{2-\alpha-\beta} + \tfrac{1}{1-\gamma}\big)t^{-\gamma}t_0^{-\frac{2-\alpha-\beta}{\alpha}}.
    \end{align*}
Since $t_0=|x-x'|^\alpha$, in all cases we have shown that 
$$\Delta\leq \overline{c}%^{(1)}
_{\gamma,\alpha,\beta}\sup_{s\in(0,T]} \|s^{\gamma}f(s,\cdot) \|_{C_b^\beta}t^{-\gamma}|x-x'|^{\alpha+\beta-1},$$
which concludes the proof for $\alpha+\beta \in (1,2)$.\smallskip
    
    The case $\alpha+\beta \in (2,3)$, $k=\lfloor\alpha+\beta\rfloor=2$ is similar -- % once 
    we replace $D$ with $D^2$ in the definition of~$\Delta$.
\end{proof}
\begin{rem}
    It is well known that the solutions to the inhomogeneous heat equation for $\mL$,
   \begin{align}\label{linhe}
\left\{
    \begin{aligned}
         \partial_t u -\mathcal{L}u&=f(t,x), \qquad &(t,x)&\in (0,T]\times \R^d, \\ 
        u(0,x)&=u_0(x), \qquad &x &\in \R^d,
    \end{aligned}
\right.
\end{align}
can be expressed as $u(t,x) = P_t u_0(x) + \int_0^t P_{t-s}[f(s)](x)\, ds$. Theorems~\ref{lem::lem1} and \ref{prop:f_bound} immediately give Schauder estimates for \eqref{linhe}. One of our main goals is to use Theorems~\ref{lem::lem1} and \ref{prop:f_bound} to obtain Schauder regularity for the nonlinear HJ equation \eqref{eq:hjb} in Section~\ref{sec::full_schauder_regularity}. The estimates for \eqref{linhe} can also be viewed as a special case of the results below.
\end{rem}

\section{Short-time existence for the viscous HJ equation}\label{sec:shorttime}\label{sec::short_time_existence}

In this section we prove short-time existence of mild solutions of the viscous HJ equation \eqref{eq:hjb}.
We consider two interesting cases separately: (I) a setting with Lipschitz $u_0$ that allows for global in time Lipschitz solutions, and (II) a setting with H\"older $u_0$ where the gradient blows up as $t\to 0$. In case (I) (see Theorem \ref{thm::hjb_sol_existence}) existence holds under rather general local Lipschitz conditions on the Hamiltonian $H$, while in case (II) (see Theorem \ref{thm::case_b_hjb_sol_existence}) there is an interplay between the blowup of the gradient and the growth in $p$ of the Lipschitz bounds for $H(t,x,p)$.

%The distinction arises mainly as the result of considering respectively Lipschitz and H\"older initial conditions.

\begin{definition}%[Mild solution]
     A function $\varphi\in C((0,T]; C^1_b(\R^d))$ is a {\em mild solution} of \eqref{eq:hjb} if 
     %it is a fixed point of the Duhamel map:
    \begin{align}
    %S[\varphi](t,x) :=
    \varphi(t,x)=
    P_t[u_0](x)+\int_0^t P_{t-s}\big[ H(s,\cdot,D\varphi(s,\cdot))%+f(s,\cdot)
    \big](x)\, ds,\ t\in (0,T],\ x\in \R^d. \label{eq:duhamel_map}
\end{align}
\end{definition}
%ERJ: Let us keep this definition. We already have Lem 2.19 and it is relatively trivial in (4.1) to get this result}

We give a classical result on translation in time of the Duhamel map, presented as in \cite[Lemma 2.30]{Amund}. The proof follows from the semigroup property and linearity of $P_t$.
\begin{lem}
\label{lem::duhamel_map_time_translation}
    Assume $T>0$, $g_0\in L^{\infty}(\R^d)$, $g\colon (0,T]\times \R^d\to\R$ is measurable, %function %defined 
    %on $(0,T]\times \R^d$ such that satisfying
    $g(s,\cdot)\in L^\infty(\R^d)$ for $s\in(0,T]$, $\sup_{x\in\R^d}\int_0^{T}|P_{T-s}[g(s,\cdot)](x)| \, ds <\infty$, and define
    \begin{align*}
        \omega(t,x) = P_t[g_0](x)+\int_0^t P_{t-s}[g(s,\cdot)](x) \, ds, \quad t\in(0,T],\ x\in\R^d, 
    \end{align*}
    If $t_0, \tau, t_0+\tau \in(0,T]$,
    %and there is a constant $C>0$ such that 
    %    $\int_0^{T}|P_{T-s}[g(s,\cdot)](x)| \, ds \leq C$ for $x\in\R^d$,
    then 
    \begin{align*}
        \omega(t_0+\tau,x) = P_{\tau}[\omega(t_0,\cdot)](x)+\int_{t_0}^{t_0+\tau}P_{t_0+\tau-s}[g(s,\cdot)](x) \, ds.
    \end{align*}
\end{lem}

\subsection{Short-time existence of solutions without gradient blow-up}\ \smallskip

\noindent Here we work under assumptions ensuring that the gradient of the solution, $\|Du(t,\cdot)\|_{\infty}$, is bounded as $t\to 0$:
\medskip
% \begin{description}    
% \item[(H$_p$)\label{assump:H2}] $H\colon [0,T]\times \R^d \times \R^d\to \R$ is continuous, and for each $R>0$, there are constants $H_0,L_R>0$ such that for $t\in(0,T)$, $x\in\R^d$, and $p,q \in B(0,R)$,
%     $$|H(t,x,0)|\leq H_0 \qquad\text{and}\qquad|H(t,x,p)-H(t,x,q)| \leq L_R|p-q|.$$
% \smallskip     
%     \item[(U$_0$)\label{assump:H1}]$u_0$ is bounded and Lipschitz.\medskip
% \end{description}
\begin{enumerate}    
\myitem{$\mathbf{(H_p)}$}\label{assump:H2} $H\colon [0,T]\times \R^d \times \R^d\to \R$ is continuous, and for each $R>0$, there are constants $H_0,L_R>0$ such that for $t\in(0,T)$, $x\in\R^d$, and $p,q \in B(0,R)$,
    $$|H(t,x,0)|\leq H_0 \qquad\text{and}\qquad|H(t,x,p)-H(t,x,q)| \leq L_R|p-q|.$$
\smallskip     
    \myitem{$\mathbf{(U_0)}$}\label{assump:H1} $u_0$ is bounded and Lipschitz.\medskip
\end{enumerate}

\begin{rem}
Under \ref{assump:H2}, the Hamiltonian $H$ is locally Lipschitz in $p$, uniformly in $(t,x)$. An example is $H(t,x,p)=b(t,x)g(p)+f(t,x)$ where $b,f\in C_b([0,T]\times\R^d)$ and $g$ is $C^1(\R^d)$ or locally Lipschitz.
\end{rem}

We have the following short-time existence result.
\begin{thm}[Short time existence (I)]
\label{thm::hjb_sol_existence}
    Assume \ref{NDa}, % with $\alpha\in (1,2]$, 
    \ref{assump:H2}, and \ref{assump:H1}.
    Then there is a $T>0$ for which there exists a unique mild solution $u\in C_b((0,T]; C_b^1(\R^d))$ of \eqref{eq:hjb}.
 %the elliptic operator given by \eqref{eq:L}.
\end{thm}

Note that a mild solution is a fixed point of the map $S$ defined by  \begin{align}\label{eq:Duhmap}
&S[\phi](x,t) = P_t[u_0](x)+\int_0^t P_{t-s}\big[ H(s,\cdot,D\phi(s,\cdot))%+f(s,\cdot)
    \big](x)\, ds,\quad t\in (0,T],\ x\in \R^d.\\
\intertext{We will prove existence by an application of Banach fixed point theorem in the set\smallskip}
\label{eq:XAdef}
    &X_A  := \{\varphi \in C_b((0,T]; C_b^1(\R^d)) \ : \ \|\varphi\|_{X_A} \leq R_1:= \|u_0\|_{C_b^1}+1%\|f\|_{\infty} 
    \}, \\
\intertext{where}
    &\|\varphi \|_{X_A} := \sup_{t \in (0,T]} \|\varphi(t,\cdot) \|_{C_b^1} = \sup_{t \in (0,T]} \big( \|\varphi(t, \cdot)\|_\infty + \|D\varphi(t, \cdot) \|_\infty \big).\nonumber
\end{align} 
The following result will be used to 
show that $S$ is a contractive self-map 
 on $X_A$.
 %\footnote{Part (i) yields uniform boundedness of $S[u](t)$ and $DS[u](t)$, while (ii) and (iii) give the remaining information to conclude that $S[u]\in X_A$.}
\begin{lem}
\label{lem::duhamel_map_boundedness}
     Assume \ref{NDa}, %with $\alpha\in(1,2]$,
     %and $m=2$
     \ref{assump:H1}, \ref{assump:H2}, and $u,v\in X_A$. 
     \medskip

     \noindent (i) \ 
     For $t\in (0,T]$, $x \in \R^d$, and $k\in\{0,1\}$,
    \begin{align*}
         \qquad&|D^kS[u](t,x)|\leq \|D^ku_0 \|_\infty + (L_{R_1}R_1+H_0)  c_0^k \frac{T^{1-\frac k\alpha}}{1 - \frac k\alpha}.\\
    &|D^kS[u](t,x)-D^kS[v](t,x)| 
     \leq 
      L_{R_1}  \sup\limits_{t\in (0,T]} \|D^ku(t,\cdot)-D^kv(t,\cdot)\|_{\infty}   c_0^k \frac{T^{1-\frac k\alpha}}{1 - \frac k\alpha}.\\
     \intertext{\noindent (ii) \ For $t\in (0,T]$, $x \in \R^d$, and $\gamma\in(0,\alpha-1)$,}
     &[DS[u](t,\cdot)]_\gamma\leq\|Du_0\|_{\infty}\, %2c_0^{\gamma}
      2^{1-\gamma}c_{2,1,d} \, t^{-\frac{\gamma}{\alpha}}+%2
      4 c_0
     %( L_{R_1}R_1+\|f\|_{\infty} )
      (L_{R_1}R_1+H_0)\tfrac{\alpha}{\alpha-(1+\gamma)} t^{\frac{\alpha-(1+\gamma)}{\alpha}}. \\
     \intertext{\noindent (iii) \  
     The mapping $t \mapsto S[u](t, \cdot)$ is continuous in the $C_b^1$-norm, i.e. for $t\in(0,T]$,} 
     &\|S[u](t+\tau,\cdot) - S[u](t,\cdot)\|_{C_b^1} \xrightarrow{\tau \rightarrow 0} 0.
     \end{align*}
\end{lem}

%Part (iii) implies that the mapping $t \mapsto S[u](t, \cdot)$ is continuous in the $C_b^1$-norm.  
Now we can prove existence.
\begin{proof}[Proof of Theorem \ref{thm::hjb_sol_existence}]
    If 
    % \begin{align}
    %  T \leq T_0:= \big(\frac{1- \frac1\alpha}{4L_{R_1}c_0}\big)^{\frac{\alpha}{\alpha-1}}\wedge \frac 1{2L_{R_1}R_1},\label{eq:intervallength}
    % \end{align}
     \begin{align}
     T \leq T_0:=\bigg(\frac{\alpha-1}{2\alpha c_0(L_{R_1}R_1+H_0)}\bigg)^{\frac{\alpha}{\alpha-1}} \ \wedge \ \frac{1}{2(L_{R_1}R_1+H_0)} ,\label{eq:intervallength}
    \end{align}
    then by Lemma \ref{lem::duhamel_map_boundedness} (i), there is $C<1$ such that
    $$\|S[u]\|_{X_A} \leq R_1 \qquad\text{and}\qquad \|S[u]-S[v]\|_{X_A} \leq C\|u-v\|_{X_A}.$$
By Lemma~\ref{lem::duhamel_map_boundedness} (ii) and (iii) we therefore get that $S[u]\in X_A$. The existence of a fixed point of $S$ in $X_A$ then follows by Banach's fixed point theorem.
%
% Now we choose $T>0$ small enough so that $S$ is a contractive map. Then, by Banach's fixed point theorem, there exists a unique fixed point $u\in X_A$ of $S$. 
\end{proof}

It remains to prove Lemma \ref{lem::duhamel_map_boundedness}.

\begin{proof}[Proof of Lemma \ref{lem::duhamel_map_boundedness}]
Note that by \ref{assump:H2},   
$|H(t,x,Du)| 
%= |H(t,x,Du) - H(t,x,0)| + |H(t,x,0)| 
\leq L_{R_1}R_1+H_0$ for $u \in X_A$. %\leq L_{R_1} R_1$.
% $|H(t,x,Du)| \leq K_{R_1}$.
\smallskip

\noindent $(i)$ %Let $k=0,1$. 
By the definition of $S$, Young's inequality for convolutions, \ref{assump:H2}, $\|p_t\|_{L^1}=1$, and \ref{NDa},
\begin{align*}
    |D^kS[u](t,x)| & \leq |D^kP_t[u_0](x)|+ \int_0^t |D^kp_{t-s}* H( s,\cdot, Du(s,\cdot))(x)|\, ds \\
    &\leq \|D^ku_0 \|_\infty + \int_0^t \|D^kp_{t-s}\|_{L^1} \|H(s,\cdot, Du(s,\cdot))\|_{\infty}  \, ds \\
    &\leq\|D^ku_0 \|_\infty +(L_{R_1}R_1+H_0)\cdot \begin{cases}
        T, & k=0, \\
        c_0\frac{T^{1-\frac{1}{\alpha}}}{1-\frac{1}{\alpha}},& k=1.
    \end{cases}
\end{align*}
%which is finite by \ref{assump:H1}, \ref{assump:H2}  \ref{assump:H3}, and the definition of $X_A$. 
%In the final inequality we used \ref{NDa} with $\sigma=0$ for $k=1$, and that $p_t$ is a probability measure for $k=0$. 
Similar computations for the differences conclude the proof of $(i)$: 
\begin{align*}
    \|D^k(S[u]-S[v])(t,\cdot)\|_{\infty} &\leq \int_0^t \sup_{x\in \R^d} \big| D^kp_{t-s}*\big(H(s,\cdot,Du)-H(s,\cdot,Dv)\big)(x) \big| \, ds \\
    &\leq  L_{R_1}\int_0^t \| Du(s,\cdot)-Dv(s,\cdot)\|_{\infty} c_0^k (t-s)^{-\frac{k}{\alpha}} \, ds.
\end{align*}

\noindent $(ii)$ 
%We will show that the Hölder seminorm of $DS[u](t,\cdot)$ is finite. Begin by writing
Note that
    \begin{align*}
    [DS[u](t,\cdot)]_\gamma %&=\sup_{\substack{x, h\in \R^d \\ h \neq 0}} \frac1{|h|^{\gamma}} \big|DS[u](t,x+h)-DS[u](t,x)\big| \\ 
     %&\leq \sup_{\substack{x, h\in \R^d \\ h \neq 0}}\Big( 
     \leq [ DP_t u_0 ]_{\gamma} 
     %\frac{1}{|h|^{\gamma}}\big| DP_t\big[u_0\big](x+h)-DP_t\big[u_0\big](x) \big| 
     + \int_0^t [g_{t-s}]_\gamma \, ds 
     \end{align*}
     for
      $g_{t-s}(x):=\int_{\R^d} Dp_{t-s}(x-y) H(s,y,Du(s,y)) \, dy.$
By interpolation (Lemma \ref{thm:holder_interpolation}), \ref{assump:H2}, and \ref{NDa},
\begin{align*}
    [g_{t-s}]_{\gamma} & \leq 2^{1-\gamma} \|g_{t-s} \|_\infty^{1-\gamma} \|Dg_{t-s} \|_\infty^{\gamma} \leq  2^{1-\gamma} \|H(\cdot,\cdot,Du)\|_{\infty} \|Dp_{t-s}\|_{L^1}^{1-\gamma}\|D^2p_{t-s}\|_{L^1}^\gamma\\
    &\leq 2^{1-\gamma} (L_{R_1}R_1+H_0)
    %c_0
    c_0^{1+\gamma}2^{\frac {2\gamma}\alpha} (t-s)^{-\frac{1+\gamma}{\alpha}}.
\end{align*}
%Similarly by Lemma \ref{thm:holder_interpolation}, \ref{NDa}, and since $\|p_t\|_{L^1}=1$, we have that 
%we complete the proof observing that by similar arguments, 
% and \ref{NDa}, the first term is bounded by 
By Theorem \ref{lem::lem1} we have that $[ DP_t u_0 ]_{\gamma} \leq 2^{1-\gamma}c_{2,1,d}\|Du_0 \|_\infty t^{-\frac{\gamma}{\alpha}}$.
\smallskip

\noindent $(iii)$ We show that $(0,T]\ni t \mapsto S[u](t, \cdot)$ is continuous in the $C_b^1$-norm. %Note that by $(ii)$ we have $S[u](t, \cdot)\in C_b^1(\R^d)$ for $t\in (0,T]$.
%In the following, let $k=0,1$. 
Let %where (without loss of generality) 
$\tau>0$ and $t+\tau \in(0,T]$ and note that by Lemma~\ref{lem::duhamel_map_time_translation}, assumptions \ref{assump:H2} and \ref{NDa}, and the fact that $\|p_t\|_{L^1}=1$,
\begin{align*}
    &\big| D^kS[u](t+\tau, x) - D^kS[u](t, x) \big| \\
    & \leq \underbrace{\big| p_\tau * D^kS[u](t,\cdot)(x) -  D^kS[u](t, x)\big|}_{A_k} + \int_t^{t+\tau}| D^kp_{t+\tau-s}* H(s,\cdot,Du)(x)| \, ds \\
    &\leq {A_k} + \big( L_{R_1}R_1 + H_0 \big)\cdot 
    \begin{cases}
        \tau, &\quad k=0, \\
         c_0 \frac{\tau^{1-\frac{1}{\alpha}}}{1-\frac{1}{\alpha}}, &\quad k=1.
    \end{cases} 
\end{align*}
To estimate $A_k$, fix $t\in(0,T)$, and let $\eta_0 = 1$ and $\eta_1 =\gamma$ such that $1<1+\gamma<\alpha$. Given $\epsilon>0$, pick $\delta>0$, and $\tau_0>0$ small enough, 
 such that
\begin{align*}
    &\delta^{\eta_k} < \frac{\epsilon}{2[D^kS[u](t,\cdot)]_{\eta_k}} \quad \text{and} \quad \int_{\R^d \setminus B_\delta(0)}p_\tau(y) \, dy < \frac{\epsilon}{4 \|D^k S[u]\|_\infty}, \quad 0<\tau<\tau_0,\ k=0,1
\end{align*}
hold, which is possible since $p_\tau\stackrel{\ast}{\rightharpoonup} \delta_0$ in $C_b(\R^d)$ by \cite[(3.2)]{pruitt_growth_random_walks} and part $(ii)$ above. Then,
\begin{equation} \label{eq:A_k_calculation}
\begin{aligned}
    A_k %&=| p_\tau * D^kS[u](t,\cdot)(x) -  D^kS[u](t, x)| \\ 
    &=  \Big| \int_{\R^d}p_\tau(y) \big[D^kS[u](t,x-y) -  D^kS[u](t, x)\big] \, dy \Big| \\
    & \leq \int_{B_\delta(0)}p_\tau(y) [D^kS[u](t,\cdot)]_{\eta_k}|y|^{\eta_k} \, dy + 2 \|D^kS[u] \|_\infty\int_{\R^d \setminus B_\delta(0)}p_\tau(y) \, dy \\
    & < [D^kS[u](t,\cdot)]_{\eta_k} \delta^{\eta_k} + \frac{\epsilon}{2} < \epsilon,
\end{aligned}
\end{equation}
so since $\epsilon>0$ was arbitrary, $A_k \xrightarrow{\tau \rightarrow 0} 0$ uniformly in $x$ for each fixed $t$. As a result, 
\begin{align*}
    \|S[u](t+\tau,\cdot) - S[u](t,\cdot)\|_{C_b^1} \xrightarrow{\tau \rightarrow 0} 0,
\end{align*}
for each $t\in(0,T)$. The proof for $t=T$ is similar, with $t-\tau$ instead of $t+\tau$. 
\end{proof}
\begin{rem}
    (a) When $\alpha\in (1,2)$, the claim in Lemma \ref{lem::duhamel_map_boundedness} $(ii)$ holds also for $\gamma=\alpha-1$. This follows from using Theorem~\ref{prop:f_bound} instead of interpolation on the integral term.

    \medskip
    \noindent (b) If $u_0\in C^{1+\epsilon}_b(\R^d)$ for $0<\epsilon <\alpha-1$, then  $[ DP_t u_0 ]_{\gamma}=[ P_t Du_0 ]_{\gamma} \leq [Du_0 ]_\gamma $, and Lemma \ref{lem::duhamel_map_boundedness} $(ii)$ holds for $\gamma \in (0,\epsilon)$ without blow-up in time. Then $(iii)$ holds on $[0,T]$, that is $t=0$ can be included.    
\end{rem}

\subsection{Short-time existence of solutions with gradient blow-up}\ \smallskip

\noindent We consider cases where the solution $u$ satisfies $\|Du(t,\cdot)\|_{\infty} = O(t^{-\gamma})\text{ as } t \rightarrow 0$ for some $\gamma>0$, using the following assumptions:

\medskip
\begin{enumerate}
\myitem{$\mathbf{(H_p')}$} \label{assump:B2_new}
%\item[({H$_p$}$'$)\label{assump:B2_new}]
$H: (0,T]\times \R^d \times \R^d\to \R$ is continuous, and there are $r\geq1$, 
$H_0,L>0$, such that for $t\in(0,T)$ and $x,p,q \in \R^d$, 
    $$|H(t,x,0)|\leq H_0 \qquad\text{and}\qquad|H(t,x,p)-H(t,x,q)| \leq L(1+|p|^2+|q|^2)^{\frac {r-1}2}|p-q|.$$
\smallskip
    
    \myitem{$\mathbf{(U_0')}$}\label{assump:B1}
%\item[({U$_0$}$'$)\label{assump:B1}]
$u_0\in C_b^\delta(\R^d)$ \quad for some $\delta\in[0,1)$ satisfying \quad 
    $\begin{cases}
        \delta = 0,\quad & r \in [1,\alpha),\\[0.1cm]
        \delta\in (0,1),\quad &r=\alpha,\\[0.1cm]
        \delta \in \big(\frac{r-\alpha}{r-1},1\big),\quad &r\in(\alpha,\infty).
    \end{cases}$
\medskip
    \end{enumerate}
\begin{rem}
    (a) \,By \ref{assump:B2_new}, $|H(t,x,p)|\leq L(1+|p|^2)^{\frac r2} + H_0$ and $r$ is the growth rate of $H$ in $p$. An example is $H(t,x,p)=b(t,x)|p|^r+f(t,x)$ where $b,f\in C_b(\R^d)$. \smallskip
    
    \noindent (b) \,Note the interplay between $\delta$ and $r$ in \ref{assump:B2_new} and \ref{assump:B1}; higher $r$ requires higher $\delta$.  Moreover, for any $r\geq 1$, there are $\delta\in[0,1)$ such that \ref{assump:B1} and \ref{assump:B2_new} hold, and vice versa for any $\delta$ there are $r$'s.
\end{rem}

We now show the short-time existence of mild solutions.
\begin{thm}[Short time existence (II)]
\label{thm::case_b_hjb_sol_existence}
    Assume  \ref{NDa}, %with $\alpha\in (1,2]$,
    %and $m=2$, 
    \ref{assump:B2_new}, and \ref{assump:B1}. Then there is $T>0$ for which there exists a  unique mild solution $u\in C((0,T]; C_b^1(\R^d))$ of \eqref{eq:hjb}, and %such that 
    $\|Du(t,\cdot)\|_\infty = O(t^{-(1-\delta)/\alpha})$ as $t\to 0$.
\end{thm}

To prove this result, we use a fixed point argument on the space $X_B$, defined for $\gamma\in (0,1)$ as
\begin{align*}
    X_B &:= \{\phi\in C((0,T]; C_b^1(\R^d)): \|\phi\|_{X_B} \leq R_2\},\ \text{where}\\[0.1cm] % \ : \ \|\varphi\|_{X_B} \leq  R_\delta:= 2\|u_0\|_{C_b^\delta}+\|f\|_{\infty} (check!) \}, \\
\|\varphi \|_{X_B} &:= \sup_{t \in (0,T]} \Big( \|\varphi(t, \cdot)\|_\infty + t^{\gamma}\|D\varphi(t, \cdot) \|_\infty \Big).
\end{align*}
The constant $R_2$ can be made explicit by inspecting the computations below.

\begin{lem}\label{lem::case_b_duhamel_map_boundedness}
    Assume \ref{NDa}, % with $\alpha\in (1,2]$,
    %and $m=2$, 
    \ref{assump:B2_new}, \ref{assump:B1}, $u,v\in X_B$, and $S[u]$ is defined in \eqref{eq:Duhmap}. If 
    \begin{align*}
    (i) \quad \gamma r < 1, \qquad (ii) \quad \frac{1-\delta}\alpha\leq \gamma, \qquad \text{and}\qquad (iii) \quad \frac1\alpha+\gamma(r-1)<1,
    \end{align*}
    then there are constants $C_1,C_2>0$ such that
\begin{align*}
\tag{a} &\|S[u]\|_{X_B} \leq  \|u_0\|_\infty+ c_{1,\delta,d}[u_0]_\delta T^{\gamma-\frac{1-\delta}\alpha}\\%[0.1cm]
&\hspace{3.5cm}+%C_1 \Big(T+T^{\gamma+1-\frac1\alpha}+(T^{2\gamma}+2R_2^2)^{\frac{r-1}{2}}\big( T^{1-\gamma r}+T^{1-\frac 1\alpha -\gamma (r-1)} \big)\Big),\hspace{1cm}
C_1 \Big(T+T^{\gamma+1-\frac1\alpha}+ L(T^{2\gamma}+R_2^2)^{\frac{r}{2}} \big( T^{1-\gamma r}+T^{1-\frac 1\alpha -\gamma (r-1)} \big)\Big),\hspace{1cm}
\\[0.1cm]
\tag{b}&\|S[u] - S[v]\|_{X_B}  \leq \ C_2\|u - v\|_{X_B}(T^{2\gamma}+2R_2^2)^{\frac{r-1}{2}}
          \big(T^{1-\gamma r}+T^{1-\frac{1}{\alpha}-\gamma(r-1)} \big).
\end{align*}     
\end{lem}
\begin{proof}
\noindent (a) \ Let $k=0,1$. By the definition of $S$ we have
\begin{align*}
    |D^kS[u](t,x)| & \leq |D^kP_t[u_0](x)|+ \int_0^t |D^kp_{t-s}*H(s,\cdot,Du(s,\cdot))(x)|\, ds.
\end{align*}
By \ref{assump:B2_new}, \ref{NDa}, and since $\|p_t\|_{L^1}=1$,
\begin{align*}
    &\int_0^t |D^kp_{t-s}*H(s,\cdot,Du(s,\cdot))(x)| \, ds \\
    & \leq H_0 c_0^k\frac{t^{1-\frac{k}{\alpha}}}{1-\frac{k}{\alpha}} +  Lc_0^k\Big(t^{2\gamma}+ \sup_{t\in(0,T]}\|t^{\gamma}Du(t,\cdot)\|_\infty^2\Big)^{\frac{r}{2}}\int_0^t (t-s)^{-\frac{k}{\alpha}}s^{-\gamma r} \, ds \\
    &\leq H_0\cdot\begin{cases}
     t, & \ k=0,\\[0.2cm]
     c_0\frac{t^{1-\frac{1}{\alpha}}}{1-\frac{1}{\alpha}}, & \ k=1,
    \end{cases} \quad + \quad L\Big(T^{2\gamma}+R_2^2\Big)^{\frac{r}{2}}\cdot \begin{cases}
        \frac{t^{1-\gamma r}}{1-\gamma r}, & \ k=0,\\[0.2cm]
        t^{1-\frac{1}{\alpha}-\gamma r}c_0 %\int_0^t (t-s)^{-\frac{1}{\alpha}}s^{-\gamma r } \, ds%C_{\alpha,\gamma, r}
        \int_0^1 (1-z)^{-\frac{1}{\alpha}}z^{-\gamma r } \, dz %C_{\alpha,\gamma, r}
        %
        %B(1-\gamma r, 1-\frac{1}{\alpha})
        , & \ k=1.
    \end{cases}
\end{align*}
By \ref{assump:B1}, \ref{NDa}, $\|p_t\|_{L^1}=1$, and Theorem \ref{lem::lem1}, we have
\begin{align*}
    |P_t[u_0](x)|\leq\|u_0\|_\infty \qquad \text{and} \qquad \|DP_t[u_0]\|_\infty \leq c_{1,\delta, d} %c_{\delta, d}
    [u_0]_\delta  t^{-\frac{1-\delta}{\alpha}}.
\end{align*}
 The result now follows by combining the estimates to compute $\|S[u](t,\cdot)\|_\infty+t^\gamma\|DS[u](t,\cdot)\|_\infty$, and noting that by the assumptions on $\alpha,\delta,\gamma,r$, the integral is finite and all the resulting powers of $t$ are nonnegative.
\medskip
   
\noindent (b) \ Let $k=0,1$. By \ref{assump:B2_new} and the definition of $S$,
    \begin{align*}
        & |D^k(S[u]-S[v])(t,x)| \\ &\leq \int_0^t \|D^kp_{t-s} \|_{L^1} \| H(s, \cdot, Du(s,\cdot))-H(s, \cdot, Dv(s,\cdot)) \|_\infty \; ds \\
        & \leq c_0^k L \int_0^t (t-s)^{-\frac{k}{\alpha}} \big(1+\| Du(s,\cdot)\|_\infty^{2}+\| Dv(s,\cdot)\|_\infty^{2} \big)^{\frac{r-1}{2}}\|Du(s,\cdot)-Dv(s,\cdot) \|_\infty \; ds \\
        &\leq c_0^k L \|u-v\|_{X_B} \sup_{t\in(0,T]}\big(t^{2\gamma}+\| t^{\gamma}Du(t,\cdot)\|_\infty^{ 2}+\|t^\gamma Dv(t,\cdot)\|_\infty^{ 2} \big)^{\frac{r-1}{2}} \int_0^t (t-s)^{-\frac{k}{\alpha}}s^{-\gamma r } \, ds \\
        % &\leq L \|u-v\|_{X_B} \sup_{t\in(0,T]}\big(t^{2\gamma}+\| t^{\gamma}Du(t,\cdot)\|_\infty+\|t^\gamma Dv(t,\cdot)\|_\infty \big) \\ &\qquad \cdot \begin{cases}
        %     \frac{t^{1-\gamma r}}{1-\gamma r}, &\quad k=0, \\
        %     t^{1-\frac{1}{\alpha}-\gamma r}c_0 \int_0^1 (1-z)^{-\frac{1}{\alpha}}z^{-\gamma r} \; dz, &\quad k=1,
        % \end{cases}
        &\leq L \|u-v\|_{X_B} \big(T^{2\gamma}+2R_2^2 \big)^{{\frac{r-1}{2}}} \cdot\begin{cases}
            \frac{t^{1-\gamma r}}{1-\gamma r}, &\quad k=0, \\[0.2cm]
            t^{1-\frac{1}{\alpha}-\gamma r}c_0 \int_0^1 (1-z)^{-\frac{1}{\alpha}}z^{-\gamma r} \; dz, &\quad k=1,
        \end{cases}
        \end{align*}
    where we have done the change of variables $z=\frac{s}{t}$ in the final line for $k=1$.  The result now follows by combining the estimates to compute $\|(S[u]-S[v])(t,\cdot)\|_\infty+t^\gamma\|D(S[u]-S[v])(t,\cdot)\|_\infty$, and noting that by the assumptions on $\alpha,\gamma,r$, the integral is finite and all the resulting powers of $t$ are nonnegative.
\end{proof}
It follows that $S$ is a contraction on $X_B$ for small $T$, if assumptions {\em(i) -- (iii)} in Lemma \ref{lem::case_b_duhamel_map_boundedness} hold:
\begin{align*}
(i) \quad \gamma < \frac{1}{r}, \qquad (ii) \quad  \frac{1-\delta}{\alpha}\leq \gamma, \qquad (iii)\quad \gamma < \frac{1-\frac{1}{\alpha}}{r-1}.
\end{align*}
The powers of $T$ in Lemma \ref{lem::case_b_duhamel_map_boundedness} (b) are then positive and the $T$-factor can be made small by taking $T$ small. 
% We need a sharp inequality in $(iii)$ to have an additional factor of the form $T^{\epsilon}$ in the estimates, which enables us to make the expressions small for small $T$.
The above conditions can be reformulated as \ref{assump:B1}. Indeed, note that 
\begin{align*}
    \frac{1}{r} \leq \frac{1-\frac{1}{\alpha}}{r-1}\quad  \iff \quad r\leq \alpha.
\end{align*}
Elementary calculations show that if
\ref{assump:B1} holds, then $(i)$, $(ii)$, and $(iii)$ are satisfied with $\gamma = (1-\delta)/\alpha$.
% To fix the discussion, we start by choosing $\gamma$ and then pick the smallest possible $\delta$ such that (iii) is satisfied. Note that it makes sense to take $\gamma$ as large as possible, because it allows us to have $\delta$ smaller, that is, we get weaker assumptions on the initial condition. This leads to
% \begin{align}\label{dg}
% \begin{split}
%      \gamma &= \Big(\frac{1}{r}\wedge \Big(\frac{1-\frac{1}{\alpha}}{r-1}\Big)\Big) - \epsilon = 
%      \begin{cases}
%         \frac{1}{r}-\epsilon, & \ r\leq \alpha\\[0.2cm]
%         \frac{1-\frac{1}{\alpha}}{r-1} -\epsilon, & \ r>\alpha,
%     \end{cases}\quad \text{and}\quad
%     \delta=
%     (1-\alpha\gamma)_+,
%     \end{split}
% \end{align}
% where $\epsilon$ can be taken arbitrarily small.
\begin{proof}[Proof of Theorem~\ref{thm::case_b_hjb_sol_existence}]
    We use Banach's fixed point theorem in $X_B$ with $R_2=\|u_0\|_\infty+ c_{1,\delta,d}[u_0]_\delta+1$. For small enough $T$, $\|S[u]\|_{X_B}\leq R_2$ for all $u\in X_B$ by Lemma~\ref{lem::case_b_duhamel_map_boundedness} (a). Moreover, $S[u](t,\cdot)\in C^1_b(\R^d)$ and is continuous in the  $C^1_b(\R^d)$-norm by Lemma~
    \ref{lem::duhamel_map_boundedness} $(ii)$ and $(iii)$ and Lemma~\ref{lem::duhamel_map_time_translation}, since functions in 
    $X_B$ have uniformly bounded spatial gradients at $t\in (\varepsilon,T]$ for any $\varepsilon>0$. It follows that $S$ maps $X_B$ into $X_B$. Contractivity for small $T$ follows from Lemma \ref{lem::case_b_duhamel_map_boundedness} (b).
\end{proof}

\section{Optimal spatial Schauder-regularity for the viscous HJ equation}\label{sec::full_schauder_regularity}
Assuming mild solutions exist on some time interval $[0,T]$, in this section we show that they attain optimal $C^{\alpha+\beta}_b$ H\"older regularity in $x$,
a gain of order $\alpha$ (= the order of $\mL$) over the $C^\beta_b$ regularity in $x$ of $H(t,x,p)$ (see \ref{assump:H} and \ref{assump:B3} below).
%, i.e. from $C^\beta$ to $C^{\alpha+\beta}$. 
When $\mL$ is a local second order elliptic operator, this corresponds to classical Schauder regularity theory \cite{MR1465184}. We are also interested in the blow-up rates as $t\to0$ of the optimal H\"older norms, and therefore we consider separately the cases where the spatial gradients of solutions are uniformly bounded in time or not.

\subsection{Schauder regularity for solutions with uniformly bounded gradients}\ \smallskip
    
\noindent We assume H\"older regularity in $x$ of $H(t,x,p)$, locally uniformly in $p$:
\medskip
% \begin{description}
% \item[(H$_x$)\label{assump:H}] There exists $\beta \in [0,1]$, such that for each $R>0$ there is a constant $M_R>0$ for which for all $t\in[0,T]$, $p\in B(0,R)$, and $x,x' \in \R^d$ with $|x-x'|<1$,  
%       $$|H(t, x, p)-H(t, x', p)| \leq M_R|x-x'|^\beta.$$
% \end{description}
\begin{enumerate}
\myitem{$\mathbf{(H_x)}$}\label{assump:H}
%\item[(H$_x$)\label{assump:H}] 
There exists $\beta \in [0,1]$, such that for each $R>0$ there is a constant $M_R>0$ for which for all $t\in[0,T]$, $p\in B(0,R)$, and $x,x' \in \R^d$ with $|x-x'|<1$,  
      $$|H(t, x, p)-H(t, x', p)| \leq M_R|x-x'|^\beta.$$
\end{enumerate}

\begin{rem}
    This condition is consistent with \ref{assump:H2}, and when $\beta=0$ it follows from \ref{assump:H2}. An example satisfying both \ref{assump:H2} and \ref{assump:H} is $H(t,x,p)=b(t,x)g(p)+f(t,x)$ for $b,f\in C_b([0,T];C^\beta(\R^d))$ and locally Lipschitz $g$.
\end{rem}
We now give the $C^{\alpha+\beta}_b$ Schauder regularity result for solutions that are Lipschitz in $x$, uniformly in~$t$.
\begin{thm}[Schauder regularity (I)]
\label{thm::full_schauder_regularity}
Assume \ref{assump:H2}, \ref{assump:H1}, \ref{assump:H} for some $\beta \in [0,1]$, \ref{NDa} with $\alpha\in (1,2]$,
%and
 %$m=\lceil \alpha+\beta \rceil$, 
 $\alpha+\beta \notin \mathbb{N}$, $T>0$, 
and $u\in C_b((0,T],C^1_b(\R^d))$ is a mild solution of \eqref{eq:hjb}.  Then the following statements hold:\smallskip

\noindent (i) \ 
%Then  
$u(t)\in C_b^{\alpha +\beta}(\R^d)$ for $t\in (0,T]$,\medskip

\noindent (ii) \ 
%We have
$%[D^{ \lceil \alpha+\beta \rceil-1} u(t,\cdot)]_{\alpha+\beta-(\lceil \alpha+\beta \rceil-1)} 
[D^{\lfloor \alpha+\beta \rfloor} u(t,\cdot)]_{\{\alpha+\beta\}%\alpha+\beta-\lfloor \alpha+\beta \rfloor
} = O(t^{-\frac{\alpha+\beta-1}{\alpha}})$ as $t \rightarrow 0$,\qquad where $\{\alpha+\beta\}=\alpha+\beta-\lfloor \alpha+\beta \rfloor$.\medskip

\noindent (iii) \ If $\alpha+\beta> 2$, then $\| D^2u(t,\cdot)\|_\infty=O(t^{-\frac{1}{\alpha}})$ as $t\rightarrow 0$.
\end{thm}
\begin{rem} 
This a priori result holds for arbitrary $T>0$ provided a mild solution exists in $[0,T]$.
\end{rem}

To prove Theorem \ref{thm::full_schauder_regularity} we first show an intermediate result on $C^{1+\beta}_b$ regularity of $u$. Then a bootstrap argument using Theorem~\ref{prop:f_bound} will give the full regularity.

\begin{lem}\label{lem:H_regularity}
    Assume \ref{NDa} with $\alpha\in (1,2]$,
    %and $m=2$,
    \ref{assump:H2}, \ref{assump:H1}, and \ref{assump:H} holds with some $\beta\in [0,1]$. If $u\in C_b((0,T];C^1_b(\R^d))$ is a mild solution of \eqref{eq:hjb}, then %$Du(t,\cdot), \ H(t, \cdot, %Du(t,\cdot)) \in %C^\beta_b(\R^d)$ for each $t\in (0,T]$, and %if $\beta>0$, 
    $$\sup_{t \in (0,T]}[t^{\frac{\beta}{\alpha}}Du(t,\cdot)]_{\beta}<\infty \qquad \text{and} \qquad \sup_{t \in (0,T]}[t^{\frac{\beta}{\alpha}}H(t,\cdot, Du(t,\cdot))]_{\beta}<\infty.$$
\end{lem}
\begin{proof}
    The result is immediate for $\beta=0$ by 
    %Theorem \ref{thm::hjb_sol_existence} and Assumption
    \ref{assump:H2}, so consider $\beta>0$. We begin as in the proof of Lemma \ref{lem::duhamel_map_boundedness} (ii), but now we exploit the regularity of $H$. 
By the definition of a mild solution,
\begin{align*}
    \frac{|Du(t,x+h)-Du(t,x)|}{|h|^\beta} \leq \frac{|DP_tu_0(x+h)-DP_tu_0(x)|}{|h|^\beta}+\int_0^t \frac{|g_{t-s}(x+h)-g_{t-s}(x)|}{|h|^\beta},
\end{align*}
where $g_{t-s}(x):=\int_{\R^d} Dp_{t-s}(y) H(s,x-y,Du(s,x-y)) \, dy$. 

 By 
Theorem \ref{lem::lem1},  $[ DP_t u_0 ]_{\beta}\leq \|Du_0\|_{\infty}2^{1-\beta}c_{2,1,d}t^{-\frac{\beta}{\alpha}}$. Assumptions \ref{assump:H2} and \ref{assump:H} yield
\begin{align*}
    &|g_{t-s}(x+h)-g_{t-s}(x)|\\
    %&=\Big|\int_{\R^d} Dp_{t-s}(y)( H(s,x-y+h,Du(s,x-y+h))-H(s,x-y,Du(s,x-y)) )\, dy \Big| \\
    &\leq \Big|\int_{\R^d} Dp_{t-s}(y)\big(  H(s,x-y+h,Du(s,x-y+h))-H(s,x-y+h,Du(s,x-y))\big)\, dy \Big| \\
    &\qquad + \Big|\int_{\R^d} Dp_{t-s}(y)\big(   H(s,x-y+h,Du(s,x-y))-H(s,x-y,Du(s,x-y)) \big) \, dy \Big| \\
    % &\leq  c_0(t-s)^{-\frac{1}{\alpha}}L_{R_1}\sup_{x\in\R^d}|Du(s,x+h)-Du(s,x)|+c_0  (t-s)^{-\frac{1}{\alpha}}M_{R_1}|h|^\beta. \\ 
    % & \leq  c_0(t-s)^{-\frac{1}{\alpha}}L_{R_1}\sup_{x\in\R^d}|Du(s,x+h)-Du(s,x)| \\ 
    % & \qquad \qquad \qquad \qquad \qquad \qquad +c_0  (t-s)^{-\frac{1}{\alpha}}\cdot \begin{cases}
    % M_{R_1}|h|^\beta, &|h|<1,\\
    % 2(L_{R_1}\|Du \|_\infty+H_0), &|h|\geq 1.
    % \end{cases}    
    &\leq  c_0(t-s)^{-\frac{1}{\alpha}}\bigg( L_{R_1}\sup_{x\in\R^d}|Du(s,x+h)-Du(s,x)| + \begin{cases}
    M_{R_1}|h|^\beta, &|h|<1,\\
    2(L_{R_1}R_1+H_0)|h|^\beta, &|h|\geq 1
    \end{cases}    \bigg),
\end{align*}
where $R_1=\|Du\|_\infty$.
Let $\overline M_{R_1}:=\max \{M_{R_1}, 2(L_{R_1}R_1+H_0) \}$. Integrating in time, we see that
\begin{align*}
    % &\int_0^t |g_{t-s}(x+h)-g_{t-s}(x)| \; ds \\ 
    \frac{1}{|h|^\beta}\int_0^t &|g_{t-s}(x+h)-g_{t-s}(x)| \; ds \\ &\leq c_0\frac{T^{1-\frac{1}{\alpha}}}{1-\frac{1}{\alpha}}\overline M_{R_1}%\max \{M_{R_1}, 2(L_{R_1}\|Du \|_\infty+H_0) \}\\ 
    % &\leq c_0M_{R_1}\frac{T^{1-\frac{1}{\alpha}}}{1-\frac{1}{\alpha}}|h|^\beta +c_0L_{R_1}\int_0^t(t-s)^{-\frac{1}{\alpha}} \sup_{x\in \R^d}|Du(s,x+h)-Du(s,x)| \, ds. \\
    + c_0L_{R_1}\int_0^t(t-s)^{-\frac{1}{\alpha}} \sup_{x\in \R^d}\frac{|Du(s,x+h)-Du(s,x)|}{|h|^\beta} \, ds.
\end{align*}
Consequently, by summing up the estimates,
%by the definition of the Duhamel map,
\begin{align}\label{eq:whdef}
w_{h}(t)&:=\sup_{x \in \R^d} \frac{|Du(t,x+h)-Du(t,x)|}{|h|^\beta} \\
&\leq \|Du_0\|_{\infty}2^{1-\beta}c_{2,1,d}\,%c_0^{\beta}
t^{-\frac{\beta}{\alpha}} +c_0 \overline M_{R_1} \frac{T^{1-\frac{1}{\alpha}}}{1-\frac{1}{\alpha}}  +c_0L_{R_1}\int_0^t(t-s)^{-\frac{1}{\alpha}} w_{h}(s) \, ds.\nonumber
\end{align}
We now want to use the generalized Grönwall inequality of Lemma~\ref{lem:generalized_gronwall}. We indeed have that $w_{h}(t)$ is locally integrable on $(0,T]$, since
\begin{align}\label{eq:w_h_loc_int}
    0 \leq w_{h}(t) = \sup_{x\in \R^d} \frac{|Du(t,x+h)-Du(t,x)|}{|h|^{{\beta}}} \leq %\frac{2}{|h|^{\alpha+\widehat{\beta}-1}} \|Du(t,\cdot)\|_{\infty}\leq 
    \frac{2}{|h|^{{\beta}}}\|Du\|_{\infty} <\infty,
\end{align}
%which, 
for each fixed $h$. %, is integrable on $[0,T]$ by definition of the space $X_A$. 
Hence, all the assumptions of Lemma~\ref{lem:generalized_gronwall} are satisfied and we get 
\begin{align}\label{eq:H_regularity_gronwall_calc}
    w_{h}(t)%=\sup_{x\in \R^d} \frac{|Du(t,x+h)-Du(t,x)|}{|h|^{\alpha+\widehat{\beta}-1}} 
     &\leq   \Big(2^{-\beta} c_{2,1,d}\, %\Big(\frac{c_0}{2}\Big)^{{\beta}}
     \|Du_0\|_{\infty} + C_{1} T^{1-\frac{1}{\alpha}}\Big)t^{-\frac{{\beta}}{\alpha}} +c_0 \overline M_{R_1}\frac{T^{1-\frac{1}{\alpha}}}{1-\frac{1}{\alpha}}+C_2T^{1-\frac{1}{\alpha}},
\end{align}
where $C_1 = C_1(T,\alpha,\beta,R_1,u_0)$ and $C_2=C_2(T,\alpha,R_1)$. Importantly, the constants do not depend on $h$. Thus, for each $t\in (0,T]$, taking the supremum over $|h|$ yields
\begin{align*}
    [Du(t)]_{\beta}  < \infty,
\end{align*}
so $Du(t,\cdot)\in C_b^\beta(\R^d)$. Multiplying \eqref{eq:H_regularity_gronwall_calc} by $t^{\frac{\beta}{\alpha}}$, we get $\sup_{t \in (0,T]}[t^{\frac{\beta}{\alpha}}Du(t,\cdot)]_{\beta}<\infty$.

The following computation completes the proof: %shows that $H(t,\cdot,Du(t,\cdot))\in C^\beta_b$: 
If $|x-x'|<1$, then
\begin{align*}
    &|H(t,x,Du(t,x))-H(t,x',Du(t,x'))| \\
    &\leq |H(t,x,Du(t,x))-H(t,x',Du(t,x))|+|H(t,x',Du(t,x))-H(t,x',Du(t,x'))| \\
    &\leq M_{R_1}|x-x'|^\beta+L_{R_1}|Du(t,x)-Du(t,x')| \\
    &\leq (M_{R_1}+L_{R_1}[Du(t,\cdot)]_{\beta})|x-x'|^\beta,
\end{align*}
whence we find that
$$[H(t,\cdot,Du(t,\cdot))]_\beta \leq L_{R_1}[Du(t,\cdot)]_\beta + \overline M_{R_1} < \infty,$$ 
where we have used assumptions \ref{assump:H2} and \ref{assump:H}, and that $Du(t,\cdot)\in C_b^\beta(\R^d)$.
\end{proof}

\begin{proof}[Proof of Theorem \ref{thm::full_schauder_regularity}] 
Differentiating both sides of $u=S[u]$ for a fixed $t\in(0,T]$, we get 
% \begin{align*}
%     Du(t,x) = P_t\big[Du_0\big](x)+\int_0^t DP_{t-s}\big[H(s,\cdot,Du(s,\cdot))\big](x)\, ds.
% \end{align*}
\begin{align}\label{eq:full_schauder_duhamel}
    Du(t,x) = P_t\big[Du_0\big](x)+\int_0^t DP_{t-s}\big[H(s,\cdot,Du(s,\cdot))\big](x)\, ds.
\end{align}
Let
$v(t,x):%&=D\int_0^t P_{t-s}\big[H(s,\cdot,Du(s,\cdot))\big](x)\, ds\\
    =\int_0^t s^{-\frac{\beta}{\alpha}} DP_{t-s}\big[s^{\frac{\beta}{\alpha}}H(s,\cdot,Du(s,\cdot))\big](x)\, ds.$

\medskip
\noindent\textbf{Case 1:} $\alpha + \beta \in (1,2)$ and $\{\alpha+\beta\}=\alpha+\beta-1$.
By \eqref{eq:full_schauder_duhamel}, Lemma \ref{lem:H_regularity}, Theorem \ref{lem::lem1} (ii), and Theorem \ref{prop:f_bound} (ii),
\begin{equation}
\begin{aligned}\label{eq::full_schauder_case_1_calc}
    [Du(t,\cdot)]_{\alpha+\beta-1} &\leq [P_t[Du_0]]_{\alpha+\beta-1}+[v(t,\cdot)]_{\alpha+\beta-1} \\
    &\leq \|Du_0\|_{\infty} 2^{2-(\alpha+\beta)}c_{2,1,d}\, t^{-\frac{\alpha+\beta-1}{\alpha}}\\ &\qquad \qquad  +t^{-\frac{\beta}{\alpha}} \overline{c}_{\frac{\beta}{\alpha},\alpha,\beta} \sup_{t\in(0,T]}\|t^{\frac{\beta}{\alpha}}H(t,\cdot,Du(t,\cdot)) \|_{C^\beta_b} < \infty. 
\end{aligned}
\end{equation}
 %where we used Lemma \ref{thm:holder_interpolation} with $\gamma = \alpha+\beta -1$ and a calculation as in \eqref{eq::existence_thm_ii_calc} for the first term.

\medskip
\noindent\textbf{Case 2:} $\alpha+\beta \in (2,3)$
%\footnote{For $\beta=1$ the $C^\beta_b$ norm should be replaced by $\|\cdot\|_\infty + [\, \cdot\, ]_\beta$.} 
and $\{\alpha+\beta\}=\alpha+\beta-2$. 
Since $(t-s)^{-\frac{2-\beta}{\alpha}}=(t-s)^{-1+\frac{\alpha+\beta-2}{\alpha}}$ is integrable at $s=t$, by
 Theorem \ref{lem::lem1} and the Dominated Convergence Theorem, 
\begin{align*}
    D^2u(t,x) % &= DP_t[Du_0](x)+D\int_0^t DP_{t-s}\big[H(s,\cdot,Du(s,\cdot))\big](x)\, ds \\
            &= DP_t[Du_0](x)+\int_0^t D^2P_{t-s}\big[H(s,\cdot,Du(s,\cdot))\big](x)\, ds, \ \ t\in(0,T],\ x\in \R^d.
\end{align*}
By Theorem \ref{lem::lem1} (i) (and $c_{1,0,d}=c_0$), Lemma \ref{lem:H_regularity}, and Theorem \ref{prop:f_bound} (i), 
%\ref{NDa}, and Young's inequality for convolutions (ERJ: not used here)
we get that
\begin{align}\label{eq::full_schauder_second_deriv_blowup}
\|D^2 u(t)\|_\infty \leq t^{-\frac{1}{\alpha}}c_0\|Du_0\|_\infty+t^{-\frac{2-\alpha}{\alpha}} c^{(2)}_{\frac{\beta}{\alpha},\alpha,\beta} \sup_{t \in (0,T]} \|t^{\frac{\beta}{\alpha}}H(t,\cdot, Du(t,\cdot)) \|_{C_b^{\beta}}<\infty.    
\end{align}
To get the blow-up rate for the optimal H\"older seminorm we do similar calculations as  in Case 1:
% by similar calculations as in Case 1, 
\begin{align*}
    [D^2u(t,\cdot)]_{\alpha+\beta-2} 
    \leq &\ \|Du_0\|_{\infty} 2^{3-(\alpha+\beta)} %c_{2,0,d} 
    c_{3,1,d} \, t^{-\frac{\alpha+\beta-1}{\alpha}}  \\ 
    &\qquad +t^{-\frac{\beta}{\alpha}} \overline{c}_{\frac{\beta}{\alpha},\alpha,\beta} \sup_{t\in(0,T]}\|t^{\frac{\beta}{\alpha}}H(t,\cdot,Du(t ,\cdot)) \|_{C^\beta_b} < \infty. \qedhere
\end{align*} 
\end{proof}

% ***
\subsection{Schauder regularity of solutions with unbounded gradients}\ \smallskip

\noindent In order to handle blow-up rates for higher order H\"older norms with non-Lipschitz initial conditions we use an $x$-H\"older condition on $H(t,x,p)$ with at most polynomial growth in $p$ of order $r$.\smallskip
  %, by a use a slightly stronger assumption on $H$.
% \begin{description}
% \item[({H$_x$}$'$)\label{assump:B3}] 
%  %If $\beta \in (0,1]$ then there
%  There exist $\beta\in [0,1]$ and $M>0$ such that for $t\in(0,T]$ and $x,x',p \in \R^d$ with $|x-x'|<1$,  
%        $$|H(t, x, p)-H(t, x', p)| \leq M(1+|p|^2)^{\frac{r}{2}}|x-x'|^\beta,$$
%       where $r$ is defined in %the same as in assumption 
%       \ref{assump:B2_new}.
% \end{description}
\begin{enumerate}
\myitem{$\mathbf{(H_x')}$}\label{assump:B3}
%\item[({H$_x$}$'$)\label{assump:B3}] 
 %If $\beta \in (0,1]$ then there
 There exist $\beta\in [0,1]$ and $M>0$ such that for $t\in(0,T]$ and $x,x',p \in \R^d$ with $|x-x'|<1$,  
       $$|H(t, x, p)-H(t, x', p)| \leq M(1+|p|^2)^{\frac{r}{2}}|x-x'|^\beta,$$
      where $r$ is defined in %the same as in assumption 
      \ref{assump:B2_new}.
\end{enumerate}

\begin{rem}
    This assumption is consistent with \ref{assump:B2_new}, and when $\beta=0$ it follows from \ref{assump:B2_new}. An example satisfying both \ref{assump:B2_new} and \ref{assump:B3} is $H(t,x,p)=b(t,x)|p|^r+f(t,x)$ for $b,f\in C([0,T];C_b^\beta(\R^d))$.
\end{rem}

 We now give the $C^{\alpha+\beta}_b$ Schauder regularity result for solutions with spatial gradients that blow up as $t\to 0$.

\begin{thm}[Schauder regularity (II)]\label{thm:full_schauder_regularity_case_b}
    %Fix $T>0$. 
    Assume \ref{assump:B2_new}, \ref{assump:B1}, \ref{assump:B3} for some $\beta \in [0,1]$, \ref{NDa} with $\alpha\in (1,2]$,
    and
 %$m=\lceil \alpha+\beta \rceil$,
 $\alpha+\beta \notin \mathbb{N}$, and $u\in C((0,T],C^1_b(\R^d))$ is a mild solution of \eqref{eq:hjb}. Then the following statements hold:\medskip

\noindent (i) \ $u(t)\in C_b^{\alpha +\beta}(\R^d)$ for each $t\in (0,T]$. 

\medskip\noindent (ii) \ $[D^{ \lfloor\alpha+\beta\rfloor}
% \lceil \alpha+\beta \rceil -1} 
u(t,\cdot)]_{\{\alpha+\beta\}}
%{\alpha+\beta-( \lceil \alpha+\beta \rceil -1)} 
= O(t^{-\frac{\alpha+\beta-\delta}{\alpha} })$ as $t \rightarrow 0$, \qquad where $\{\alpha+\beta\}=\alpha+\beta-\lfloor\alpha+\beta\rfloor$.

\medskip\noindent (iii) \ If $\alpha+\beta> 2$, then $\| D^2u(t,\cdot)\|_\infty=O(t^{-\frac{2-\delta}{\alpha}})$ as $t\rightarrow 0$. %and for $\alpha+\beta \geq 2$ we have $\| D^2u(t,\cdot)\|_\infty=O(t^{-\frac{2-\delta}{\alpha}  })$ as $t\rightarrow 0$. 
\end{thm}

\begin{rem}
    Since solutions are assumed to be %become 
    Lipschitz for $t>0$, we can use Theorem~\ref{thm::full_schauder_regularity} to see that part (i) holds even if we replace  \ref{assump:B3} by the weaker assumption \ref{assump:H}.
    %To only show $u(t)\in C^{\alpha+\beta}_b(\R^d)$ for $t\in (0,T]$, it suffices to assume \ref{assump:H} instead of \ref{assump:B3}, because the solutions immediately become Lipschitz, so we can use Theorem~\ref{thm::full_schauder_regularity}.
\end{rem}

To prove Theorem \ref{thm:full_schauder_regularity_case_b} $(ii)$ and $(iii)$ we first show an intermediate result on the $C^{1+\beta}_b$ blow-up rate as $t\rightarrow 0$ of $u$. 
 
\begin{lem}\label{lem:H_regularity_case_B}
    Assume \ref{NDa} with $\alpha\in (1,2]$,
    %and $m=2$, 
     \ref{assump:B3} with %for some
    $\beta\in [0,1]$, \ref{assump:B2_new} with $r\geq1$,  \ref{assump:B1} with $\delta\in[0,1)$, and $u\in C((0,T],C^1_b(\R^d))$ is a mild solution of \eqref{eq:hjb}. Then 
    %$Du(t,\cdot), \ H(t, \cdot, %Du(t,\cdot)) \in %C^\beta_b(\R^d)$ for each $t\in (0,T]$, and %if $\beta>0$, 
    \begin{align}\sup_{t \in (0,T]}[t^{\frac{1+\beta-\delta}{\alpha}}Du(t,\cdot)]_{\beta}<\infty\qquad\text{and}\qquad \sup_{t \in (0,T]}[t^{\frac{r(1-\delta)+\beta}{\alpha}}H(t,\cdot, Du(t,\cdot))]_{\beta}<\infty.\label{eq:Hbetareg}\end{align}
\end{lem}

The proof is given at the end of the section.

\begin{proof}[Proof of Theorem~\ref{thm:full_schauder_regularity_case_b}]
$(i)$ \ $C^{\alpha+\beta}$ regularity for $t>0$ follows from Theorem~\ref{thm::full_schauder_regularity} since $u(t)$ by assumption is Lipschitz for $t>0$. It remains to show the blow-up rates in $(ii)$ and $(iii)$.
\smallskip

\noindent\textbf{Case 1:} $\alpha+\beta \in (1,2]$ and $\{\alpha+\beta\}=\alpha+\beta-1$. \ Let $0<\epsilon<T$, $t\in (0,T-\epsilon]$. From \eqref{eq::full_schauder_case_1_calc} we have that
\begin{align*}
    [Du(t+\epsilon,\cdot)]_{\alpha+\beta-1}   &\leq \|Du(\epsilon,\cdot)\|_{\infty} \,2^{2-(\alpha+\beta)}c_{2,1,d}\, t^{-\frac{\alpha+\beta-1}{\alpha}} \\ &\qquad +t^{-\frac{\beta}{\alpha}} \overline{c}_{\frac{\beta}{\alpha},\alpha,\beta} \sup_{s\in(0,T-\epsilon]}\|s^{\frac{\beta}{\alpha}}H(s+\epsilon,\cdot,Du(s+\epsilon,\cdot)) \|_{C^\beta_b} \\
    % &\leq \|Du(\epsilon)\|_{\infty}2\Big(\frac{c_0}{2} \Big)^{\alpha+\beta-1} t^{-\frac{\alpha+\beta-1}{\alpha}} \\ &\qquad +t^{-\frac{\beta}{\alpha}} \epsilon^{-\frac{r(1-\delta)}{\alpha}}\overline{c}^{(1)}_{\frac{\beta}{\alpha},\alpha,\beta} \sup_{t\in(0,T-\epsilon]}\|(t+\epsilon)^{\frac{r(1-\delta)}{\alpha}}t^{\frac{\beta}{\alpha}}H(t+\epsilon,\cdot,Du(t+\epsilon,\cdot)) \|_{C^\beta_b}. \\
    &\leq \|Du(\epsilon,\cdot)\|_{\infty}\, 2^{2-(\alpha+\beta)}c_{2,1,d}\, t^{-\frac{\alpha+\beta-1}{\alpha}} \\ &\qquad +t^{-\frac{\beta}{\alpha}} \epsilon^{-\frac{r(1-\delta)}{\alpha}}\overline{c}%^{(1)}
    _{\frac{\beta}{\alpha},\alpha,\beta} \sup_{s\in(0,T-\epsilon]}\|(s+\epsilon)^{\frac{r(1-\delta)+\beta}{\alpha}}H(s+\epsilon,\cdot,Du(s+\epsilon,\cdot)) \|_{C^\beta_b}. 
\end{align*}
Evaluating at $t=\epsilon$ and noting that $\alpha -\delta -r(1-\delta)\geq 0$ by \ref{assump:B1},
\begin{align}
[Du(2\epsilon,\cdot)]_{\alpha+\beta-1} &\leq \epsilon^{-\frac{\alpha+\beta-\delta}{\alpha}} \sup_{s\in(0,T]}\|s^{\frac{1-\delta}{\alpha}} Du(s,\cdot)\|_{\infty}2^{2-(\alpha+\beta)}c_{2,1,d} \nonumber \\ 
& \qquad +\epsilon^{-\frac{r(1-\delta)+\beta}{\alpha}} \overline{c}%^{(1)}
_{\frac{\beta}{\alpha},\alpha,\beta} \sup_{s \in(0,T]}\|s^{\frac{r(1-\delta)+\beta}{\alpha}}H(s,\cdot,Du(s,\cdot)) \|_{C^\beta_b}\nonumber\\
    &\leq \epsilon^{-\frac{\alpha+\beta-\delta}{\alpha}}\Big(\sup_{s\in(0,T]}\|s^{\frac{1-\delta}{\alpha}} Du(s,\cdot)\|_{\infty} 2^{2-(\alpha+\beta)}c_{2,1,d}\label{stestim}\\ &\qquad +T^{\frac{\alpha-\delta-r(1-\delta)}{\alpha}} \overline{c}%^{(1)}
    _{\frac{\beta}{\alpha},\alpha,\beta} \sup_{s \in(0,T]}\|s^{\frac{r(1-\delta)+\beta}{\alpha}}H(s,\cdot,Du(s,\cdot)) \|_{C^\beta_b} \Big).\nonumber
\end{align}
By Theorem \ref{thm::case_b_hjb_sol_existence} and Lemma \ref{lem:H_regularity_case_B} the expression in the brackets is finite.
\smallskip

\noindent\textbf{Case 2:} $\alpha+\beta \in (2,3)$ and $\{\alpha+\beta\}=\alpha+\beta-2$.  The Hölder estimate follows in a similar manner and yields the same blow-up rate as Case 1. The  blow-up rate for $\|D^2u(t,\cdot)\|_\infty$ can be shown in a similar way from \eqref{eq::full_schauder_second_deriv_blowup}.
\end{proof}

We now prove Lemma \ref{lem:H_regularity_case_B}. In the proof we have to distinguish between two cases, $\delta\geq \beta$ and $\delta<\beta$. The second case relies on a bootstrapping argument that uses the result from the first case and reuses arguments from the proofs of both the first case and Theorem \ref{thm:full_schauder_regularity_case_b} above.

\begin{proof}[Proof of Lemma \ref{lem:H_regularity_case_B}]
The case $\beta = 0$ follow from Theorem \ref{thm::case_b_hjb_sol_existence} and assumption \ref{assump:B2_new}, so assume $\beta>0$.\smallskip

\noindent\textbf{Case 1:} $\delta\geq\beta$.  \ 
By Theorem~\ref{lem::lem1},
\begin{align}\label{eq:Du0beta}
    [ DP_t u_0 ]_{\beta} &\leq 
   2^{1-\beta} c_{2,\delta,d} [u_0]_\delta t^{-\frac{1+\beta-\delta}{\alpha}}.
\end{align}
 Let $g_{t-s}(x):=\int_{\R^d} Dp_{t-s}(y) H(s,x-y,Du(s,x-y)) \, dy$, and define $$G:=\sup_{t\in (0,T]}\big(t^{2\frac{1-\delta}{\alpha}}+2\|t^{\frac{1-\delta}{\alpha}}Du(t,\cdot) \|_\infty^2\big),$$ $\widetilde L:= G^{\frac{r-1}{2}}L$, and $K:=\max\big\{G^{\frac{r}{2}}M, H_0 T^{\frac{r(1-\delta)}{\alpha}}+LG^{\frac{r}{2}} \big\}$.
By %Assumptions 
\ref{assump:B2_new}, \ref{assump:B3}, and the bound $\|Du(t,\cdot)\|_\infty=O(t^{-\frac{1-\delta}\alpha})$ from Theorem \ref{thm::case_b_hjb_sol_existence} (for $t$  small) and the assumption $u\in C((0,T],C^1_b(\R^d))$ (for $t$ not small),
\begin{align*}
    &|g_{t-s}(x+h)-g_{t-s}(x)|
    % &\leq  c_0(t-s)^{-\frac{1}{\alpha}}\bigg(\widetilde L s^{-\frac{(r-1)(1-\delta)}{\alpha}} \sup_{x\in\R^d}|Du(s,x+h)-Du(s,x)| \\
    % &\qquad + s^{-\frac{r(1-\delta)}{\alpha}}\cdot\begin{cases}
    % G^{\frac{r}{2}}M|h|^\beta, &|h|<1,\\
    % (LG^{\frac{r}{2}}+H_0 T^{\frac{r(1-\delta)}{\alpha}})|h|^\beta, &|h|\geq 1
    % \end{cases}    \bigg).
        \leq  c_0(t-s)^{-\frac{1}{\alpha}}\Big(\widetilde L s^{-\frac{(r-1)(1-\delta)}{\alpha}} \sup_{x\in\R^d}|Du(s,x+h)-Du(s,x)| + K s^{-\frac{r(1-\delta)}{\alpha}} |h|^\beta  \Big).
\end{align*}
%$K:=\max\{G^{\frac{r}{2}}M, \big(LG^{\frac{r}{2}}+H_0 T^{\frac{r(1-\delta)}{\alpha}}\big) \}$ and 
 Integrating in time, with  $B:=B(1-\tfrac{r(1-\delta)}{\alpha},1-\tfrac1\alpha)$ (see Remark \ref{rem:w-thm} (c)),  yields
\begin{align}
    \frac{1}{|h|^\beta}\int_0^t &|g_{t-s}(x+h)-g_{t-s}(x)| \; ds \leq c_0 KB t^{1-\frac{1}{\alpha}-\frac{r(1-\delta)}{\alpha}} \label{eq:gbeta}\\
    &+ c_0 \widetilde L \int_0^t(t-s)^{-\frac{1}{\alpha}}s^{-\frac{(r-1)(1-\delta)}{\alpha}} \sup_{x\in \R^d}\frac{|Du(s,x+h)-Du(s,x)|}{|h|^\beta} \, ds.\nonumber
\end{align}
 Since $u$ is a mild solution of \eqref{eq:hjb}, it follows from estimates \eqref{eq:Du0beta} and \eqref{eq:gbeta} that 
\begin{align*}
&w_h(t)  :=\sup_{x \in \R^d} \frac{|Du(t,x+h)-Du(t,x)|}{|h|^\beta}  \\
& \leq t^{-\frac{1+\beta-\delta}{\alpha}} \big(
%1+
2^{1-\beta} c_{2,\delta,d} [u_0]_\delta
%^{1+\frac{\beta}{\delta}}
+ c_0BK T^{1+\frac{\beta-\delta}{\alpha}-\frac{r(1-\delta)}{\alpha}} \big) 
%\\ &\qquad 
+c_0\widetilde L\int_0^t(t-s)^{-\frac{1}{\alpha}}s^{-\frac{(r-1)(1-\delta)}{\alpha}}  w_h(s) \, ds.
 \end{align*}
% We also have that $w_h(t)$ is locally integrable on $(0,T]$:
% \begin{align*}
%     0 \leq w_h(t) = \sup_{x\in \R^d} \frac{|Du(t,x+h)-Du(t,x)|}{|h|^{{\beta}}} \leq %\frac{2}{|h|^{\alpha+\widehat{\beta}-1}} \|Du(t,\cdot)\|_{\infty}\leq 
%     \frac{2\, t^{-\frac{1-\delta}{\alpha}}}{|h|^{{\beta}}}\sup_{t\in(0,T]}\|t^{\frac{1-\delta}{\alpha}}Du(t,\cdot)\|_{\infty} <\infty,
% \end{align*} 
Then since \ref{assump:B1} and
$\delta\geq \beta$ hold, it follows from the Grönwall-type inequality in Lemma \ref{lem:generalized_gronwall_2} with $\overline\alpha=1-\frac{1+\beta-\delta}{\alpha}$, $\overline\beta=1-\frac1\alpha$, $\overline\gamma=1-\frac{(r-1)(1-\delta)}{\alpha}$  
%\footnote{ Here $\overline\alpha=1-\frac{1+\beta-\delta}{\alpha}$, $\overline\beta=1-\frac1\alpha$, $\overline\gamma=1-\frac{(r-1)(1-\delta)}{\alpha}$, and hence by \ref{assump:B1} and $\delta\geq \beta$, $\overline\nu=\overline\beta+\overline\gamma-1>0$ and $\overline\delta=\overline\alpha+\overline\gamma-1= \overline\nu +\frac{\delta-\beta}{\alpha}>0$.} 
that
\begin{align}\label{eq:H_regularity_gronwall_calc_case_B}
    w_h(t) \leq \overline{C} \, t^{-\frac{1+\beta-\delta}{\alpha}}, 
\end{align}
for some constant $\overline{C}>0$.
%  is a finite constant given by
% \begin{align*}
%      & \overline{C} = \big(
%      %1+
%      2^{1-\beta}c_0C_{\delta,d} [u_0]_\delta 
%      %^{1+\frac{\beta}{\delta}}
%      + c_0BK T^{1+\frac{\beta-\delta}{\alpha}-\frac{r(1-\delta)}{\alpha}} \big) \cdot \sum_{m=0}^\infty C'_m(c_0\widetilde L \Gamma(1-\tfrac{1}{\alpha}) T^{1-\frac{1}{\alpha}-\frac{(r-1)(1-\delta)}{\alpha}} )^m, %\label{eq:series}
%      %\\
%     % &\qquad \mathrm{for}\quad C'_0=1 \quad \text{and}\quad \tfrac{C_{m+1}'}{C_m'} = \tfrac{\Gamma((m+1)(1-\frac{1}{\alpha}-\frac{(r-1)(1-\delta)}{\alpha})+\frac{\delta-\beta}{\alpha})}{\Gamma((m+1)(1-\frac{1}{\alpha}-\frac{(r-1)(1-\delta)}{\alpha})+\frac{\delta-\beta}{\alpha}+1-\frac{1}{\alpha})}.\nonumber
% \end{align*}
% for $C'_0=1$ and $\tfrac{C_{m+1}'}{C_m'} = \tfrac{\Gamma((m+1)(1-\frac{1}{\alpha}-\frac{(r-1)(1-\delta)}{\alpha})+\frac{\delta-\beta}{\alpha})}{\Gamma((m+1)(1-\frac{1}{\alpha}-\frac{(r-1)(1-\delta)}{\alpha})+\frac{\delta-\beta}{\alpha}+1-\frac{1}{\alpha})}$.
% Taking the supremum over $|h|$, we find that  $[Du(t)]_{\beta} < \infty$ and hence $Du(t,\cdot)\in C_b^\beta(\R^d)$. 

Multiplying \eqref{eq:H_regularity_gronwall_calc_case_B} by $t^{\frac{1+\beta-\delta}{\alpha}}$ and taking the supremum over $|h|$, we see that $\sup_{t \in (0,T]}[t^{\frac{1+\beta-\delta}{\alpha}}Du(t,\cdot)]_{\beta}<\infty$.
 The proof is complete by noting that by \ref{assump:B2_new}, \ref{assump:B3}, and the bound $\|Du(t,\cdot)\|_\infty=O(t^{-\frac{1-\delta}\alpha})$ from Theorem \ref{thm::case_b_hjb_sol_existence} (for $t$  small) and the assumption $u\in C((0,T],C^1_b(\R^d))$ (for $t$ not small), % since $u\in C((0,T],C^1_b(\R^d))$,
%Lemma \ref{lem:H_regularity} (for $t$ not small),
\begin{align}\label{eq:H_Du_triangle_ineq}
\begin{aligned}
    |H(t,x,&Du(t,x))-H(t,x',Du(t,x'))| \\
   % &\leq |H(t,x,Du(t,x))-H(t,x',Du(t,x))|+|H(t,x',Du(t,x))-H(t,x',Du(t,x'))| \\
    % &\leq M(1+|Du(t,x)|^2)^{\frac{r}{2}}|x-x'|^\beta+L(1+|Du(t,x)|^2+|Du(t,x')|^2)^{\frac{r-1}{2}}|Du(t,x)-Du(t,x')| \\
    &\leq Kt^{-\frac{r(1-\delta)}{\alpha}}|x-x'|^\beta+\widetilde Lt^{-\frac{(r-1)(1-\delta)}{\alpha}}|Du(t,x)-Du(t,x')| \\
     %&\leq Kt^{-\frac{r(1-\delta)}{\alpha}}|x-x'|^\beta+\widetilde Lt^{-\frac{(r-1)(1-\delta)}{\alpha}-\frac{1+\beta-\delta}{\alpha}}|x-x'|^\beta \sup_{t\in (0,T]}[t^{\frac{1+\beta-\delta}{\alpha}}Du(t,\cdot)]_\beta \\
     %&\leq Kt^{-\frac{r(1-\delta)}{\alpha}}|x-x'|^\beta+\widetilde Lt^{-\frac{r(1-\delta)+\beta}{\alpha}}|x-x'|^\beta \sup_{t\in (0,T]}[t^{\frac{1+\beta-\delta}{\alpha}}Du(t,\cdot)]_\beta \\
    &\leq t^{-\frac{r(1-\delta)+\beta}{\alpha}}\Big(KT^{\frac{\beta}{\alpha}}+\widetilde L\sup_{t\in (0,T]}[t^{\frac{1+\beta-\delta}{\alpha}}Du(t,\cdot)]_{\beta}\Big)|x-x'|^\beta.
\end{aligned}
\end{align}

\noindent{\textbf{Case 2:}} $0\leq\delta<\beta$. \ 
Here we use a combination of Gr\"onwall and bootstrapping arguments. \smallskip

\noindent {\bf (i)} \  Initial $C^{1+\beta_0}$ estimate for some small $\beta_0\in(0,\beta]$. We begin by proving that estimates \eqref{eq:Hbetareg} hold when $\beta$ is replaced by $\beta_0$. If $\delta>0$ we take $\beta_0=\delta$, and by Case 1 (with $\beta=\beta_0$) it follows that  
%and proceed as in Case 1 (with $\beta=\delta$), obtaining 
$$\sup_{t\in(0,T]} [t^{\frac{1}{\alpha}}Du(t,\cdot)]_{\beta_0} < \infty.$$ 
%(cf. \eqref{eq:Du0beta} with $\beta=\delta$). 
If $0=\delta<\beta$, then $r\in[1,\alpha)$ by \ref{assump:B1} and we take $\beta_0\in(0,\alpha-r)$. 
% %By Theorem \ref{lem::lem1},
%     $[DP_tu_0]_{\beta_0} \leq c_0C_{\beta_0,d}\|u_0\|_\infty t^{-\frac{1+\beta_0}{\alpha}}$.
%Replacing \eqref{eq:Du0beta} by this estimate, we 
We redo the proof of Case 1 with $\beta=\beta_0$ and $\delta=0$. By our choice of $\beta_0$, we can again apply Gr\"onwall (Lemma \ref{lem:generalized_gronwall_2}) and conclude that $$\sup_{t\in(0,T]} [t^{\frac{1+\beta_0}{\alpha}}Du(t,\cdot)]_{\beta_0} < \infty.$$
%
% If $0=\delta<\beta$, we estimate as follows
% \begin{align*}
%     [DP_tu_0]_{\gamma_0} \leq c_0C_{\gamma_0,d}\|u_0\|_\infty t^{-\frac{1+\gamma_0}{\alpha}}.
% \end{align*}
% Then, we can carry out the Gr\"onwall argument, provided that $\gamma_0\in(0,\alpha-r)$. We can always do that because of the relationship between $r$ and $\delta$ from \ref{assump:B1}.
%
 In either case, since $H(t,\cdot,p)$ is $\gamma$-Hölder regular for any $\gamma\in(0,\beta]$ by \ref{assump:B3}, estimate \eqref{eq:H_Du_triangle_ineq} holds with $\beta_0$ replacing $\beta$, and hence
\begin{align}\label{eq::blow_up_H_bootstrap}
\sup_{t\in(0,T]}[t^{\frac{r(1-\delta)+\beta_0}{\alpha}}H(t,\cdot,Du(t,\cdot))]_{\beta_0} \leq KT^\frac{\beta_0}{\alpha}+\widetilde{L}\sup_{t\in(0,T]} [t^{\frac{1+\beta_0-\delta}{\alpha}}Du(t,\cdot)]_{\beta_0}<\infty.    
\end{align}
We conclude that estimates \eqref{eq:Hbetareg} hold when $\beta$ is replaced by $\beta_0$.\smallskip

\noindent {\bf (ii)} \ 
Bootstrapping. In view of (i), we can now redo the proof of Theorem \ref{thm:full_schauder_regularity_case_b} with $\beta_0$ replacing $\beta$ to obtain $C^{\alpha+\beta_0}$ regularity of $u$ and corresponding 
%version of the 
blow-up estimate \eqref{stestim}: For $\beta_1=\alpha+\beta_0-1$,
\begin{align}\label{eq::blow_up_Du_bootstrap}
    \begin{aligned}
    [Du(t, \cdot)]_{\beta_1} &\leq t^{-\frac{1+\beta_1-\delta}{\alpha}} \Big(\sup_{s\in(0,T]}\|s^{\frac{1-\delta}{\alpha}} Du(s,\cdot)\|_{\infty} 2^{1-\beta_1}c_{2,1,d} \\ &\qquad +T^{\frac{\alpha-\delta-r(1-\delta)}{\alpha}} \overline{c}%^{(1)}
    _{\frac{\beta_0}{\alpha},\alpha,\beta_0} \sup_{s \in(0,T]}\|s^{\frac{r(1-\delta)+\beta_0}{\alpha}}H(s,\cdot,Du(s,\cdot)) \|_{C^{\beta_0}_b} \Big),
\end{aligned}
\end{align}
where the norms are finite by Lemma \ref{thm::case_b_hjb_sol_existence} and part (i). We can now update \eqref{eq::blow_up_H_bootstrap}, replacing $\beta_0$ by $\beta_1$, and conclude that estimates \eqref{eq:Hbetareg} hold when $\beta$ is replaced by $\beta_1$. Note the gain, $\beta_1-\beta_0=\alpha-1>0$.\smallskip

\noindent {\bf (iii)} \ Iteration. For $k=1,2,\dots,N-1$, we repeat part (ii) replacing $\beta_k$ with $\beta_{k+1}$, where $\beta_{k}=k(\alpha-1)+\beta_0$. The result is $C^{1+\beta_{k+1}}$ regularity and that estimates \eqref{eq:Hbetareg} hold when $\beta$ is replaced by $\beta_{k+1}$.
% \begin{align*}
%     [Du(t,\cdot)]_{\gamma_{k}} \lesssim t^{-\frac{1+\gamma_{k}-\delta}{\alpha}}\qquad \mathrm{and}\qquad [H(t,\cdot,Du(t,\cdot))]_{\gamma_{k}} \lesssim t^{-\frac{r(1-\delta)+\gamma_k}{\alpha}}.
% \end{align*}
% Now we repeat this procedure multiple times by returning to \eqref{eq::blow_up_H_bootstrap} and \eqref{eq::blow_up_Du_bootstrap}, each time replacing $\gamma_{k}=k(\alpha-1)+\gamma_0$ with $\gamma_{k+1} =\alpha+\gamma_{k}-1=(k+1)(\alpha-1)+\gamma_0$, resulting in the bounds \begin{align*}
%     [Du(t,\cdot)]_{\gamma_{k}} \lesssim t^{-\frac{1+\gamma_{k}-\delta}{\alpha}}\qquad \mathrm{and}\qquad [H(t,\cdot,Du(t,\cdot))]_{\gamma_{k}} \lesssim t^{-\frac{r(1-\delta)+\gamma_k}{\alpha}}.
% \end{align*} 
Let $N$ be smallest integer such that $(N+1)(\alpha-1)+\beta_0>\beta$. Once we reach $k=N-1$, we do a final iteration with $\beta_{N+1} =\beta$ to achieve $C^{1+\beta}$ regularity and conclude the proof of \eqref{eq:Hbetareg}.
\end{proof}

\section{Classical solutions, long-time existence, and optimal regularity in time.}
%for the viscous HJ equation}
\label{sec::long_time_existence}

\subsection{Classical solutions of the viscous HJ equation}\label{sub_sec::classical_sol}

We use the regularity results of Section \ref{sec::full_schauder_regularity} to prove that mild solutions of \eqref{eq:hjb} are classical solutions. We first state a result for the linear case.
%We will now upgrade the mild solution to a classical solution using the higher spatial regularity achieved above.
\begin{lem}\label{lem:heateqn}
    Assume \ref{NDa}
    %with $m=1$
    and $\phi\in L^\infty(\R^d)$. Then $P_t\phi\in C^\infty_b(\R^d)$ for $t>0$ and
    \begin{align*}
        \partial_t P_t \phi(x) = \mathcal{L}P_t\phi(x),\qquad t>0,\ x\in \R^d.
    \end{align*}
\end{lem} 
\begin{proof}
First note that since $P_t\phi$ is an $L^1-L^\infty$ convolution, it is uniformly continuous. By \ref{NDa} and the dominated convergence theorem we further get that $P_t\phi\in C^1_b(\R^d)$ and $DP_{t}\phi(x) = P_{t/2}DP_{t/2}\phi(x)$. By a bootstrap argument we get that $P_t\phi\in C^{\infty}_b(\R^d)$. By \cite[Example~4.8.21]{MR1873235}, the convolution operators $P_t$ form a $C_b$-semigroup and by \cite[Example~4.8.26]{MR1873235}, $C^2_b(\R^d)$ belongs to the domain of the $C_b$-generator of $P_t$, which we denote by $A$ (see \cite[Definition~4.8.14]{MR1873235}). Therefore, by \cite[Remark~4.8.15]{MR1873235} we have
    \begin{align*}
        \partial_t P_t\phi = \partial_t P_{t/2}P_{t/2}\phi = AP_{t/2}P_{t/2}\phi = AP_t \phi.
    \end{align*}
    It remains to prove that for $C^2_b(\R^d)$ functions the operator $A$ agrees pointwise with $\mathcal{L}$. To this end, let $\tau$ be a smooth cut-off function vanishing outside $B(x,1)$. Then,
    \begin{align*}
        AP_t\phi(x) = A[(P_t\phi)\tau](x) + A[(P_t\phi)(1-\tau)](x) = \mathcal{L}[(P_t\phi)\tau](x) + \mathcal{L}[(P_t\phi)(1-\tau)](x),
    \end{align*}
    where the last equality follows from \cite[Theorem~31.5]{MR3185174} for the first term and from \cite[Corollary~8.9]{MR3185174} and the definition of $A$ for the second term.
\end{proof}

When the maximal regularity of the mild solution is $\alpha+\beta < 2$, we naturally need %an additional assumption on the 
%a restriction on the maximal order of $\mL$ to show that solutions are classical:
to restrict the maximal order of the operator $\mL$ to be less than $\alpha+\beta$ for our problems to have classical solutions. We give such a condition in terms of the Levy triplet:\medskip
\begin{enumerate}
\myitem{$\mathbf{(L2)}$}\label{assump:L2} $\mL$ is of the form \eqref{eq:L} with $A=0$ and
%$\sigma\in (\alpha,2)$ 
 $\int_{\R^d} (1\wedge |z|^\sigma)\, d\mu(z)< \infty$ for some $\sigma\in (\alpha,\alpha+\beta)$. % ERJ: drop footnote, this case is not relevant \footnote{We could also consider $\sigma<\alpha$ in this assumption, but it is unlikely that it would hold together with \ref{NDa}, cf. the fractional Laplacian of order $\alpha$.}
\end{enumerate}\smallskip
\begin{rem}\label{rem:L}  (a) Assumption \ref{NDa} implies that the order of $\mL$ is at least $\alpha$, while Assumption \ref{assump:L2} implies that $\mL$ does not contain any terms of order $\geq \sigma(<\alpha+\beta)$. E.g. if $\beta=\frac12$, then $\mL=-(-\partial^2_{x_1})^{\frac14}-(-\Delta_{\R^d})^{\frac{1.1}2}-(-\partial^2_{x_2})^{\frac34}$ satisfies \ref{NDa} with $\alpha=1.1$ and \ref{assump:L2} for any $\sigma\in(\frac32,1.1+\frac12)$.\smallskip

\noindent (b) %Assumption \ref{assump:L2} does not follow from \ref{NDa}. 
%If $A$ is degenerate and nontrivial, then 
Let $\mL u = \Div(ADu) -(-\Delta)^{\frac{\tilde\alpha}2}u$ with  degenerate nontrivial $A$ and $\tilde\alpha\in (1,2)$. Then \ref{NDa} holds with $\alpha=\tilde\alpha<2$, but because of the second derivatives, $\alpha+\beta>2$ is needed for solutions of \eqref{eq:hjb} to be $C^2$ and classical. We use \ref{assump:L2} to exclude this case.
\smallskip

\noindent (c) Any L\'evy operator $\mL$ satisfies $\int_{\R^d}(1\wedge |z|^2)\, d\mu(x)<\infty$ and hence $\|\mL u\|_\infty \leq C\|u\|_{C^2_b(\R^d)}$. Assumption \ref{assump:L2} implies $\|\mL u\|_\infty \leq C\|u\|_{C^{\sigma+}_b(\R^d)}$ for $\sigma<2$.
\end{rem}
\begin{thm}\label{thm:u_is_classical_sol}
Assume the assumptions of Theorem \ref{thm::full_schauder_regularity} hold for $\alpha\in (1,2]$, $\beta \in (0, 1]$, $\alpha + \beta \notin \mathbb{N}$, and 
either $\alpha+\beta > 2$ or \ref{assump:L2} holds.
%for some $\sigma\in (\alpha,\alpha+\beta)$. 
Then the mild solution $u$ of \eqref{eq:hjb} is a classical solution, meaning that it satisfies \eqref{eq:hjb} 
% satisfying  \eqref{eq:hjb}  
pointwise and $u(t,x)\to u_0(x)$ as $t\to 0$ for $x\in \R^d$.   
\end{thm}
%***
\begin{rem}
    The result also holds if $u_0\in C^\delta_b(\R^d)$ under the assumptions of Theorem~\ref{thm:full_schauder_regularity_case_b}. Since the solution immediately becomes $\alpha+\beta$-H\"older, the proof is almost identical. The only difference is that to show the convergence to $u_0$ we use the fact that the bound on $\|H(s,\cdot,Du(s,\cdot))\|_\infty$ is integrable in $s$ (in Case (I) it is bounded).
\end{rem}
The following proof mirrors that of \cite[Lemma 5]{MR2121115}.

\begin{proof}
      Consider the Duhamel representation of $u$:
    \begin{align}
    u(t,x) = P_t\big[u_0\big](x)+\int_0^t P_{t-s}\big[H(s,\cdot,Du(s,\cdot))\big](x)\, ds.
    \end{align}
By Theorem~\ref{thm::full_schauder_regularity} $Du$ is uniformly bounded, so $H(s,x,Du(s,x))$ is also uniformly bounded. From this we immediately get that $u(t,x)\to u_0(x)$ as $t\to 0$ for all $x\in \R^d$, so it suffices to verify that $u$ satisfies \eqref{eq:hjb}.

By Lemma~\ref{lem::duhamel_map_time_translation} and 
Theorem \ref{thm::hjb_sol_existence}
%\ref{lem::duhamel_map_boundedness} 
we can assume that $u\in C([0,T],C^1_b(\R^d))$. By Lemma~\ref{lem:heateqn},
    \begin{align*}
        \partial_t P_t\big[u_0\big](x) = \mathcal{L}P_t\big[u_0\big](x).
    \end{align*}
To simplify the notation, let $f(t,x)=H(t,x,Du(t,x))$. It remains to show that
\begin{align}\label{eq:classical_sol_remains_to_show}
    \partial_t \int_0^t P_{t-s}\big[f(s,\cdot)\big](x) \, ds = \mathcal{L}\int_0^t P_{t-s}\big[f(s,\cdot)\big](x) \, ds + f(t,x),
\end{align}
which formally follows from the Leibniz integral rule, but requires some work due to the singularity of the heat kernel at time zero. It suffices to show \eqref{eq:classical_sol_remains_to_show} for $t \in [\delta_0, T]$, with $\delta_0>0$ fixed. Let $\delta \in (0,\delta_0/2)$, 
%and define $\Phi(t,x)=\int_0^t P_{t-s}\big[f(s,\cdot)\big](x) \, ds$,  
\begin{align*}
 &\varphi(t,s,x) = P_{t-s}\big[f(s,\cdot)\big](x), \qquad  (t,s,x) \in [\delta_0, T]\times(0,T)\times \R^d, \ s< t, \\
&\Phi_\delta(t,x)= \int_0^{t-\delta} \varphi(t,s,x) \, ds, \quad \Phi(t,x)=\int_0^t \varphi(t,s,x) \, ds,  \qquad (t,x)\in [\delta_0, T] \times \R^d.
\end{align*}
We will first show that $\partial_t \Phi_\delta = \mathcal{L}\Phi_\delta$. Then, we prove that $\Phi_\delta(\cdot,x)$ converges uniformly to $\Phi(\cdot,x) + \varphi(\cdot,\cdot,x)$ on $[\delta_0,T]$ and that on the same interval $\partial_t\Phi_\delta(\cdot,x)$ converge uniformly to $\mathcal{L} \Phi(\cdot,x) + f(\cdot,x)$. This implies that $\partial_t\Phi$ exists and is equal to $\mathcal{L} \Phi + f$.

Let $t_1,t_2\in(0,T]$, $R>0$. By continuity of $H$ in \ref{assump:H2} there is a modulus of continuity $\omega_R$ such that,
\begin{align*}
    |H(t_1,x,p) - H(t_2,x,p)| \leq \omega_R(t_1-t_2),\qquad x,p\in B(0,R),
\end{align*}
and %by Lemma \ref{lem::duhamel_map_boundedness} $(iii)$ 
since $u\in C([0,T],C^1_b(\R^d))$  by assumption, there is a modulus of continuity $\Tilde{\omega}$ such that $|Du(t_1,y)-Du(t_2,y)|\leq \Tilde{\omega}(t_2-t_1)$ for all $y\in \R^d$. Take $R_1=\max_{t\in[\delta_0,T]}\|Du(\cdot,t)\|_{L^\infty(\R^d)}$, then for all $x \in B(0,R)$,
\begin{equation}\label{eq:f_locally_uniformly_cont}
\begin{aligned}
    |f(t_2,x)-f(t_1,x)| &\leq |H(t_2,x,Du(t_2,x))-H(t_1,x,Du(t_2,x))| \\
    &\qquad \qquad +|H(t_1,x,Du(t_2,x))-H(t_1,x,Du(t_1,x))| \\
%    & \leq \omega_{R}(t_2-t_1)+L_{R_1}|Du(t_2,x)-Du(t_1,x)| \\
    &\leq  \omega_{R}(t_2-t_1)+L_{R_1}\Tilde{\omega}(t_2-t_1) =: \overline{\omega}_R(t_2-t_1).
\end{aligned}
\end{equation}

\noindent \textbf{Step 1: The result holds for $\Phi_\delta$.} We claim that
\begin{align}\label{eq:classical_sol_calc_leibniz}
    \partial_t \Phi_\delta(t,x) &= \varphi(t,t-\delta,x)+\int_0^{t-\delta}\partial_t\varphi(t,s,x) \, ds =\varphi(t,t-\delta,x) + \mathcal{L}\Phi_\delta(t,x).
\end{align}
The first equality is justified as follows: By splitting integrals and using the mean value theorem for integrals we can write, for some $\lambda \in (0,1)$ depending on $\tau$,
\begin{align*}
    \frac{\Phi_\delta(t+\tau,x)-\Phi_\delta(t,x)}{\tau} %= \frac{\int_{t-\delta}^{t+\tau-\delta}\varphi(t+\tau,s,x) \, ds}{\tau} + \int_0^{t-\delta} \frac{\varphi(t+\tau,s,x)-\varphi(t,s,x)}{\tau} \, ds.
    = \varphi(t+\tau, t-\delta +\lambda\tau, x) + \int_0^{t-\delta} \frac{\varphi(t+\tau,s,x)-\varphi(t,s,x)}{\tau} \, ds.
\end{align*}
Here $\varphi(t+\tau, t+\lambda\tau-\delta, x)\to\varphi(t, t-\delta, x)$ as $\tau\to 0$: For any $\epsilon>0$, we can take $R$ large so that $\underset{s\in [0,\delta]}{\sup}\int_{B(0,R)^c}p_s(y)\, dy < \epsilon$ \cite[(3.2)]{pruitt_growth_random_walks}, and then
\begin{align*}
    &|\varphi(t+\tau, t-\delta+\lambda\tau, x) - \varphi(t,t-\delta,x)|\\ &= \int_{\R^d} %p_{\delta-\lambda\tau}
    p_{(1-\lambda)\tau+\delta}
    (x-y) |f(t-\delta+\lambda\tau,y) - f(t-\delta,y)|\, dy +\, |(P_\delta - P_{(1-\lambda)\tau+\delta}
    %\delta-\lambda \tau}
    ) f(t-\delta)( x)|\\
    &\leq\, 2\epsilon \|f\|_{\infty} + \overline{\omega}_R(\lambda \tau) + |(P_{ (1-\lambda) \tau} - I)P_{\delta}f(t-\delta)( x)|\to 2\epsilon \|f\|_{\infty}\qquad\text{as}\qquad\tau\to 0.
\end{align*}
%By taking $\tau$ small enough we find that the last expression is smaller than $C\epsilon$.

%As $p_\cdot(x)$ and $f(\cdot,x)$ (by \eqref{eq:f_locally_uniformly_cont}) are continuous, $\varphi$ is continuous in the first and second arguments by the theorem of continuity under the integral. Therefore, $\lim_{\tau \rightarrow 0}\varphi(t+\tau, t+\lambda \tau - \delta, x) = \varphi(t,t-\delta,x)$.

By Lemma \ref{lem:H_regularity}, $\sup_{t\in(0,T]}\|f(t,\cdot)\|_\infty < \infty$, so (in particular) $\varphi(t,s,\cdot) \in C^2_b(\R^d)$ for all $0<s<t$. For $s\in (0,t-\delta)$, Lemma \ref{lem:heateqn}
%\ref{lem:L_well_defined_on_holder}
and Theorem \ref{lem::lem1} (with $\beta=0$) then yield
\begin{align}\label{eq::varphi_bounded_calc}
    |\partial_t \varphi(t,s,x)| &= |\mathcal{L}\varphi(t,s,x)| \lesssim \|\varphi(t,s,\cdot)\|_{C^2_b} \lesssim (t-s)^{-\frac{2}{\alpha}} \leq \delta^{-\frac{2}{\alpha}} < \infty.
\end{align}
Therefore, by the dominated convergence theorem, 
\begin{align*}
    \lim_{\tau \rightarrow 0} \int_0^{t-\delta} \frac{\varphi(t+\tau,s,x)-\varphi(t,s,x)}{\tau} \, ds = \int_0^{t-\delta} \partial_t \varphi(t,s,x) \, ds.
\end{align*}
The second equality of \eqref{eq:classical_sol_calc_leibniz} then follows by the arguments related to \eqref{eq::varphi_bounded_calc} and Fubini's theorem:
\begin{align*}
    \int_0^{t-\delta} \partial_t \varphi(t,s,x) \, ds=  \int_0^{t-\delta} \mathcal{L}\varphi(t,s,x) \, ds =  \mathcal{L}\Phi_\delta(t,x).
\end{align*}
\noindent \textbf{Step 2: $\partial_t \Phi_\delta$ converges uniformly as $\delta \rightarrow 0$.} In the rest of the proof, let $\sigma=2$ if $\alpha+\beta > 2$, and $\sigma\in(\alpha,\alpha+\beta)$ be given by \ref{assump:L2} if $\alpha+\beta<2$.
Recall that $\Phi(t,\cdot)\in C^{\alpha+\beta}_b(\R^d)$ by Theorem \ref{thm::full_schauder_regularity}, so $\mL\Phi(t,x)$ is well-defined. Furthermore, Theorems \ref{lem::lem1} and \ref{lem:H_regularity} (and \ref{assump:L2} if $\alpha+\beta<2$) yield
\begin{align*}
    &\int_0^t\int_{\R^d}|\varphi(t,s,x+z) - \varphi(t,s,x) - D\varphi(t,s,x)\cdot z\textbf{1}_{B(0,1)}(z)|\, d\mu(z)\, ds \\
    &\leq \int_0^t \|\varphi(t,s,\cdot)\|_{C^{\sigma}_b} \int_{\R^d} (1\wedge|z|^\sigma)\, d\mu(z)\, ds\leq C\int_0^t (t-s)^{-\frac{\sigma-\beta}{\alpha}}s^{-\frac{\beta}{\alpha}}\, ds < \infty.
\end{align*}
By this, Fubini's theorem, and the dominated convergence theorem (if $\mL$ has a second order part),
%\footnote{The dominated convergence is used to prove that the second order part of $\mL$ commutes with the integral.} 
$\mathcal{L}\Phi(t,x) = \int_0^{t} \mathcal{L}\varphi(t,s,x) \, ds$.
% \begin{align*}
%     &\mathcal{L}\Phi(t,x) = \int_0^{t} \mathcal{L}\varphi(t,s,x) \, ds
% \end{align*}

Consider now the right-hand side of \eqref{eq:classical_sol_calc_leibniz}. Let
%$0<\epsilon < \beta$ such %that $\alpha + \epsilon < 2$ if $\alpha +\beta < 2$ and $\alpha +\epsilon >2$ if $\alpha +\beta >2$, and
$$M_{\delta_0} := \sup_{s\in (t-\delta,t)} \|f(s, \cdot) \|_{C_b^{\beta}} \leq \sup_{s\in (\delta_0/2,T]} \|f(s, \cdot) \|_{C_b^{\beta}} < \infty.$$ 
%Then by Lemma
%s \ref{thm:holder_interpolation} and
%\ref{lem:H_regularity}, and Theorem \ref{lem::lem1}  (and \ref{assump:L2} if $\alpha+\beta<2$),
Similar computations as above, using Theorem \ref{lem::lem1}, Lemma
\ref{lem:H_regularity}, (and \ref{assump:L2} if $\alpha+\beta<2$), show that
\begin{align*}
    |\mathcal{L}\Phi(t,x)-\mathcal{L}\Phi_\delta(t,x)|
    &\leq  \int_{t-\delta}^t |\mathcal{L}\varphi(t,s,x)| \, ds \lesssim \int_{t-\delta}^t \| \varphi(t,s,\cdot) \|_{C_b^{
    \sigma}} \\
    & \lesssim M_{\delta_0}\int_{t-\delta}^t (t-s)^{-\frac{
    \sigma-\beta}{\alpha}} \, ds \simeq M_{\delta_0} \delta^{\frac{\alpha+\beta-
    \sigma}{\alpha}} \xrightarrow{\delta \rightarrow 0} 0.
\end{align*}
By adding and subtracting terms, we see that
\begin{align*}
    \varphi(t,t-\delta,x) =\int_{\R^d}p(\delta,x-y)(f(t-\delta,y)-f(t,y) ) \, dy  + \int_{\R^d}p(\delta,x-y)f(t,y) \, dy.
\end{align*}
Arguing as in \eqref{eq:A_k_calculation} and using the $\beta$-Hölder continuity of $f$ from Lemma \ref{lem:H_regularity}, 
\begin{align*}
    \Big|f(t,&x)-\int_{\R^d}p_\delta(x-y)f(t,y) \, dy \Big| %= \Big|\int_{\R^d}p_\delta(x-y)(f(t,x)-f(t,y)) \, dy \Big|  \\
    %&\leq [f(t,\cdot)]_\beta \int_{B(x,1)}p_\delta(x-y)|x-y|^\beta \, dy  + \|f(t,\cdot)\|_\infty \int_{\R^d \setminus B(x,1)}p_\delta(x-y) \, dy  \\
    %& \leq 2 M_{\delta_0} c_0 \delta^{\frac{\beta}{\alpha}} 
    \xrightarrow{\delta \rightarrow 0} 0,
\end{align*}
and by \eqref{eq:f_locally_uniformly_cont},
\begin{align*}
   \Big| \int_{\R^d} &p_\delta(x-y)(f(t-\delta,y)-f(t,y) ) \, dy \Big| \\ 
   &\leq 2\|f \|_\infty \int_{\R^d \setminus B(x,1)} p_\delta(x-y) \, dy + \overline{\omega}_1(\delta)\int_{B(x,1)} p_\delta(x-y) \, dy \xrightarrow{\delta \rightarrow 0} 0,
\end{align*}
where the first term tends to zero as $\delta \rightarrow 0$ by \cite[(3.2)]{pruitt_growth_random_walks}.

Consequently, $\partial_t\Phi_\delta$ converges uniformly to $H(t,x,Du(t,x)) + \mathcal{L}\Phi(t,x)$. This, together with the fact that $\Phi_\delta$ converges uniformly to $\Phi$,
\begin{align*}
    |\Phi(t,x)-\Phi_\delta(t,x)| \leq M_{\delta_0}\int_{t-\delta}^t \|p(t-s)\|_{L^1} \, ds = M_{\delta_0} \delta \xrightarrow{\delta \rightarrow 0} 0,
\end{align*}
implies 
%by a standard result (see e.g. \cite[Theorem 3.7.1]{tao2022analysis})
that $\partial_t \Phi$ exists and equals the limit as $\delta \rightarrow 0$ of $\partial_t \Phi_\delta$, i.e. \eqref{eq:classical_sol_remains_to_show} holds.
\end{proof}

\subsection{Long-time existence for the viscous HJ equation}
Under additional assumptions on the Hamiltonian $H$, global Lipschitz bounds on $u$ have been shown in the literature. In such settings, we now prove long-time existence for mild and smooth solutions of \eqref{eq:hjb}. 
The gradient estimates are usually obtained in the context of viscosity solutions, see e.g. \cite[Section~1.1]{MR2121115} or \cite{MR2129093}, but they apply also to smooth solutions since smooth solutions are viscosity solutions and viscosity solutions are unique.
\begin{thm}\label{thm:long_time_existence}
    Assume $T>0$, \ref{assump:H2}, \ref{assump:H1}, \ref{NDa} with $\alpha\in (1,2]$,
    and that one of the following hold:

\begin{enumerate}[label=(\roman*)]
    \item There exists an $L>0$ such that 
    \begin{align*}
    |H(t,x,p)-H(t,y,p)|\leq L(1+|p|)|x-y|, \quad x,y,p\in\R^d, \ t\in (0,T].
\end{align*}
\item There exists $m>1$, $\overline{K}, b_m, L_H>0$ and a modulus of continuity $\zeta$ such that for all $\mu \in (0,1)$, $x,y, p, q \in \R^d$, $|q|\leq 1$ and $t\in(0,T]$ the following hold:
\begin{align*}
    &\mu H(t,x,\mu^{-1}p)-H(t,x,p) \geq (1-\mu)(b_m|p|^m-\overline{K}), \\
    &H(t,y,p+q)   -H(t,x,p) \leq L_H|x-y|(1+|p|^m)+\zeta(|q|)(1+|p|^{m-1}).
\end{align*}
\end{enumerate}

\noindent Then there is a unique bounded viscosity solution $u$ of \eqref{eq:hjb} on $(0,T]\times \R^d$ and a constant $M_T>0$ such that %$u(0,x)=u_0$ and 
$\|Du(t,\cdot) \|_\infty \leq M_T$. Furthermore, $u$ is a mild and classical solution on $(0,T]$ and $u(t,\cdot) \in C_b^{\alpha+1}(\R^d)$ if $\alpha\in (1,2)$, while $u(t,\cdot) \in C_b^{\alpha+1-\epsilon}(\R^d)$ for arbitrarily small $\epsilon$ if $\alpha=2$. 
\end{thm}
\begin{rem}
    (a) When $(i)$ holds, we can take $M_T = e^{2L T}\big(\tfrac{1}{2}L + \|Du_0 \|_\infty^2 \big)^{\frac{1}{2}}$ as in \cite[Theorem 5.3]{MR4309434} (with $f=0$, see also \cite{MR2121115}).  When $(ii)$ holds, $M_T$ 
    %is not explicitly given but
    depends only on $u_0$, $[u_0]_1$, $T$, and $\text{osc}_T(u):=\sup_{t\in (0,T]} \{\sup_{\R^d} u(t,\cdot) - \inf_{\R^d} u(t,\cdot) \}$ by \cite[Proposition 3.3]{Barles_2017}.

    \medskip
    \noindent (b) $H$ is superlinear in case (ii). %$(ii)$ is the case of superlinear $H$.
The first inequality %in $(ii)$ 
 is a coercivity condition that enables a weak Bernstein argument to be used
 %allows for the application of the weak Bernstein method, and implies in particular that $H$ is coercive 
 (see the discussion in   \cite[Assumption (H1)]{Barles_2017}). An example %(with detailed computations) of $H$ that satisfies
 satisfying $(ii)$ is given in %e.g. 
 \cite[Section 1.1]{GALISE2016194},
\begin{align*}
    H(x,p)=c(x)|p|^m+a(x)|p|^l, \qquad x,p \in \R^d,
\end{align*}
with $c, a$ bounded uniformly continuous in $\R^d$, $c(x)\geq \underline{c}>0$, $m>1$, $1\leq l <m$.

\medskip
\noindent (c) Both $(i)$ and $(ii)$ assume (locally) $x$-Lipschitz $H$, %placing us in the regime where our 
and hence \ref{assump:H} holds and mild solutions are smooth (Schauder regular) %are more than $C^2$ 
%twice differentiable 
by 
%, cf. 
Theorem \ref{thm::full_schauder_regularity} and classical solutions by Theorem \ref{thm:u_is_classical_sol}. 
Using the Ishii--Lions method and strong ellipticity of $\mL$, it is possible to obtain global Lipschitz bounds when $H$ is 
just continuous but satisfy certain growth bounds. We refer to % see e.g. 
\cite{MR2911421} for general results and examples,
% There are other results that give global Lipschitz bounds for Hamiltonians that are 
% %only Hölder continuous, or 
% just continuous but satisfy certain growth bounds, %in $x$,
% see e.g. \cite{MR2911421}. 
 % The assumptions therein involve a slightly different uniform ellipticity conditions than those in our paper, which are also tied to the assumptions on $H$; therefore we refrain from formulating a full result here. 
 %An % representative 
 %example is 
 including a setting with
 $H(t,x,Du) = c(x) |Du|^\alpha$, $c\in C_b(\R^d)$, and $\mu(dz) \approx |z|^{-d-\alpha}\, dz$
 %, 
 %in the setting of periodic solutions, see 
  \cite[Eq. (25)]% and Corollary~7]
 {MR2911421}.

% (c) Both $(i)$ and $(ii)$ assume $x$-Lipschitz $H$, placing us in the regime where our solutions are classical without assuming \ref{assump:L2}, cf. Theorem \ref{thm:u_is_classical_sol}. There may be other assumptions which only yield $\beta$-Hölder regularity of $H$ in $x$ as well as global Lipschitz bounds on $u$, but which place us in the regime $\alpha+\beta<2$. If one wants to avoid assuming \ref{assump:L2} but still have long time existence, it might be possible to bypass the requirement of solutions being classical by showing directly that our mild solutions are in fact viscosity solutions (ref. Ishii?); this would yield the uniqueness required for the patching argument to work.

 \medskip 
    \noindent (d) It is well known that in the special case of Lipschitz-in-$p$ Hamiltonian, %global a priori 
    Lipschitz bounds on $u$ are easy to obtain. %and consequently
    Assumptions like (i) or (ii) are then not needed for long time existence. Indeed, for such $H$ we have
    \begin{align*}
        \|u(t,\cdot)\|_{C^1_b} \lesssim \|u_0\|_{C^1_b} + \int_0^t (t-s)^{-\frac 1\alpha}(H_0 + \|u(s,\cdot)\|_{C^1_b})\, ds.
    \end{align*}
    %and therefore, by the 
      Then by Gr\"onwall's inequality in Lemma~\ref{lem:generalized_gronwall}, 
$\|u(t,\cdot)\|_{C^1_b}$ is bounded by a constant which depends on $t\in [0,\infty)$, but is locally bounded.  This global-in-time a priori bound %then  in this case, by 
lets us glue the short-time solutions together (as explained in the proof of Theorem~\ref{thm:long_time_existence} below), meaning that the existence of Lemma \ref{thm::hjb_sol_existence} actually holds on any finite time horizon. 
      A particular example is given by $H(t,x,p)=b(t,x)\cdot p + f(t.x)$ for $f,b\in C_b([0,T]\times \R^d)$.

\end{rem}
\begin{proof} %\ \medskip\noindent 
\ {\em 1) \ There exists a unique globally Lipschitz viscosity solution $u$ of \eqref{eq:hjb}.} 
    Assume first $(i)$ holds. Existence, uniqueness, and Lipschitz bounds for bounded viscosity solutions are given by \cite[Theorem 5.3]{MR4309434}.  Our assumptions are the same, except the local Lipschitz continuity for $H(t,x,\cdot)$ in \ref{assump:H2} instead of their $C^3$-type condition $(A3)$. But the proofs are not affected, since higher derivatives of $H$ are only needed to prove higher regularity of $u$.
    
    Then we consider case $(ii)$.
    %\bigskip
    %\noindent If $(ii)$ holds: 
    Existence and uniqueness of a bounded viscosity solution and the Lipschitz bound are given by \cite[Corollary 2.5, Proposition 3.3]{Barles_2017} respectively. Note that their (A) and (J) assumptions are trivially satisfied in our setting (corresponding to $j(x,z)=z$ and $A=0$ in \cite{Barles_2017}).

    \bigskip
    \noindent {\em 2) \ The viscosity solution $u$ is a mild and classical solution satisfying Schauder regularity.}
     Let $u$ be the  viscosity solution from step 1,
    $\widetilde{M}_T:=\sup_{t\in [0,T]} \|u(t,\cdot) \|_{C_b^1}<\infty$, and
    %\leq M_T+\sup_{t\in [0,T]} \|u(t,\cdot)\|_\infty $. 
     $t_0\in [0,T]$. Then equation \eqref{eq:hjb} with initial condition $u(t_0,\cdot)$, has a mild solution $v$ on $(t_0,t_0+\widetilde{T}_0)$ by Theorem~\ref{thm::hjb_sol_existence}, where
    %an interval of length 
    \begin{align*}
        \widetilde{T}_0:= \bigg(\frac{\alpha-1}{2\alpha c_0(L_{\widetilde{R}_1}\widetilde{R}_1+H_0)}\bigg)^{\frac{\alpha}{\alpha-1}} \ \wedge \ \frac{1}{2(L_{\widetilde{R}_1}\widetilde{R}_1+H_0)}\qquad \text{and} \qquad \widetilde{R}_1 = \widetilde{M}_T+1.
    \end{align*} 
    %where $\widetilde{R}_1 = \widetilde{M}_T+1$. 
    Since both (i) and (ii) imply that \ref{assump:H} holds (with $\beta=1$), the mild solution $v$ is also classical by Theorem~\ref{thm:u_is_classical_sol} and $C^{\alpha+1}_b$ in space by Theorem~\ref{thm::full_schauder_regularity}. Note that the existence time $\widetilde T_0$ is independent of $t_0$.
    Since classical solutions are viscosity solutions and viscosity solutions are unique, $u=v$ on $(t_0,t_0+\widetilde T_0)$. Since $t_0\in[0,T]$ is arbitrary, this means that $u$ is a mild and classical solution satisfying the Schauder regularity for every $t\in(0,T)$. 
\end{proof}

\subsection{Optimal regularity in time and space-time for viscous HJ equations} In this section we consider $H$ with fixed spatial H\"older regularity of order $\beta\in (0,1)$ given by \ref{assump:H}. If $\mL$ is an operator of order $\alpha$, the corresponding (fractional) parabolic scaling is $dt\sim (dx)^\alpha$, and optimal  $C_x^{\alpha+\beta}$ regularity in space should then correspond to optimal $C_t^{\frac{\alpha+\beta}\alpha}$ regularity in time. Overall, the space-time Schauder estimates here state that the viscous HJ equation \eqref{eq:hjb} transforms $C^{\frac{\beta}{\alpha},\beta}_b$ `data' into $C^{\frac{\alpha+\beta}\alpha,\alpha+\beta}_b$ solutions (see \eqref{eq:spacetimenorm} below). When $\alpha=2$ this is consistent with the classical space-time Schauder theory of linear equations \cite{MR1465184,MR0241822}.  To achieve this optimal regularity, we need to enforce that $\mL$ is an order $\alpha$ operator and add (the expected) $C_t^{\frac{\beta}\alpha}$ conditions on the time regularity of $H$:
\bigskip
\begin{enumerate}
\myitem{$\mathbf{(L2')}$}\label{assump:L2'}
%\item[(L2')\label{assump:L2'}]  %L4
%Let $\alpha \in [1,2)$.  
$\mL$ is of the form \eqref{eq:L} with $A=0$ and there is $c>0$ such that for $r\in(0,1)$,
% $$r^\alpha \int_{|z|<1} \frac{|z|^2}{r^2} \wedge 1 \ d\mu(z) \leq c \qquad \text{for} \qquad r\in(0,1).$$
\begin{align*}
\begin{cases}\displaystyle
\frac1{r^{\beta}}\int_{B_r} |y|^{\beta+\alpha}\, d\mu(y)+\frac1{r^{1-\alpha}}\int_{B_1\setminus B_r} |y|\, d\mu(y) \leq c , &\text{when} \ \ \alpha+\beta \in (1,2),\ \beta\in(0,1),\\[0.5cm]
\displaystyle
\frac1{r^{2-\alpha}}\int_{B_r} |y|^2\, d\mu(y)+ \frac1{r^{\beta-1}}\int_{B_1\setminus B_r} |y|^{\beta+\alpha-1}\, d\mu(y) \leq c, & \text{when} \ \  \alpha+\beta \in (2,3),\ \beta\in(0,1),
\end{cases}
\end{align*}
where $\alpha\in (1,2]$, $\beta$
%\in (0,1)$ 
are given by \ref{NDa}, \ref{assump:H}.
\bigskip 
\myitem{$\mathbf{(H_t)}$} \label{assump:H_t}
%\item[(H$_t$)%$\mathbf{({H_t})}$
    %\label{assump:H_t}]  
    For each $R>0$ there is $K_R>0$ such that for all $t_1,t_2\in [0,T]$ with $|t_2-t_1|<1$, $x\in \R^d$, $p\in B(0,R)$,
$$|H(t_2,x,p)-H(t_1,x,p)| \leq K_R|t_2-t_1|^{\frac{\beta}{\alpha}},$$
where $\alpha\in (1,2]$, $\beta\in (0,1]$ are given by \ref{NDa}, \ref{assump:H}.
\end{enumerate}\bigskip

 Assumption \ref{assump:L2'} is, like \ref{assump:L2} and \cite[(5)]{MR4309434}, a (more precise) condition %(a more precise one) 
 on the maximal order of $\LL$.
 %, but it is weaker than \ref{assump:L2} and more precise than \cite[(5)]{MR4309434}. 
 Combined with \ref{NDa}, it implies that the order %of $\LL$ 
 is exactly $\alpha$, see Remark \ref{rem:L} above and Lemma \ref{lem:order} below. 
It is satisfied for a large class of operators, including $\LL=-(-\Delta)^{\frac\alpha2}$ for $\alpha\in(1,2)$ and the ones in Example \ref{ex:NDa} when $\alpha_i=\alpha$. The condition implies several a priori and interpolation estimates:

\begin{lem}\label{lem:order}
Assume \ref{assump:L2'} and $\alpha+\beta \in (1,2)\cup(2,3)$. Then there are $C_1,\dots,C_5>0$ such that for every $x\in\R^d$, $r\in (0,1]$, and $\phi\in C^2_b(\R^d)$,
\begin{align*}
    &\tag{i}\int_{\R^d} \big|\phi(x) - \phi(x+y) - D\phi(x)\cdot y\textnormal{\textbf{1}}_{B_1}(y)\big|\, d\mu(y) \leq C_1\Big( \|D^2 \phi\|_\infty r^{2-\alpha} + \|D\phi\|_\infty r^{1-\alpha} + \|\phi\|_\infty\Big),\hspace{0.4cm}
    \\[0.2cm]
    & \tag{ii}\|\mL\phi\|_\infty \leq C_2\Big(\|\phi\|_\infty +\|\phi\|_{C^1_b}^{2-\alpha}\|\phi\|_{C^2_b}^{\alpha-1}\Big),\\[0.2cm]
%\|D\phi\|_\infty^{2-\alpha}\|D^2\phi\|_{\infty}^{\alpha-1}\Big),
    &\tag{iii}\|\mL\phi\|_{C^\beta_b} \leq C_3 \|\phi\|_{C^{\alpha+\beta}_b},%\label{eq:LphiCbeta}
    \\[0.2cm]
    \intertext{and if in addition $\alpha+\beta<2$, then for all $\phi \in C^{\alpha+\beta}_b(\R^d)$}
     &\tag{iv}\int_{\R^d} \big|\phi(x) - \phi(x+y) - D\phi(x)\cdot y\textnormal{\textbf{1}}_{B_1}(y)\big|\, d\mu(y) \leq C_4\Big( [D \phi]_{\alpha+\beta-1} r^{\beta} + \|D\phi\|_\infty r^{1-\alpha} + \|\phi\|_\infty\Big),
    \\[0.2cm]
    & \tag{v}\|\mL\phi\|_\infty \leq C_5\Big(\|\phi\|_\infty +\|\phi\|_{C^1_b}^{\frac{\beta}{\alpha+\beta-1}}\|\phi\|_{C^{\alpha+\beta-1}_b}^{\frac{\alpha-1}{\alpha+\beta-1}}\Big).
    %\|D\phi\|_\infty^{\frac{\beta}{\alpha+\beta-1}}[D\phi]_{\alpha+\beta-1}^{\frac{\alpha-1}{\alpha+\beta-1}}\Big).
\end{align*}
\end{lem}

\begin{rem}\label{rem:weaker}
In view of Lemma \ref{lem:order}, the proof and statement of Lemma \ref{thm:u_is_classical_sol} still hold when \ref{assump:L2} is replaced by \ref{assump:L2'}.
\end{rem}
Bounds like (i), (ii), (iv), (v) were given in \cite[Lemma~2.1]{MR4309434} under assumption \cite[(5)]{MR4309434}.
Estimates like (iii) were proved e.g. in \cite{MR2270163} for $(-\Delta)^{\alpha/2}$.   We give a short proof for completeness.
\begin{proof}
To prove (i) and (iv),  %\eqref{eq:lemma21}, 
we split:
\begin{align}\label{eq:Lphisplit}
    |\mL \phi(x)| &\leq \bigg(\int_{B_1^c} + \int_{B_1\setminus B_r} + \int_{B_r}\bigg) \big|\phi(x) - \phi(x+y) - D\phi(x)\cdot y\textbf{1}_{B_1}(y)\big|\, d\mu(y)\\
    &\leq 2\|\phi\|_\infty\mu(B_1^c) + 2\|D\phi\|_\infty \int_{B_1\setminus B_r} |y|\, d\mu(y) + \|D^2\phi\|_\infty \int_{B_r}|y|^2\, d\mu(y).\nonumber%\label{eq:Lphisplit2}
\end{align}
When $\alpha+\beta<2$, we can bound the integral over $B_r$ by $[D\phi]_{\alpha+\beta-1}\int_{B_r}|y|^{\alpha+\beta}\, d\mu(y)$ instead.
% Note that for $\alpha+\beta\in (1,2)$, $\int_{B_r} |y|^2\, d\mu(y) \leq cr^{2-\alpha-\beta} \int_{B_r} |y|^{\alpha+\beta}\, d\mu(y),$ while
Furthermore, for $\alpha+\beta \in (2,3)$, $\int_{B_1\setminus B_r} |y|\, d\mu(y) \leq   r^{2-(\alpha+\beta)}\int_{B_1\setminus B_r} |y|^{\beta+\alpha-1}\, d\mu(y)
%\leq cr^{1-\alpha}
$. Estimates (i) and (iv) %\eqref{eq:lemma21} 
then follow from \ref{assump:L2'}.  
Estimates (ii) and (v) follow from \eqref{eq:Lphisplit}, (i), and (iv) by taking $r = \|\phi\|_{C^1_b}/\|\phi\|_{C^2_b}$ and $r = \|\phi\|_{C^1_b}/\|\phi\|_{C^{\alpha+\beta-1}_b}$ respectively. 
%after minimizations over $r>0$.
%
% In a similar way, when $\alpha-\beta<2$ we have using \ref{assump:L2'} that
% \begin{align*}
%     |\mL \phi(x)| 
%     &\leq 2\|\phi\|_\infty\mu(B_1^c) + 2\|D\phi\|_\infty \int_{B_1\setminus B_r} |y|\, d\mu(y) + [D\phi]_{\alpha+\beta-1}  \int_{B_r}|y|^{\alpha+\beta}\, d\mu(y)\\
%     &\leq C(\|\phi\|_\infty +\|D\phi\|_\infty r^{1-\alpha}+[D\phi]_{\alpha+\beta-1}r^\beta)\\
%     &\leq C(\|\phi\|_\infty +\|D\phi\|_\infty^{\frac{\beta}{\alpha+\beta-1}}[D\phi]_{\alpha+\beta-1}^{\frac{\alpha-1}{\alpha+\beta-1}})
% \end{align*}
%
% Now we prove (ii) and (v). By \eqref{eq:Lphisplit} with $r=1$ 
% and \ref{assump:L2'}, a direct computation shows that $\|\mL \phi\|_\infty \leq C \|\phi\|_{C^{\alpha+\beta}_b}$. 

We now prove (iii). Estimates (ii) and (iv) give the correct $L^\infty$ control over $\mL \phi$, so it remains to  estimate $[\mL \phi]_\beta$.   Let $r = |x-x'| < 1$. 
We estimate the integrand in $\mL \phi(x) - \mL\phi(x')$,
    \begin{align*}
        &|\phi(x+y)-\phi(x) - y\cdot D\phi(x) -  \phi(x'+y)+\phi(x') + y\cdot D\phi(x') | \\
        & \leq \int_0^1 |y||D\phi(x+\lambda y) - D\phi(x) - D\phi(x'+\lambda y) + D\phi(x') |\, d\lambda \\
        &\leq 2|y|\cdot\begin{cases}
            [D\phi]_{\alpha+\beta-1}|x-x'|^{\alpha+\beta-1} \quad &\text{for} \quad y\in B_1\setminus B_r, \\[0.2cm]
            [D\phi]_{\alpha+\beta-1}|y|^{\alpha+\beta-1} \quad &\text{for}\quad y\in B_r, 
        \end{cases}  
        \end{align*}
 where the last inequality holds for $\alpha+\beta \in (1,2)$.  
If $\alpha+\beta\in(2,3)$, we instead use that
%can after the second line write 
\begin{align*}
    & |y||D\phi(x+\lambda y) - D\phi(x) - D\phi(x'+\lambda y) + D\phi(x') |  \\
    &\leq |y|\cdot\begin{cases}
          \int_0^1 |x-x'||D^2\phi(\lambda y+  x+\theta(x'-x)) - D^2\phi(x+\theta(x'-x)) |\, d\theta \quad &\text{for} \quad y\in B_1\setminus B_r,\\[0.2cm]\int_0^1 |y||D^2\phi(x+\theta\lambda y) - D^2\phi(x'+\theta\lambda y)|\, d\theta
          \quad &\text{for} \quad y\in B_r,
    \end{cases} \\
    &\leq |y|\cdot
    \begin{cases}
            [D^2\phi]_{\alpha+\beta-2}|x-x'||y|^{\alpha+\beta-2}\quad &\text{for} \quad y\in B_1\setminus B_r,\\[0.2cm]
            [D^2\phi]_{\alpha+\beta-2}|y||x-x'|^{\alpha+\beta-2}\quad &\text{for} \quad y\in  B_r.
    \end{cases}
\end{align*}
By these estimates, \eqref{eq:Lphisplit}, and \ref{assump:L2'}, when  $\alpha+\beta\in(1,2)$ we have that
\begin{align*}
    |\mL \phi(x) - \mL\phi(x')| &\lesssim \|\phi\|_{C^\beta_b}\mu(B_1^c)|x-x'|^\beta + \|\phi \|_{C^{\alpha+\beta}_b}|x-x'|^{\alpha+\beta-1}\int_{B_1\setminus B_r}|y|\, d\mu(y) \\
    &\quad + \|\phi \|_{C^{\alpha+\beta}_b}\int_{B_r}|y|^{\alpha+\beta}\, d\mu(y) \lesssim \| \phi\|_{C^{\alpha+\beta}_b}|x-x'|^{\beta},
\end{align*}
and when $\alpha+\beta \in(2,3)$ we have that 
\begin{align*}
    &|\mL \phi(x) - \mL\phi(x')| \lesssim \|\phi\|_{C^\beta_b}\mu(B_1^c)|x-x'|^\beta + \|\phi \|_{C^{\alpha+\beta}_b}|x-x'|\int_{B_1\setminus B_r}|y|^{\alpha+\beta-1}\, d\mu(y) \\
    &\qquad\qquad\qquad\qquad\quad + \|\phi \|_{C^{\alpha+\beta}_b}|x-x'|^{\alpha+\beta-2}\int_{B_r}|y|^{2}\, d\mu(y)\\ 
    & \lesssim |\phi\|_{C^\beta_b}|x-x'|^\beta  + \|\phi \|_{C^{\alpha+\beta}_b}|x-x'|\int_{B_1\setminus B_r}|y|^{\alpha+\beta-1}\, d\mu(y)+ \|\phi \|_{C^{\alpha+\beta}_b}|x-x'|^{\alpha+\beta-2}\int_{B_r}|y|^2\, d\mu(y) \\
    &\lesssim \| \phi\|_{C^{\alpha+\beta}_b}|x-x'|^{\beta}.\qedhere
\end{align*}
\end{proof}

The a priori estimates from Lemma \ref{lem:order} and regularity properties of the semigroup $P_t$ given by Theorem~\ref{prop:f_bound} imply that $P_t$ has the optimal H\"older continuity at $t=0$.
\begin{lem}\label{lem:t=0}
Assume \ref{NDa}, \ref{assump:L2'}, and $\alpha+\beta \in (1,2)\cup(2,3)$. Then there is $C>0$ such that
\begin{align*}%\label{eq:spacetotimebeta}
\|(P_t - I) \phi\|_\infty \leq C \|\phi\|_{C^\beta_b}t^{\frac \beta \alpha} \qquad {\rm for}\qquad t\in (0,1).
\end{align*}
\end{lem}
\begin{proof}
Take $t\in(0,1)$. Note that by properties of $P_t$ and absolute convergence of all integrals involved,
  \begin{align*}
(P_t - I) \phi(x)&=\int_0^t \partial_t P_s\phi(x)\, ds=\int_0^t \mL P_s\phi(x)\, ds = \mL\int_0^t P_s\phi(x)\, ds = \mL w(t,x),
\end{align*}  
where $w(x,t)=\int_0^t P_{t-r}\phi(x)\,dr$. Assume first $\alpha+\beta>2$. Then by Lemma \ref{lem:order} (ii), the estimate $\|w(t)\|_\infty\leq t\|\phi\|_\infty$, and Theorem~\ref{prop:f_bound} (i), 
\begin{align*}
&\|(P_t - I) \phi\|_\infty \leq C_2 \big(\|w(t)\|_\infty + \|w(t)\|_{C_b^1}^{2-\alpha}\|w(t)\|_{C_b^2}^{\alpha-1}\big)\\
% \|Dw(t)\|_\infty^{2-\alpha}\|D^2w(t)\|_\infty^{\alpha-1}\big)\\
&\leq C_2\big(\|\phi\|_\infty t + \|\phi\|_{C_b^\beta}\,t^{\frac{\alpha+\beta-1}{\alpha}\cdot(2-\alpha)}\,t^{\frac{\alpha+\beta-2}{\alpha}\cdot(\alpha-1)}\big)\leq 2C_2 \|\phi\|_{C_b^\beta} t^{\frac{\beta}{\alpha}}.
\end{align*}
When $\alpha+\beta<2$, we use Lemma \ref{lem:order} (v) and Proposition \ref{prop:f_bound} (i) and (ii), 
\begin{align*}
&\|(P_t - I) \phi\|_\infty 
\leq C_5\Big(\|w(t)\|_\infty +
\|w(t)\|_{C^1_b}^{\frac{\beta}{\alpha+\beta-1}}\|w(t)\|_{C^{\alpha+\beta-1}_b}^{\frac{\alpha-1}{\alpha+\beta-1}}\Big)\\
% \|Dw(t)\|_\infty^{\frac{\beta}{\alpha+\beta-1}}[Dw(t)]_{\alpha+\beta-1}^{\frac{\alpha-1}{\alpha+\beta-1}}\Big)\\
&\leq C_5\big(\|\phi\|_\infty t + \|\phi\|_{C_b^\beta}\,t^{\frac{\alpha+\beta-1}{\alpha}\cdot\frac{\beta}{\alpha+\beta-1}}\,1^{\frac{\alpha-1}{\alpha+\beta-1}}\big)\leq 2C_5 \|\phi\|_{C_b^\beta} t^{\frac{\beta}{\alpha}}.\qedhere
\end{align*}
\end{proof}

We now state the optimal time regularity result for our problem.
  
\begin{thm}\label{thm:timeSchauder}
    Assume the assumptions of Theorem \ref{thm::full_schauder_regularity}, \ref{assump:H_t}, %\ref{assump:L3}, 
    \ref{assump:L2'}, 
    %with $\alpha\in (1,2)$ and $\beta\in (0,1)$ same in all the assumptions, 
    $u_0\in C^{\alpha+\beta}_b(\R^d)$, and $\alpha+\beta\in (1,2)\cup(2,3)$. Then 
    \begin{align*}
        \sup_x \|u(\cdot,x)\|_{C^{1+\frac{\beta}{\alpha}}_b([0,T])} \leq C(d,\alpha,\beta,T,\|Du\|_\infty,H)  (\sup_{t\in [0,T]} \|u(t,\cdot)\|_{C^{\alpha + \beta}_b}+1).
    \end{align*}

    %In the following result we omit
    % for every $\epsilon\in (0,T)$ we have $\sup_x \|u(\cdot,x)\|_{C^{1+\frac{\beta}{\alpha}}_b((\epsilon,T])} <\infty$.
    %\in C_b^{1+\frac{\beta}{\alpha}}((0,T])$ uniformly in $x$. 
\end{thm}
    \begin{rem}
       The classical case $\alpha = 2$ requires a different proof and has been omitted. Theorem~\ref{thm:timeSchauder}  combined with the spatial bounds from Theorem \ref{thm::full_schauder_regularity}, can be used to prove blow-up rates for the $C^{1+\frac{\beta}{\alpha}}_b([\varepsilon,T])$-norm as $\varepsilon\to0$ for initial conditions $u_0\in C^1_b(\R^d)$. 
    \end{rem} 
\begin{proof}
Let $f(t,x):=H(t,x,Du(t,x))$.  Then by Theorem~\ref{thm:u_is_classical_sol} and Remark~\ref{rem:weaker} we find that
    $\partial_t u = \mL u + f$, and hence it suffices to show that $\mL u$ and $f$ are $\frac \beta \alpha$-H\"older in time. 
    
    We start with $f$. Let $\Delta t>0$. Set $R_1 = \|Du\|_\infty$ and use \ref{assump:H_t} %Theorem~\ref{thm:u_is_classical_sol} 
    and \ref{assump:H2} to see that
\begin{equation}\label{eq:temporal_reg_f}
\begin{aligned}
    |f(t+\Delta t,x)-f(t,x)| &\leq |H(t+\Delta t,x,Du(t+\Delta t,x))-H(t+\Delta t,x,Du(t,x))| \\
    &  \qquad +|H(t+\Delta t,x,Du(t,x))-H(t,x,Du(t,x))|\\
    & \leq L_{R_1}\sup_x[Du(\cdot,x)]_{\frac \beta \alpha} \Delta t^\frac{\beta}{\alpha}+ \begin{cases}
        K_{R_1}\Delta t^{\frac{\beta}{\alpha}}, \quad &\Delta t<1,\\[0.2cm]
         2(L_{R_1}R_1+H_0)\Delta t^{\frac{\beta}{\alpha}}, \quad &\Delta t\geq1.
    \end{cases}
\end{aligned}
\end{equation}
To estimate $[Du(\cdot,x)]_{\frac \beta \alpha}$, we note that 
$Du(t,x) = P_t[Du_0](x) + \int_0^t DP_{t-s}[f(s,\cdot)](x)\, ds$. By  %\ref{assump:L3},
% and
Lemma \ref{lem:t=0}, 
    $\sup\limits_{x\in \R^d}\|P_{(\cdot)} Du_0(x)\|_{C^{\frac \beta\alpha}_b} \lesssim \|u_0\|_{C^{1+\beta}_b}$.
%\end{align*}
Furthermore, using Theorem~\ref{lem::lem1} and Lemma \ref{lem:t=0},
\begin{align*}
    &\bigg|\int_0^{t+\Delta t} DP_{t+\Delta t-s}f(s,\cdot)(x)\, ds - \int_0^t DP_{t-s}f(s,\cdot)(x)\, ds\bigg|\\
    &\leq \int_0^t\|D(P_{\Delta t} - I)P_{t-s}f(s,\cdot)\|_\infty \, ds + \int_0^{\Delta t} \|DP_s f(t+\Delta t - s)(x)\|_\infty \, ds\\
    &\lesssim (\Delta t)^{\frac \beta\alpha} \int_0^t \|DP_{t-s} f(s,\cdot)\|_{C^\beta_b}\, ds + \sup\limits_{t\in (0,T]} \|f(t,\cdot)\|_{C^\beta_b} \int_0^{\Delta t}s^{-\frac{1-\beta}{\alpha}}\, ds\\
    &\lesssim (\Delta t)^{\frac \beta\alpha}\sup\limits_{t\in (0,T]} \|f(t,\cdot)\|_{C^\beta_b}.
\end{align*}
Thus we get
\begin{align*}
     \sup\limits_{x\in \R^d} \|f(\cdot,x)\|_{C^{\frac \beta \alpha}_b} &\leq C(d,\alpha,\beta,T,\|Du\|_\infty,H) \big(\sup\limits_{t\in (0,T]} \|f(t,\cdot)\|_{C^\beta_b}+1\big)
     \\
     &\leq C'(d,\alpha,\beta,T,\|Du\|_\infty,H) \big(\sup\limits_{t\in (0,T]} \|u(t,\cdot)\|_{C_b^{1+\beta}} +1\big),
\end{align*}
with the last inequality following from the last lines of the proof of Lemma~\ref{lem:H_regularity}.

It remains to show the H\"older regularity of \begin{align*}
    \mL u(t,x) = \mL P_t u_0(x) + \mL \int_0^t P_{t-s} f(s,\cdot)(x)\, ds.
\end{align*}
For the first term we use Lemma \ref{lem:t=0}  and Lemma \ref{lem:order} (iii):  %\eqref{eq:LphiCbeta}:
\begin{align*}
   |\mL P_{t+\Delta t} u_0(x) - \mL P_t u_0(x)| \leq |(P_{\Delta t} - I)\mL P_tu_0(x)| \lesssim (\Delta t)^{\frac \beta\alpha} \|\mL P_t u_0\|_{C^\beta_b} \lesssim (\Delta t)^{\frac \beta\alpha} \|u_0\|_{C^{\alpha+\beta}_b}.
\end{align*}
For the second term we write %by \eqref{eq:lemma21}
\begin{align*}
    &\bigg|\mL \int_0^{t+\Delta t} P_{t+\Delta t -s} f(s,\cdot)(x)\, ds - \mL \int_0^t P_{t-s} f(s,\cdot)(x)\, ds\bigg|\\ 
    &\leq |(P_{\Delta t} - I) \mL \int_0^{t} P_{t-s} f(s,\cdot)(x)\, ds| + \int_0^{\Delta t} |\mL P_s f(t+\Delta t-s)(x)|\, ds =: I_A + I_B.
\end{align*}
 The fact that $(P_{\Delta t} - I)$ commutes with the integral and $\mL$ is the result of the uniform boundedness of the integrand and the $\alpha+\beta$ H\"older regularity of the time integral in $I_A$. The application of Fubini's theorem in the second term is possible because of the absolute convergence of the double integral given by $\mL \int_0^{\Delta t}$. The absolute convergence follows from %the proof of 
 %\eqref{eq:lemma21}
 Lemma \ref{lem:order} (i) and (iv)  with $r=s^{\frac{1}{\alpha}}$ and the following computations which use Theorem \ref{lem::lem1}: 
%is justified by the finiteness of $I_B$, which follows from \eqref{eq:lemma21}\color{purple}, with  $r=s^{\frac{1}{\alpha}}$, and :
\begin{align*}
    I_B&= \int_{0}^{\Delta t} \big| \mathcal LP_s[f(t+\Delta t-s, \cdot)](x) \big|\, ds \\
    &\lesssim \int_0^{\Delta t} \Big(\|D^2P_s[f(t+\Delta t-s, \cdot)] \|_\infty s^{\frac{2-\alpha}{\alpha}}\\
    &\qquad\qquad+\|DP_s[f(t+\Delta t-s, \cdot)] \|_\infty s^{\frac{1-\alpha}{\alpha}}+ \|P_s[f(t+\Delta t-s, \cdot)] \|_\infty \Big)\, ds\\
    &\lesssim \int_0^{\Delta t} \big([f(t+\Delta t-s, \cdot)]_{\beta}(s^{-\frac{2-\beta}{\alpha}}s^{\frac{2-\alpha}{\alpha}}+s^{-\frac{1-\beta}{\alpha}}s^{\frac{1-\alpha}{\alpha}})+\|f(t+\Delta t-s,\cdot) \|_\infty\big)\, ds \\
    &\lesssim \sup_{t\in (0,T]}\|f(t,\cdot) \|_{C_b^{\beta}}\int_0^{\Delta t}\big(s^{\frac{\beta}{\alpha}-1}+1\big)\, ds \lesssim (\Delta t)^{\frac{\beta}{\alpha}}\sup_{t\in (0,T]}\|f(t,\cdot) \|_{C_b^{\beta}}.
\end{align*}
For $I_A$ we use Lemma \ref{lem:t=0}, Lemma \ref{lem:order} (iii), %\eqref{eq:LphiCbeta}, 
and Theorem~\ref{prop:f_bound} (ii):
% and Theorem~\ref{prop:f_bound}:
\begin{align*}
    I_A \lesssim (\Delta t)^{\frac \beta\alpha} \|\mL \int_0^tP_{t-s} f(s,\cdot)\, ds\|_{C^{\beta}_b} \lesssim (\Delta t)^{\frac \beta\alpha} \sup\limits_{t\in (0,T]}\|f(t,\cdot)\|_{C^{\beta}_b}.
\end{align*}
We conclude the proof by summing up the estimates and get that
\begin{align*}
    \sup\limits_{x\in \R^d} \|\mL u(\cdot,x)\|_{C^{\frac \beta\alpha}_b} &\leq C(d,\alpha,\beta,T,\|Du\|_\infty,H)  \sup\limits_{t\in (0,T]}\|f(t,\cdot)\|_{C^\beta_b} \\
    &\leq C'(d,\alpha,\beta,T,\|Du\|_\infty,H) \big(\sup\limits_{t\in (0,T]} \|u(t,\cdot)\|_{C_b^{1+\beta}} +1\big). \qedhere
\end{align*}
%    &\leq \\C(d,\alpha,\beta,T,\|Du\|_\infty,H) \sup\limits_{t\in (0,T]} (\|u(t,\cdot)\|_{C_b^{\alpha+\beta}} + 1).
% \end{align*}
%\noindent The proof is complete.
\end{proof}

Combining spatial regularity from Theorem \ref{thm::full_schauder_regularity} with temporal regularity from Theorem \ref{thm:timeSchauder}, we can show that the solutions of \eqref{eq:hjb} belong to the space-time H\"older space $C_b^{\frac{\alpha+\beta}{\alpha}, \alpha+\beta}([0,T]\times \R^d)$ and  hence get a full space-time Schauder regularity result. Recall that $C_b^{\frac{\alpha+\beta}{\alpha}, \alpha+\beta}([0,T]\times \R^d)$ is  the subspace of continuous functions such that the following norm is finite:
%A consequence of the above result is that the full spacetime norm, defined by
\begin{equation}\label{eq:spacetimenorm}\begin{split}
    \| u \|_{C_b^{\frac{\alpha+\beta}{\alpha}, \alpha+\beta}([0,T]\times \R^d)}&:= \sup_{t\in[0,T]}\|u(t,\cdot) \|_{C^{\alpha+\beta}_b(\R^d)}+\sup_{x\in \R^d}\|u(\cdot,x) \|_{C^{\frac{\alpha+\beta}{\alpha}}_b([0,T])} \\
    &\qquad +\sup_{t\in[0,T]}[\partial_tu(t,\cdot)]_{\beta}+\sum_{k=1}^{\lfloor \alpha +\beta\rfloor}\sup_{x\in \R^d}[D^ku(\cdot,x)]_{\frac{\alpha+\beta-k}{\alpha}}.
    \end{split}
\end{equation}
\begin{thm}\label{thm:spacetimeSchauder}
    Assume the assumptions of Theorem \ref{thm:timeSchauder}. Then
    \begin{align*}
        \| u \|_{C_b^{\frac{\alpha+\beta}{\alpha}, \alpha+\beta}([0,T]\times \R^d)} \leq C(d,\alpha,\beta,T,\| Du\|_\infty, H) (\sup_{t\in[0,T]}\| u(t,\cdot) \|_{C^{\alpha+\beta}_b}+1).
    \end{align*}
\end{thm}
\begin{proof}
    We only give the proof for $\alpha+\beta\in(2,3)$; the case $\alpha+\beta \in (1,2)$ is simpler. By Theorems \ref{thm::full_schauder_regularity} and \ref{thm:timeSchauder}, the first two terms in the norm are finite. % so the mixed terms remain. 
    As before we let $f(t,x):=H(t,x,Du(t,x))$, and recall that $\partial_t u=\mathcal L u+f$.
%     \medskip
         Then by Lemma \ref{lem:order} (iii), $\sup_t [\partial_t u(t,\cdot)]_\beta$ is also finite since
\begin{equation}\label{eq:full_schauder_time_mixed}
    [\partial_t u(t,\cdot)]_\beta \leq [\mathcal Lu(t,\cdot)]_\beta+[f(t,\cdot)]_\beta  \lesssim \sup_{t\in [0,T]} \|u(t,\cdot) \|_{C^{\alpha+\beta}_b}+\sup_{t\in [0,T]} \|f(t,\cdot) \|_{C^{\beta}_b},
        \end{equation}
which is finite by the spatial regularity of $u$ and $f$.

 Next we consider $\sup_x[Du(\cdot,x)]_{ \frac{\alpha+\beta-1}{\alpha}}$.
 %1+\frac{\beta}{\alpha}-\frac{1}{\alpha}}$. 
 Let $\partial_i$ denote the derivative and $\delta^h_i$ the central difference in the $i$-th direction, that is $\delta^h_ig(x)=(g(x+he_i)-g(x-he_i))/(2h)$. Then, by the triangle inequality,
    \begin{align*}
        |\partial_i u(t+\Delta t,x)-\partial_iu(t,x)| &\leq |\partial_i u(t+\Delta t,x)-\delta^h_i u(t+\Delta t,x)| \\
        &\quad +|\delta^h_i u(t+\Delta t,x)-\delta^h_i u(t,x)|+|\delta^h_i u(t,x)-\partial_i u(t,x)|.
    \end{align*}
    Consider the last term (the first is similar). By the fundamental theorem of calculus,
    \begin{align*}
        2(\delta^h_i u(t,x)-\partial_i u(t,x)) &= \int_{-1}^1 (\partial_iu(t,x+\lambda h e_i)-\partial_iu(t,x))\, d\lambda =\int_{-1}^1\int_0^1\partial_{ii}^2u(t,x+\theta \lambda h e_i)\lambda h \, d\theta \, d\lambda \\
        %&=\partial_{ii}^2u(t,x) h\int_0^1 \int_{-1}^1 \lambda   \, d\lambda\, d\theta +\int_{-1}^1\int_0^1(\partial_{ii}^2u(t,x+\theta \lambda h e_i)-\partial_{ii}^2u(t,x) )\lambda h \, d\theta \, d\lambda \\
        &=\int_{-1}^1\int_0^1(\partial_{ii}^2u(t,x+\theta \lambda h e_i)-\partial_{ii}^2u(t,x) )\lambda h \, d\theta \, d\lambda,
    \end{align*}
    where the last equality follows from adding and subtracting $\partial_{ii}^2u(t,x)$. Consequently, 
    \begin{align*}
        |\partial_i u(t+\Delta t,x)-\delta^h_i u(t+\Delta t,x)|+|\delta^h_i u(t,x)-\partial_i u(t,x)| \leq 2\sup_t[D^2u(t,\cdot)]_{\alpha+\beta-2}h^{\alpha+\beta-1}.
    \end{align*}
    For the middle term, the mean value theorem with some $t'\in (t,t+\Delta t)$ yields
    %and the equation yield
    \begin{align*}
        &|\delta^h_i u(t+\Delta t,x)-\delta^h_i u(t,x)|\\
        &=\frac{|(u(\cdot,x+he_i)-u(\cdot,x-he_i) )(t+\Delta t)-(u(\cdot,x+he_i)-u(\cdot,x-he_i) )(t)|}{2h} \\
        &=\frac{\Delta t}{2h}|\partial_t u(t',x+he_i)-\partial_t u(t',x-he_i)|  \leq \frac{\Delta t}{h^{1-\beta}}(\sup_t \|u(t,\cdot) \|_{C^{\alpha+\beta}_b}+\sup_t \| f(t,\cdot)\|_{C^\beta_b}),
    \end{align*}
    where we used \eqref{eq:full_schauder_time_mixed} in the last inequality. Choosing $h=\Delta t^{\frac{1}{\alpha}}$ yields $h^{\alpha+\beta-1} = \frac{\Delta t}{h^{1-\beta}} = \Delta t ^\frac{\alpha+\beta-1}{\alpha}$ and
    \begin{align*}
        |D u(t+\Delta t,x)-D u(t,x)|\lesssim %h^{\alpha+\beta-1}
        %+\frac{\Delta t}{h^{1-\beta}}\simeq
        %\Delta t^{1+\frac{\beta}{\alpha}-\frac{1}{\alpha}}
        \Delta t^{\frac{\alpha+\beta -1}{\alpha}} \big(\sup_t \|u(t,\cdot) \|_{C^{\alpha+\beta}_b}+\sup_t \| f(t,\cdot)\|_{C^\beta_b}\big).
    \end{align*}
    
    Finally we consider $\sup_x[D^2u(\cdot,x)]_{ \frac{\alpha+\beta-2}{\alpha} }$: 
    %1+\frac{\beta}{\alpha}-\frac{2}{\alpha}}$, estimating in a similar way as before:
    %for the $t$-regularity of $D^2u$ we split in a similar manner:
        \begin{align*}
        |\partial^2_{ij} u(t+\Delta t,x)-\partial^2_{ij}u(t,x)| &\leq |\partial_i \partial_ju(t+\Delta t,x)-\delta^h_i \partial_j u(t+\Delta t,x)| \\
        &\quad +|\delta^h_i \partial_j u(t+\Delta t,x)-\delta^h_i \partial_j u(t,x)|+|\delta^h_i \partial_j u(t,x)-\partial_i \partial_ju(t,x)|.
    \end{align*}
    For the first and the last term we have
    \begin{align*}
        2|\delta^h_i \partial_ju(t,x)-\partial^{2}_{ij}  u (t,x)|&= \left| \int_{-1}^1 (\partial^2_{ij} u(t,x+\lambda h e_i)-\partial^2_{ij}u(t,x))\,d\lambda \right| 
        \leq 2\sup_t [\partial^2_{ij}u(t,\cdot)]_{\alpha+\beta-2}h^{\alpha+\beta-2}.% \leq 2\sup_t \| u(t,\cdot) \|_{C^{\alpha+\beta}_b} h^{\alpha+\beta-2}
    \end{align*}
     For the middle term we 
    use the $[Du(\cdot,x)]_{\frac{\alpha+\beta-1}{\alpha}}$ estimate obtained above:
    \begin{align*}
        &2h|\delta^h_i \partial_j u(t+\Delta t,x)-\delta_i^h\partial_j u(t,x)|\\
        &\leq |\partial_j u(t+\Delta t,x+he_i) - \partial_j u(t,x+he_i)| + |\partial_j u(t+\Delta t,x-he_i) - \partial_j u(t,x-he_i)|\\
        &\leq \Delta t^{\frac{\alpha+\beta -1}{\alpha}} 
        %\Delta t^{1+\frac{\beta}{\alpha}-\frac{1}{\alpha}}
        \big(\sup_t \|u(t,\cdot) \|_{C^{\alpha+\beta}_b}+\sup_t \| f(t,\cdot)\|_{C^\beta_b}\big).
    \end{align*}
    Hence by combining the estimates and taking $h = \Delta t^{\frac 1\alpha}$ we find that 
    \begin{align*}
        \qquad |D^2 u(t+\Delta t,x) - D^2 u(t,x)| \lesssim 
        %\Big(\frac{\Delta t^{\frac{\alpha+\beta-1}{\alpha}}}{h} + h^{\alpha+\beta-2}\Big)
         \Delta t^{\frac{\alpha+\beta-2}{\alpha}} \big(\sup_t \|u(t,\cdot) \|_{C^{\alpha+\beta}_b}+\sup_t \| f(t,\cdot)\|_{C^\beta_b}\big). \qquad\qedhere
    \end{align*}
   % By again taking $h = \Delta t^{\frac 1\alpha}$ we obtain the desired estimate.
\end{proof}

\section{Extensions and remarks}\label{sec::remarks_and_extensions}
In this section we discuss some  extensions of our results and sketch their proofs.

\subsection{Critical diffusions ($\alpha=1$) with small data}
With relatively little modification to our methods, it is possible to obtain Schauder estimates for diffusions satisfying $\|Dp_t\|_{L^1(\R^d)} \leq c_0 t^{-1}$ (i.e. \ref{NDa'} with $\alpha=1$),  under certain smallness constraints on the data. It is unclear to us whether one can obtain a result with no such constraints without a major change in the proof.
\begin{thm}
    Assume that 
    \begin{align*}\|Dp_t\|_{L^1} \leq c_0 t^{-1},\quad t\in (0,T],\end{align*}
    and there exist constants $H_0>0$, $\beta\in (0,1)$, and $L_R^p$, $L_R^{pp}$, $L_R^x$, $L_R^{xp}$ for $R\in(0,\infty)$, such that
    \begin{align*}
        \begin{cases}
            &|H(t,x,0)| \leq H_0,\\
            &|D_pH(t,x,p)| \leq L_R^p,\\
            &|D_{pp}H(t,x,p)| \leq L_R^{pp},\\
            &|H(t,x,p) - H(t,x',p)| \leq L_R^x|x-x'|^\beta,\\
            &|D_pH(t,x,p) - D_pH(t,x',p)| \leq L_R^{xp}|x-x'|^{\beta},
        \end{cases}\qquad t>0,\ x,x'\in \R^d,\ p,q\in B(0,R).
    \end{align*}
    If $\|u_0\|_{C^{1+\beta}(\R^d)}$ and $H_0$ are small enough, and $L_R^p,L_R^x$, and $L_R^{xp}$, 
    are small enough for small $R$, then the problem \eqref{eq:hjb} has a unique mild solution $u$ such that $u(t,\cdot)\in C^{1+\beta}_b(\R^d)$ for every $t\in [0,T]$.
\end{thm}
For example, if $g(p)=|p|^r$ for $r\geq 2$ or $g(p)=(1+|p|^2)^{r/2}$
for $r>0$, then we can take $H(t,x,p) = b(t,x)g(p) %|p|^{r}
+ f(t,x)$ 
%with  $g(p)=|p|^r$ for $r\geq 2$ \cb or $g(p)=(1+|p|^2)^{r/2}$
%$H(t,x,p) = b(t,x)(1+|p|^2)^{r/2}+ f(t,x)$ 
%for $r>0$,  %$$r>1$
%   $H(t,x,p) = b(t,x)g(p)%|p|^{r}
% + f(t,x)$ with $g\in W^{2,\infty}_{\textup{loc}}(\R)$ 
% % $r\geq 2$  %$$r>1$
%and
provided $b,f\in C([0,T],C^\beta_b(\R^d))$ and $\sup_{t\in[0,T]}\|f(t,\cdot)\|_{C^\beta_b}$ is sufficiently small.
\begin{proof}[Sketch of the proof] In contrast to the previous proofs, here we obtain the Schauder estimates together with the existence result. Let
\begin{align*}
    X = \{u\in B([0,T],C^{1+\beta}_b(\R^d)): \sup\limits_{t\in [0,T]} \|u(t,\cdot)\|_{C^{1+\beta}_b} \leq R_1\},
\end{align*}
with the metric induced by $\sup\limits_{t\in [0,T]} \|\cdot\|_{C^{1+\beta}_b}$ and $R_1$ to be determined, and as usual,
\begin{align*}
    S[u](t,x) = P_t u_0(x) + \int_0^t P_{t-s}H(s,\cdot,Du(s,\cdot))(x)\, ds.
\end{align*}
By rather standard computations using Theorem~\ref{prop:f_bound} we get
\begin{align*}
    \|S[u](t,\cdot)\|_{C^{1+\beta}_b} &\leq \|u_0\|_{C^{1+\beta}_b} + (C(T^\beta + 1) + T)\sup\limits_{s\in [0,T]}\|H(s,\cdot,Du(s,\cdot))\|_{C^\beta_b}\\
    &\leq \|u_0\|_{C^{1+\beta}_b} + (C(T^\beta + 1) + T)(H_0 + L_{R_1}^x + L_{R_1}^p R_1).
\end{align*}
The last expression is smaller than $R_1$ under postulated smallness conditions, in which case $S$ maps $X$ to $X$. For the contractivity, again using Theorem~\ref{prop:f_bound} we find that
\begin{align*}
    \|S[u](t,\cdot) - S[v](t,\cdot)\|_{C^{1+\beta}_b} \leq (T + C(T^{\beta} + 1)) \sup\limits_{s\in [0,T]} \|H(s,\cdot,Du(s,\cdot)) - H(s,\cdot,Dv(s,\cdot))\|_{C^{\beta}_b}.
\end{align*}
Furthermore, 
\begin{align*}
\|H(s,\cdot,Du(s,\cdot)) - H(s,\cdot,Dv(s,\cdot))\|_{\infty} \leq L_{R_1}^p \|Du-Dv\|_\infty,
\end{align*}
and by using the fundamental theorem of calculus and splitting we find that
\begin{align*}
    [H(s,\cdot,Du(s,\cdot)) - H(s,\cdot,Dv(s,\cdot))]_\beta \leq L_{R_1}^{xp} \|Du - Dv\|_\infty + L_{R_1}^{pp}R_1 \|Du - Dv\|_\infty + L_{R_1}^p [Du - Dv]_{\beta}.
\end{align*}
Therefore, 
\begin{align*}
    \|S[u]- S[v]\|_{X} \leq (T + C(T^{\beta} + 1)) (L_{R_1}^p + L_{R_1}^{xp} + L_{R_1}^{pp}R_1)\|u-v\|_{X},
\end{align*}
which is smaller than 1 if we take $R_1$ small enough. Thus, under the assumptions of the theorem $S$ is a contractive mapping on $X$, so by Banach's fixed point theorem it has a unique fixed point. The Schauder estimates follow from the definition of $X$.
\end{proof}

\subsection{Nonlinearities with respect to other lower order operators} 
With almost no change in the proofs we can get optimal Schauder estimates for Hamiltonians depending also on other linear quantities $Qu$ of order not exceeding 1. More precisely, we let $H = H(t,x,p,q)$ be continuous and assume:\smallskip
\begin{align}
    &|H(t,x',p',q') - H(t,x,p,q)| \leq C_R\big(|x-x'|^\beta\wedge1 + |p-p'| + |q-q'|\big)\quad \textrm{and}\quad H(t,x,0,0)\leq C\label{eq:Hcond}\\
    &\qquad\qquad\textrm{for all}\quad t\in [0,T],\ x,x'\in \R^d,\ p,p',q,q'\in B(0,R), \nonumber\\[0.2cm] 
 &Q \text{ is linear,} \ \ \|Q\phi\|_\infty \leq C\|\phi \|_{C^1_b(\R^d)}, \ \  \textrm{and}\ \  [Q \phi]_\sigma \leq C \|\phi\|_{C^{1+\sigma}_b}\ \ \text{for all} \ \sigma\in (0,\beta].\label{eq:Qcond}
\end{align}\smallskip
We then have the following generalization of Theorems \ref{thm::hjb_sol_existence} and \ref{thm::full_schauder_regularity} on short-time existence and Schauder regularity for the case of solutions with bounded gradients uniformly in time.
\begin{thm}\label{thm:Qdependent}
    Assume \ref{assump:H1}, \eqref{eq:Hcond}, \eqref{eq:Qcond} for some $\beta\in [0,1]$, and \ref{NDa} with $\alpha\in (1,2]$ such that $\alpha+\beta\notin \mathbb{N}$.
    %and $m=\lceil 
     %\alpha+\beta\rceil$. 
    %    , and suppose that $u_0\in C^1_b(\R^d)$. 
    Then, for sufficiently small $T$, the equation
\begin{align}\label{eq:Qproblem}
\left\{
    \begin{aligned}
         \partial_t u -\mathcal{L}u-H(t,x,Du,Qu)&=0, \qquad &(t,x)&\in (0,T]\times \R^d, \\ 
        u(0,x)&=u_0(x), \qquad &x &\in \R^d,
    \end{aligned}
\right.
\end{align}
has a unique  mild solution $u\in C_b((0,T],C^1_b(\R^d))$ and for all $t\in (0,T]$ we have $u(t)\in C^{\alpha+\beta}_b(\R^d)$.
\end{thm}
\begin{proof}[Sketch of the proof]
    The short-time existence is done in almost the same way as in Theorem~\ref{thm::hjb_sol_existence}, in the set $X_A$ defined in \eqref{eq:XAdef}. The only difference is that now to control the Hamiltonian term, we use \eqref{eq:Hcond} and the fact that $\|Qu\|_\infty \leq \|u\|_{C^1_b}$.

    For the $C^{\alpha+\beta}_b(\R^d)$ estimate we also use a bootstrap argument -- first establishing $C^{1+\beta}_b(\R^d)$ regularity and then improving it as was done in Theorem~\ref{thm::full_schauder_regularity}. The difference now is that we use different arguments to get $C^{1+\beta}_b(\R^d)$ regularity of $u(t)$ (Gr\"onwall's inequality seems to impose additional restrictions on $Q$):
    %\footnote{We want to avoid using Gr\"onwall's inequality, since it seems to impose additional restrictions on $Q$.}
    If $\alpha > 1+\beta$ it follows immediately from (ii) in Lemma~\ref{lem::duhamel_map_boundedness}, while for $\alpha\leq 1+\beta$ we only get  $u(t)\in C^{\alpha-\epsilon}_b(\R^d)$ from Lemma~~\ref{lem::duhamel_map_boundedness}, but then $H(t,x,Du,Qu)$ is $\alpha-\epsilon-1$ H\"older regular by \eqref{eq:Hcond} and \eqref{eq:Qcond} and we can get $1+\beta$ regularity by bootstrapping with the use of Theorem~\ref{prop:f_bound}. Once we get that $u(t)\in C^{1+\beta}_b(\R^d)$, the remainder of the proof is identical to that of Theorem~\ref{thm::full_schauder_regularity}.
\end{proof}
The space-time regularity and classical solvability results can also be extended to this context. 
\begin{thm}
%    Assume \eqref{eq:Hcond} and \eqref{eq:Qcond}. % Then we have the following:
Assume \ref{assump:H1}, \eqref{eq:Hcond}, \eqref{eq:Qcond} for some $\beta\in (0,1]$, \ref{NDa} with $\alpha\in (1,2]$ such that $\alpha+\beta\notin \mathbb{N}$, 
and $u\in C_b((0,T],C^1_b(\R^d))$ is a mild solution of \eqref{eq:Qproblem}. 

    \medskip
    \noindent (a) \ 
If $\alpha+\beta > 2$ or \ref{assump:L2} holds, then $u$ is a classical solution of \eqref{eq:Qproblem}.
   % 
    % Under the additional assumptions of Theorem \ref{thm:u_is_classical_sol}, the mild solution $u$ of \eqref{eq:Qproblem} is classical. 
    %conclusion of Theorem \ref{thm:u_is_classical_sol} holds.

    \medskip 
    \noindent (b) \ If in addition $\beta\in(0,1)$, $\alpha+\beta\in (1,2)\cup(2,3)$, 
    $u_0\in C^{\alpha+\beta}_b(\R^d)$, \ref{assump:L2'} holds, and
%   Under the additional assumptions of Theorem \ref{thm:timeSchauder} \color{orange} and assuming

  \begin{align*}
    &|H(t,x,p,q)-H(t',x,p,q)| \leq K_R|t-t'|^{\frac{\beta}{\alpha}}, \quad t,t'\in[0,T] \text{ with } |t-t'|<1,\ x\in \R^d, \ p,q\in B(0,R),
\end{align*} 
    then the space-time H\"older estimate of Theorem %\ref{thm:timeSchauder} and 
    \ref{thm:spacetimeSchauder} holds:
        \begin{align*}
        \| u \|_{C_b^{\frac{\alpha+\beta}{\alpha}, \alpha+\beta}([0,T]\times \R^d)} \leq C(d,\alpha,\beta,T,\| Du\|_\infty, H) (\sup_{t\in[0,T]}\| u(t,\cdot) \|_{C^{\alpha+\beta}_b}+1).
    \end{align*}
\end{thm}
\begin{proof}
    Let $  f(t,x):=H(t,x,Du,Qu)$, and note that $f(t,\cdot)\in C^\beta_b$ by Theorem \ref{thm:Qdependent}, \eqref{eq:Hcond}, and \eqref{eq:Qcond}.
    %. In both cases we have that Theorem \ref{thm:Qdependent} holds, so importantly $f(t,\cdot)\in C^\beta_b$ by \eqref{eq:Hcond} and \eqref{eq:Qcond}. 
    % Inspecting the proofs of Theorems \ref{thm:u_is_classical_sol} and \ref{thm:timeSchauder}, we see that
    Then (a) and (b) follow by inspecting the proofs of Theorem~\ref{thm:u_is_classical_sol} and Theorems~\ref{thm:timeSchauder} and \ref{thm:spacetimeSchauder} respectively.
    The only difference is that we now have the extra term
    $$C_{R_1}|Qu(t+\Delta t, x)-Qu(t,x)| \qquad \text{where} \qquad R_1:=\max(C,1)\sup_{t\in [0,T]}\|u(t,\cdot)\|_{C^1_b}$$ in \eqref{eq:f_locally_uniformly_cont} and \eqref{eq:temporal_reg_f} respectively, with $C$ from \eqref{eq:Qcond}. As $Q$ is linear, by \eqref{eq:Qcond} we have that
    \begin{align*}
        |Qu(t+\Delta t, x)-Qu(t,x)|\leq C
        %\|u(t,\cdot)\|_{C_b^1} 
        \cdot\begin{cases}
            %\widehat 
            \omega_u(\Delta t), \qquad &\text{in case } (a),\\[0.2cm]
            \Delta t^{\frac{\beta}{\alpha}}%(\|u_0\|_{C^{1+\beta}_b}+\sup\limits_{t\in (0,T]} \|f(t,\cdot)\|_{C^\beta_b})
            \big(\|u(\cdot,x)\|_{C^{\frac{\beta}{\alpha}}_b} + \|Du(\cdot,x)\|_{C^{\frac{\beta}{\alpha}}_b}\big), \qquad &\text{in case } (b),
        \end{cases}
    \end{align*}
    where $\omega_u$ is the modulus of continuity of $u$ and $Du$, 
     and the cases follow from the discussion preceding \eqref{eq:f_locally_uniformly_cont} and the computations following \eqref{eq:temporal_reg_f} respectively. Thus, the rest of the proofs go through without modification.
\end{proof}
\begin{example}\ \smallskip
\begin{itemize}
%\noindent (a)  
\item[(a)] (0 order terms) $Qu = u$ satisfies assumption \eqref{eq:Qcond}.
 
\medskip

%\noindent (b) 
\item[(b)] (Lower order nonlocal operators)  Let $Q u(x) = \int_{\mathbb{R}^d} (u(x+j(x,z)) - u(x))\, d\nu(z)$ for $x\in \R^d$, where 
$\nu\geq0$, % is a L\'evy measure, 
$\int_{\R^d} (1 \wedge |z|)\, d\nu(z)< \infty$, and
$$|j(x,z)|+\frac{|j(x,z)-j(y,z)|}{|x-y|^\beta}\leq K(1\wedge |z|)
%\qquad\text{and}\qquad|j(x,z)|\leq K(1\wedge |z|)
\quad \text{for}\quad x,y,z\in \R^d.$$ 
A straightforward calculation shows that \eqref{eq:Qcond} holds:
\begin{align*}
    \|Q u\|_\infty &\leq C \|u\|_{C^1_b}\qquad\text{and}\qquad [Qu]_\sigma\leq C \|u\|_{C^{1+\sigma}_b} ,\quad \sigma \in (0,1].
\end{align*}
% Let $\nu$ be a L\'evy measure satisfying $\int_{\R^d} (1 \wedge |z|)\, d\nu(z)< \infty$ and define 
% \begin{align*}
%     Q u(x) = \int_{\mathbb{R}^d} (u(x+j(x,z)) - u(x))\, d\nu(z), \quad x\in \R^d,
% \end{align*}
% where for $x,y,z\in \R^d$,
% $$|j(x,z)-j(y,z)|\leq L(1\wedge |z|)|x-y|^\beta\qquad\text{and}\qquad|j(x,z)|\leq K(1\wedge |z|).$$ 
% Then a straightforward calculation shows that \eqref{eq:Qcond} holds:
% \begin{align*}
%     \|Q u\|_\infty &\leq C \|u\|_{C^1_b}\qquad\text{and}\qquad |Qu(x) - Qu(y)| \leq C \|u\|_{C^{1+\sigma}_b} |x-y|^\sigma,\quad \sigma \in (0,1],\ x,y\in \R^d.
% \end{align*}

%\noindent (c) %{\em Operators $\mL$ with modulated jumps.} 
\item[(c)] (Operators with modulated jumps) 
%Theorem \ref{thm:Qdependent} 
% covers problems \eqref{eq:hjb} where 
 Replace $\mathcal{L}$ 
%in \eqref{eq:hjb} can be 
%is replaced 
by an operator with modulated jumps,
\begin{align*}
    \mL_{\bfj} u(x) = \lim\limits_{\epsilon\to 0}\int_{B(0,\epsilon)^c} (u(x) - u(x+\bfj(z)))\, d \mu(z),\quad x\in \R^d,
\end{align*}
%provided that 
where $\mu$ is a L\'evy measure and $\bfj\colon \R^d \to \R^d$ behaves well at the origin, e.g. \begin{align}\label{eq:jcond}|\bfj (z) - z| \leq C|z|^2,\quad |z|\leq 1.\end{align} 
Then Theorem \ref{thm:Qdependent} still holds because we can 
%We simply 
% write $\mL_{\bfj} = \mL + (\mL_{\bfj} - \mL)$ and view this as
% $\mL$ plus a linear perturbation $Q$,
write $\mL_{\bfj} = \mL + Q$ where the linear perturbation
\begin{align*}
    Qu(x):=(\mL_{\bfj} - \mL)u(x) = \lim\limits_{\epsilon\to 0}\int_{B(0,\epsilon)^c} (u(x+z) - u(x+\bfj(z)))\, d \mu(z),\quad x\in \R^d.
\end{align*}
Assuming \eqref{eq:jcond}, it is easy to see that \eqref{eq:Qcond} holds, 
% we get the following estimates:
% %similar to the above:
% \begin{align*}
%     \|(\mL_{\bfj} -\mL) u\|_\infty &\leq C \|u\|_{C^1_b},\\ |(\mL_{\bfj} -\mL)u(x) - (\mL_{\bfj}-\mL) u(y)| &\leq C \|u\|_{C^{1+\sigma}_b} |x-y|^\sigma,\quad \sigma \in (0,1],\ x,y\in \R^d.
% \end{align*}
and by redefining $H(t,x,Du) + Qu$ as $H(t,x,Du,Qu)$, we see that also \eqref{eq:Hcond} holds.
\end{itemize}
\end{example}
%\smallskip

\begin{rem}
%Controlled jump-diffusion of order less than one. 
It is now quite standard to extend Theorem \ref{thm:Qdependent} to Bellman/dynamic programming equations for optimal control problems for jump-diffusions of order less than one \cite{MR3931325,Ha:Book},
\begin{align*}
\left\{
    \begin{aligned}
         \partial_t u -\mathcal{L}u-\sup_{\theta}\Big\{Q^\theta u + b^\theta(t,x)Du + c^\theta (t,x)u + f^\theta(t,x)\Big\}&=0, \qquad &(t,x)&\in (0,T]\times \R^d, \\ 
        u(0,x)&=u_0(x), \qquad &x &\in \R^d,
    \end{aligned}
\right.
\end{align*}
where $Q^\theta$ is defined as in (b) with $j^\theta$ in place of $j$, and we assume the following uniformly in $\theta$:
%the bounds on $j^\theta$ are uniform in $\theta$,
$$|j^\theta(x,z)|+\frac{|j^\theta(x,z)-j^\theta(y,z)|}{|x-y|} \leq K(1\wedge |z|),$$ 
combined with (standard) uniform in $\theta$ boundedness and Lipschitz conditions on $(b^\theta,c^\theta,f^\theta)$.

\end{rem}

\subsection{More regularity in certain directions}\label{subseq:directions}

Our optimal Schauder regularity implies that under assumptions \ref{assump:H} and \ref{NDa} the solutions are $C^{\alpha+\beta}_b(\R^d)$, in the sense that the total derivative of order $\lfloor \alpha+\beta\rfloor$ is $\{\alpha +\beta\}$-H\"older regular. However, for some specific operators $\mL$ we can expect that some directional derivatives will have more regularity. This is the case e.g. when $\mL = (-\Delta)^{\alpha_1/2}_{x_1} + (-\Delta)^{\alpha_2/2}_{x_2}$ with $1<\alpha_2\leq\alpha_1\leq 2$ and $x_1\in \R^{d_1}$, $x_2 \in \R^{d_2}$. The reason for this is that the heat kernel $p_t$ of such an operator is a convolution of the heat kernels of the two fractional Laplacians: $p_t^1$ and $p_t^2$. Then, given a function $f\in C^{\beta}_b(\R^{d_1+d_2})$, we have $p_t \ast f = p_t^1\ast (p_t^2\ast f)$ and since $\|p_t^2\ast f(\cdot,x_2)\|_{C^\beta_b(\R^{d_1})}$ are uniformly bounded for $x_2\in \R^{d_2}$, the regularity in $x_1$ only depends on $p_t^1$. The following more general result holds true.

\begin{thm}
    Let $d=d_1+d_2$ for some $0\leq d_1,d_2\leq d$ and for $x\in \R^d$ denote $x=(x_1,x_2)$ where $x_i\in\R^{d_i}$, $i=1,2$. Assume \ref{assump:H2}, \ref{assump:H1}, and \ref{assump:H} with $\beta\in[0,1]$, and that $\mL^1$ and $\mL^2$ are L\'evy operators on $\R^{d_1}$ and $\R^{d_2}$ respectively, which satisfy \ref{NDa} with $
    \alpha_1$ and $\alpha_2$ respectively, with $1<\alpha_2\leq\alpha_1\leq 2$.
    %, $m=\lceil \alpha_1+\beta\rceil$. 
    Define a L\'evy operator on $\R^d$ by $\mL = \mL^1_{x_1} + \mL^2_{x_2}$.\medskip
    
    \noindent (i)\ For sufficiently small $T$, there exists a mild solution $u\in C_b((0,T];C^1_b(\R^d))$ of \eqref{eq:hjb}.\medskip
    
    \noindent (ii)\ If $\alpha_2+\beta \notin \N$, then $u(t)\in C^{\alpha_2+\beta}_b(\R^d)$ for each $t\in (0,T]$.\medskip
    
    \noindent (iii)\ If also $\alpha_1+\beta \notin \N$, then for each $t\in (0,T]$ we have $u(t,\cdot,x_2)\in C^{\alpha_1+\beta}_b(\R^{d_1})$ uniformly for $x_2\in \R^{d_2}$, that is, $\sup_{x_2}[D_{x_1}^{ \lfloor\alpha_1+\beta\rfloor} u(t,\cdot,x_2)]_{\{\alpha_1+\beta\}} < \infty$ where $\{\alpha+\beta\}=\alpha+\beta-\lfloor \alpha+\beta \rfloor$.
\end{thm}
Part (iii) above states that we get more regularity in the direction of $x_1$. The result could be further generalized -- in some cases we could get more regularity in non-axial directions, for example by adding a one-dimensional diffusion in such direction. We refer to \cite[Theorem~4.2]{MR4309434} for a related discussion.
\section*{Acknowledgements}
A. Rutkowski was supported by the National Science Center (Poland) grant 2023/51/B/ST1/02209. E. R. Jakobsen and R. \O. Lien received funding from the Research Council of Norway under Grant Agreement No. 325114 “IMod. Partial differential equations, statistics and data: An interdisciplinary approach to data-based modelling”.
\appendix

\section{Some technical results}\label{sec::app}
In this section we collect some technical results used throughout the paper. The first is a Hölder-interpolation result. 
\begin{lem}[Hölder interpolation]
\label{thm:holder_interpolation}
    % If $g\in C_b^1(\R^d)$ and $0<\gamma<1$, then
    % %Then the following inequality holds:
    % \begin{align*}
    % [g]_{\gamma} \leq 2^{1-\gamma}(\|g\|_{\infty})^{1-\gamma}(\|Dg\|_{\infty})^{\gamma}\qquad\text{and}\qquad\|Dg\|_{\infty} \leq C_{\gamma, d} [g]_{\gamma}^{\gamma}[Dg]_{\gamma}^{1-\gamma}.
    % \end{align*}
    If $g\in C_b^\eta(\R^d)$ and $0<\gamma<\eta\leq 1$, then
    $$[g]_{\gamma} \leq 2^{1-\frac{\gamma}{\eta}}(\|g\|_{\infty})^{1-\frac{\gamma}{\eta}}([g]_{\eta})^{\frac{\gamma}{\eta}}.$$
    When $\eta=1$ we also have 
    $$\|Dg\|_{\infty} \leq C_{\gamma, d} [g]_{\gamma}^{\gamma}[Dg]_{\gamma}^{1-\gamma}.$$
\end{lem}
\begin{proof}
%     The first inequality follows from a straight-forward calculation:
%     \begin{align*}
%     [g]_{\gamma} %&= \sup_{\substack{x, h\in \R^d \\ h \neq 0}} \frac{|f(x+h)-f(x)|}{|h|^{\gamma}} \\
%     &\leq \sup_{\substack{x, h\in \R^d \\ h \neq 0}} (2\|g\|_{\infty})^{1-\gamma}\bigg( \frac{|g(x+h)-g(x)|}{|h|}\bigg)^{\gamma} \leq 2^{1-\gamma}\|g\|_{\infty}^{1-\gamma}\|Dg \|_{\infty}^{\gamma}.
% \end{align*}
The first inequality follows from a straightforward calculation:
    \begin{align*}
    [g]_{\gamma} %&= \sup_{\substack{x, h\in \R^d \\ h \neq 0}} \frac{|f(x+h)-f(x)|}{|h|^{\gamma}} \\
    &\leq \sup_{\substack{x, h\in \R^d \\ h \neq 0}} (2\|g\|_{\infty})^{1-\frac{\gamma}{\eta}}\bigg( \frac{|g(x+h)-g(x)|}{|h|^\eta}\bigg)^{\frac{\gamma}{\eta}} = 2^{1-\frac{\gamma}{\eta}}\|g\|_{\infty}^{1-\frac{\gamma}{\eta}}[g ]_{\eta}^{\frac{\gamma}{\eta}}. %\leq 2^{1-\gamma}\|g\|_{\infty}^{1-\gamma}\|Dg \|_{\infty}^{\gamma}.
\end{align*}
%and recall that $[g ]_{\eta}=\| Dg\|_\infty$ for $\eta=1$
The second inequality follows from \cite[Exercise 3.3.7]{krylov1996lectures}.
\end{proof}
% \begin{proof}
%     See \cite[Lemma 2.30]{Amund}. 
% \end{proof}
Another result we will need is a generalization of Grönwall's inequality. %We will use it to achieve full Schauder-regularity of the solution to the HJB equation. The result is originally from \cite{MR2290034}, but we present it as in \cite[Lemma 2.11]{Amund}. %MAYBE: We present the result in two different forms; one version as in the original paper \cite{MR2290034}, and another as in \cite[Lemma 2.11]{Amund}.
\begin{lem}[Generalized Grönwall inequality I]\label{lem:generalized_gronwall}
    Assume $a_0, a_{T_0}, c\geq 0$, $\gamma, \zeta <1$, $T_0>0$, and $u(t)$ is a nonnegative and locally integrable function on $[0,T_0)$ satisfying 
    \begin{align*}
        u(t) \leq a_0 t^{-\gamma}+a_{T_0}+c\int_{0}^t (t-s)^{-\zeta}u(s)\, ds, \qquad t\in[0,T_0).
    \end{align*}
    Then there are constants $C_1, C_2 \geq 0$ depending only on $T_0, \gamma, \zeta, c$, in particular they are independent of $t$, such that for $t \in [0,T_0)$ the following bound holds: %there exist constants $b_0, b_{T_0} \geq 0$ %independent of $t$ 
    %such that $u(t) \leq b_0 t^{-\gamma}+b_{T_0}$ for %any 
    $$u(t) \leq (a_0+C_1 T_0^{1-\zeta}) t^{-\gamma}+a_{T_0}+C_2 T_0^{1-\zeta}.$$
\end{lem}
\begin{proof}
    See e.g. \cite[Theorem 1]{MR2290034} or \cite[Lemma 7.1.1]{henry81:GTS}. This version of the result is proved in \cite[Lemma 2.11]{Amund}.
\end{proof}
\begin{lem}[Generalized Grönwall inequality II]\label{lem:generalized_gronwall_2}
     Assume $T_0>0$, $a, b\geq 0$, $\overline\alpha, \overline\beta, \overline\gamma>0$ such that $ \overline\nu := \overline\beta + \overline\gamma - 1 > 0$ and $\overline\delta  :=  \overline\alpha + \overline\gamma - 1 > 0$,   $u(t)$ is nonnegative, $t^{\overline \gamma-1}u(t)$ is locally integrable on $[0,T_0)$, and that
\begin{align*}
u(t) &\leq a t^{\overline\alpha - 1} + b \int_0^t (t - s)^{\overline\beta - 1} s^{\overline\gamma - 1} u(s) \, ds \quad \text{for} \quad t\in (0, T_0].
\end{align*}
Then
\begin{align*}
u(t) &\leq a t^{\overline\alpha - 1} \sum_{m=0}^{\infty} C_m' (b \Gamma(\overline\beta))^m t^{m \overline\nu} \quad \text{for} \quad t\in (0,T_0],
\end{align*}
where \ $C_0' = 1$, \ $C_{m+1}' / C_m' = \dfrac{\Gamma(m \overline\nu + \overline\delta)}{\Gamma(m \overline\nu + \overline\delta + \overline\beta)}$, and the right hand series converges uniformly in $[0,T_0]$. % Furthermore, the sum on the right-hand side is uniformly bounded on $[0,T_0]$. 
\end{lem}
\begin{proof}
    This is \cite[Exercise 3 p. 190 and Lemma 7.1.2]{henry81:GTS}, except for the uniform convergence of the series. Since the series is
 an increasing function of $t$, uniform convergence follows from the Weierstrass $M$-test if the following series converges:
    \begin{align*}
        \sum_{m=0}^\infty M_m:=\sum_{m=0}^{\infty} C_m' (b \Gamma(\overline\beta))^m T_0^{m \overline\nu}.
    \end{align*}
    By Stirling's formula, $\frac{M_{m+1}}{M_m}=b\Gamma(\overline{\beta})T_0^{\overline{\nu}}\frac{\Gamma(m \overline\nu + \overline\delta)}{\Gamma(m \overline\nu + \overline\delta + \overline\beta)}\sim (m\overline{\nu}+\overline{\delta})^{-\overline{\beta}}\to 0$ as $m\to\infty$, so $\sum_m{M_m}$ converges by the the ratio test.
\end{proof}

\bibliographystyle{abbrv} %{abbrv}
\bibliography{references}

%\printbibliography[heading = bibintoc, title = Bibliography]

\end{document}